\documentclass[noinfoline]{imsart}

\RequirePackage[OT1]{fontenc}
\usepackage[utf8]{inputenc}
\RequirePackage{amsthm,amsmath}
\RequirePackage{dsfont}
\RequirePackage[colorlinks,citecolor=blue,urlcolor=blue]{hyperref}
\RequirePackage{hypernat}
\RequirePackage{subfigure}
\RequirePackage{array}
\RequirePackage{chngpage}
\RequirePackage{multirow}
\RequirePackage{cite}

\usepackage{algorithm}
\usepackage{algorithmicx}
\usepackage{algpseudocode}
\usepackage{url}

\usepackage{comment}
\newcommand\AlgPhase[1]{%
\Statex\hspace*{-8pt}\textbf{#1}%
}

\usepackage[english]{babel}

\usepackage{graphicx}
\usepackage{amsfonts}
\usepackage{amssymb}    

\usepackage{color}

\def\var{\mbox{Var}}

\def\X{\mathbf{X}}
\def\Y{\mathbf{Y}}

\def\1{{\mathbf 1}}
\def\N{{\mathcal{N}}}

\def\D{{\mathcal{D}}}
\def\H{{\mathbf{H}}}

\def\S{{\S}}
\def\K{{\mathcal{K}}}
\newcommand{\hyp}{\mathcal{H}}

\newcommand{\pa}[1]{\left({#1}\right)}

\newcommand{\ac}[1]{\left\{{#1}\right\}}

\def\SLasso{{\widehat{\mathcal{S}}_{\text{Lasso}}}}

\newtheorem{thrm}{Theorem}[section]
\newtheorem{proposition}[thrm]{Proposition}
\newtheorem{lemma}[thrm]{Lemma}

\newtheorem{definition}[thrm]{Definition}
\newtheorem{remark}{Remark}[section]

\def\S{{\bf {S}}}
\def\cS{{\mathcal{S}}}

\def\eps{\boldsymbol{\epsilon}}
\usepackage{vmargin}
\setmarginsrb{3.7cm}{2cm}{3.7cm}{3cm}{1cm}{1cm}{1cm}{1cm}
\begin{document}

\begin{frontmatter}
\title{A Global Homogeneity Test for High-Dimensional
  Linear Regression}
\runtitle{Homogeneity tests}
\runauthor{Charbonnier et al.}
\begin{aug}

\author{\fnms{Camille} \snm{Charbonnier}\thanksref{t1}\ead[label=u]{camille.charbonnier@ensae.org}},
\author{\fnms{Nicolas} \snm{Verzelen}\thanksref{t2}\ead[label=v]{nicolas.verzelen@supagro.inra.fr}}
\and
\author{\fnms{Fanny} \snm{Villers}\thanksref{t3}\ead[label=w]{fanny.villers@upmc.fr}}

\thankstext{t1}{\printead{u}}
\thankstext{t2}{\printead{v}}
\thankstext{t3}{\printead{w}}

\address{}

\end{aug}

 \begin{abstract}
 This paper is motivated by the comparison of genetic networks based on microarray samples. The aim is to test whether the differences observed between two inferred Gaussian graphical models come from real differences or arise from estimation uncertainties.
 Adopting a neighborhood approach, we consider a two-sample linear regression model with random design and propose a procedure to test whether these two regressions are the same. Relying on  multiple testing  and variable selection strategies, we develop a testing procedure that applies to high-dimensional settings where the number of covariates $p$ is larger than the number of observations $n_1$ and $n_2$ of the two samples.  Both type I and type II errors are explicitely controlled from a non-asymptotic perspective and  the test is proved to be minimax adaptive to the sparsity.
The performances of the test are  evaluated on simulated
data. Moreover, we illustrate how this procedure can be used to
compare genetic networks on Hess \emph{et al} breast cancer microarray dataset. 
\end{abstract}
 

\end{frontmatter}

\section{Introduction}

The recent flood of high-dimensional data has motivated the development of a vast range
of sparse estimators for linear regressions, in particular a large variety of derivatives
from the Lasso. Although theoretical guarantees have been
provided in terms of prediction, estimation and selection
performances (among a lot of others~\cite{2009_AS_Bickel,2009_IEEE_Wainwright,2008_AS_Meinshausen}), the
research effort has only recently turned to the construction of high-dimensional confidence intervals or parametric
hypothesis testing schemes~\cite{2011_arxiv_Zhang,2012_arXiv_Buhlmann,Montanari,2013arXiv1301.7161L,2013arXiv1309.3489M}. Yet, quantifying the confidence surrounding coefficient estimates and selected covariates is essential in areas of application where these will nourrish further targeted investigations.

In this paper we consider the two-sample linear regression model with
Gaussian random design.
\begin{eqnarray}
 {Y}^{(1)} &= & { X}^{(1)}\beta^{(1)}+ \epsilon^{(1)} \label{Sodel_01}\\
 {Y}^{(2)} &= & { X}^{(2)}\beta^{(2)}+\epsilon^{(2)}\ \label{Sodel_02} .
 \end{eqnarray}
In this statistical model, the size $p$ row vectors $X^{(1)}$ and $X^{(2)}$ follow Gaussian
distributions $\N(0_p,\Sigma^{(1)})$ and $\N(0_p,\Sigma^{(2)})$ whose
covariance matrices remain unknown. The noise components
$\epsilon^{(1)}$ and $\epsilon^{(2)}$ are independent from the design
matrices and follow a centered Gaussian
distribution with unknown standard deviations $\sigma^{(1)}$ and
$\sigma^{(2)}$. In this formal setting, our objective is to develop a
test for the equality of $\beta^{(1)}$ and $\beta^{(2)}$ which remains
valid in high-dimension.

Suppose that we observe an $n_1$-sample of $(Y^{(1)},X^{(1)})$ and an
$n_2$-sample of $(Y^{(2)},X^{(2)})$ noted ${\bf Y}^{(1)}$, ${\bf
  X}^{(1)}$, and ${\bf Y}^{(2)}$, ${\bf X}^{(2)}$, with $n_1$ and $n_2$ remaining smaller than $p$.
Defining analogously $\eps^{(1)}$ and $\eps^{(2)}$, we obtain the decompositions 
$ {\bf Y}^{(1)} =  {\bf X}^{(1)}\beta^{(1)}+ \eps^{(1)}$ and 
${\bf Y}^{(2)} =  {\bf X}^{(2)}\beta^{(2)}+\eps^{(2)}$.
Given these observations, we want to test whether models (\ref{Sodel_01}) and (\ref{Sodel_02}) are the same, that is 
\begin{equation}\label{eq:hypothese}
\begin{cases}
\hyp_0:&  \beta^{(1)}=\beta^{(2)}\ ,\quad \sigma^{(1)}=\sigma^{(2)}\ ,\quad  \text{and}\quad  \Sigma^{(1)}=
\Sigma^{(2)} \\
\hyp_1:&  \beta^{(1)}\neq\beta^{(2)}\ \text{ or }\quad
\sigma^{(1)}\neq\sigma^{(2)}\ .\\
\end{cases}
\end{equation}

In the null hypothesis, we include the assumption that the population covariances of the covariates are equal ($\Sigma^{(1)} = \Sigma^{(2)}$), while under the alternative hypothesis the population covariances are not  required to be the same.  This choice of assumptions is primarely motivated by our final objective to derive homogeneity tests for Gaussian graphical models (see below). A discussion of the design hypotheses is deferred to Section \ref{sec:discussion}.

\subsection{Connection with two-sample Gaussian graphical model testing}

This testing framework is mainly motivated by the validation of differences observed between 
Gaussian graphical models (modelling regulation networks) inferred from transcriptomic data from two
samples \cite{2006_AS_Meinshausen,2008_B_Friedman,2011_SC_Chiquet}
when looking for new potential drug or knock-out targets
\cite{2011_JSFDS_Jeanmougin}. Following the development of
univariate differential analysis techniques, there is now a surging demand for the detection of differential regulations between pairs of conditions (treated vs. placebo, diseased vs. healthy,
exposed vs. control, \dots). Given  two gene regulation networks inferred from two transcriptomic data samples,
it is however difficult to disentangle differences in the estimated networks that are due to estimation errors from real differences in the true underlying networks.

We suggest to build upon our two-sample high-dimensional linear regression testing scheme to derive a global test for the equality of high-dimensional Gaussian graphical models inferred under pairs of conditions.


Formally speaking, the global two-sample GGM testing problem is defined as follows. Consider two Gaussian
random vectors $ {Z}^{(1)} \sim \mathcal{N}(0,[\Omega^{(1)}]^{-1})$ and $
Z^{(2)} \sim \mathcal{N}(0,[\Omega^{(2)}]^{-1})$. The dependency graphs  are characterized by the non-zero entries of the precision matrices $\Omega^{(1)}$ and $\Omega^{(2)}$~\cite{1996_Lauritzen}. Given an $n_1$-sample of $Z^{(1)}$ and an $n_2$-sample of $Z^{(2)}$, the objective is to test 
\begin{equation}\label{eq:hypothese_GGM}
\hyp^{G}_0:\  \Omega^{(1)}=\Omega^{(2)}\quad \text{ versus }\quad 
\hyp^{G}_1:\  \Omega^{(1)}\neq\Omega^{(2)}\ ,
\end{equation}
where $\Omega^{(1)}$ and $\Omega^{(2)}$ are assumed to be sparse (most of their entries are zero). This testing problem is global as the objective is to assess a statistically significant difference between the two distributions. If the test is rejected, a more ambitious objective is to infer the entries where the precision matrices differ (ie $\Omega^{(1)}_{i,j}\neq \Omega^{(2)}_{i,j}$).

%

Adopting a neighborhood selection approach \cite{2006_AS_Meinshausen}
as recalled in Section \ref{sec:GGM}, high-dimensional GGM estimation can be solved by multiple high-dimensional
linear regressions. As such, two-sample GGM testing \eqref{eq:hypothese_GGM}
 can be solved via
multiple two-sample hypothesis testing as \eqref{eq:hypothese} in the usual linear regression framework. This extension of two-sample linear regression tests to  GGMs is described in Section 
 \ref{sec:GGM}.



\subsection{Related work}

The literature on high-dimensional two-sample tests is  very light. In
the context of high-dimensional two-sample comparison of means,
\cite{1996_SS_Bai,2008_JMA_Srivastava,2010_AS_Chen,2011_NIPS_Lopes}
have introduced global tests to compare the means of two
high-dimensional Gaussian vectors with unknown variance. Recently,
\cite{cai11_covariance_comparison,chen_li_comparison}  developped 
two-sample tests for covariance matrices of two high-dimensional
vectors.  

In contrast, the one-sample analog of our problem has recently
attracted a lot of attention, offering as many theoretical bases for
extension to the two-sample problem. In fact, the high-dimensional
linear regression tests for the nullity of coefficients can be interpreted as a limit of the two-sample test in the case where $\beta^{(2)}$ is known to be zero, and the sample size $n_2$ is considered infinite so that we perfectly know the distribution of the second sample.

There are basically two different objectives in high-dimensional
linear testing: a local and a global approach. In the
local approach, one considers the $p$ tests for the nullity of each coefficient $\hyp_{0,i}:~\beta_i^{(1)} =0$ ($i=1,\ldots,p$) with the purpose of controling error measures such as the false discovery rate of the resulting multiple testing procedures. 
In a way,  one aims to assess the  individual statistical significance
of each of the variables. This can be achieved by providing a confidence region for $\beta^{(1)}$~\cite{2011_arxiv_Zhang,2012_arXiv_Buhlmann,Montanari,2013arXiv1301.7161L,2013arXiv1309.3489M}.
Another line of work derives $p$-values for the nullity of each of the coefficients.
Namely, \cite{2009_AS_Wasserman}
suggests a screen and  clean procedure based upon half-sampling. Model
selection is first applied upon a random half of the sample in order to test for the significance of each coefficient
using the usual combination of ordinary least squares and Student t-tests on a
model of reasonnable size on the remaining second half. To reduce the dependency
of the results to the splitting, \cite{2009_JASA_Meinshausen} 
advocate to use half-sampling $B$ times and then aggregate the $B$
$p$-values obtained for variable $j$ in
a way which controls either the family-wise error rate or false
discovery rate.




In the global approach, the objective is to test the null hypothesis $\hyp_{0}:~\beta^{(1)}=0$.
Although global approaches are clearly less informative than approaches providing individual
significance tests like \cite{2009_JASA_Meinshausen,2011_arxiv_Zhang,2012_arXiv_Buhlmann},
they can reach  better performances from smaller sample  sizes. Such a
property is of tremendous importance when dealing with high-dimensional datasets.
The idea of \cite{2010_AS_Verzelen}, based upon
the work of \cite{2003_AS_Baraud}, is to approximate the alternative
$\hyp_1:  \beta^{(1)} \neq 0$
by a collection of tractable alternatives $\{\hyp_1^{S}: \exists j
\in S\ ,\ \beta_j^{(1)} \neq 0, S \in \mathcal{S}\}$
working on subsets $S\subset\{1,\ldots,p\}$ of reasonable sizes. The null hypothesis is
rejected if the null hypothesis $\hyp_0^{S}$ is
rejected     for      at     least     one      of     the     subsets
$S\in\mathcal{S}$. Admittedly, the
resulting procedure is computationally intensive. Nonetheless it is
non-asymptotically minimax adaptive to the unknown sparsity of
$\beta^{(1)}$, that is it achieves the optimal rate of detection
without any assumption on the population covariance $\Sigma^{(1)}$ of
the covariates. Another series of work relies on higher-criticism.
This last testing framework was originally introduced in
orthonormal designs \cite{2004_AS_Donoho}, but has been
proved to reach optimal detection rates in high-dimensional linear
regression as well \cite{2011_AS_Arias-Castro,2010_EJS_Ingster}. 
In the end, higher-criticism is  highly
competitive in terms of computing time and achieves the asymptotic rate of detection with the optimal constants. However, these nice properties require strong assumptions on the design.\\ 


While writing this paper, we came across the parallel work of St\"adler
and Mukherjee \cite{2013_ArXiv_Stadler}, which adopts a local approach
in an elegant adaptation of the
\emph{screen    and    clean}    procedure   in    its    simple-split
\cite{2009_AS_Wasserman}  and multi-split \cite{2009_JASA_Meinshausen}
versions to the two-sample framework. Interestingly, their work also led to an extension to GGM
testing \cite{2014_ArXiv_Stadler}. 

In contrast we build our testing strategy upon the global approach
developed by \cite{2003_AS_Baraud} and \cite{2010_AS_Verzelen}. A more
detailed  comparison  of  \cite{2013_ArXiv_Stadler,2014_ArXiv_Stadler}
with   our   contribution   is   deferred  to   simulations   (Section
\ref{sec:num}) and discussion (Section \ref{sec:discussion}).

\subsection{Our contribution}

Our suggested approach stems from the fundamental assumption that either the true
supports of $\beta^{(1)}$ and $\beta^{(2)}$ are sparse or that their
difference $\beta^{(1)}-\beta^{(2)}$ is sparse, so that the test can
be successfully led in a subset $S^\star\subset\{1,\ldots,p\}$ of variables with reasonnable size,
compared to the sample sizes $n_1$ and $n_2$. Of course, this low
dimensional subset $S^\star$ is unknown. The whole objective of the
testing strategy is to achieve similar rates of detection (up to a
logarithmic constant) as an oracle test which would know in advance
the optimal low-dimensional subset $S^\star$.

\medskip

\noindent 
Concretely, we proceed in three steps : 
\begin{enumerate}
\item We define algorithms to select a data-driven collection of
  subsets $\widehat{\cS}$  identified as most informative for our
  testing problem, in an attempt to circumvent the optimal subset $S^\star$.
\item New parametric statistics related to the likelihood ratio statistic 
 between the conditional distributions  $Y^{(1)}|X_S^{(1)}$ and
 $Y^{(2)}|X_S^{(2)}$ are defined for $S \in \widehat{\cS}$. 
\item We define two calibration procedures which both guarantee a control
on type-I error:
\begin{itemize}
\item we use a Bonferroni calibration which is both computationally
  and conceptually simple, allowing us to prove that this procedure is minimax adaptive to the sparsity of
$\beta^{(1)}$ and $\beta^{(2)}$ from a non-asymptotic point of view; 
\item we define a calibration procedure based upon permutations to reach a fine
tuning of multiple testing calibration in practice, for an increase in
empirical power.
\end{itemize}
\end{enumerate}
The resulting testing procedure  is completely data-driven and its type I
 error is explicitely controlled. 
Furthermore, it is computationally amenable in a large $p$ and small $n$ setting. 
 Interestingly, the procedure does not require any half-sampling steps which are known to
 decrease the robustness when the sample size is small.

The  procedure is  described in  Section \ref{sec:test}  while Section
\ref{sec:testdetails}  is devoted  to technical  details,  among which
theoretical controls on Type I error, as well as some useful empirical
tools for interpretation.
 Section \ref{sec:power1} provides a non-asymptotic control of the power. 
 Section \ref{sec:num} provides simulated experiments comparing the performances
of the suggested procedures with the approach of
\cite{2013_ArXiv_Stadler}. In Section \ref{sec:GGM}, we detail the
extension of the procedure to handle the comparison of Gaussian
graphical models. The method is illustrated on Transcriptomic Breast
Cancer Data.  Finally, all the proofs are postponed to Section \ref{sec:proofs}.

The $R$ codes of our algorithms are available at \cite{code}.


%


\subsection{Notation}

In the sequel, $\ell_p$ norms are denoted $|\cdot|_p$, except for the $l_2$ norm which is
 referred as $\|.\|$ to alleviate notations. For any positive
definite matrix $\Sigma$, $\|.\|_{\Sigma}$ denotes the Euclidean norm
associated with the scalar product induced by $\Sigma$: for every
vector $x$, $\|x\|^2_{\Sigma} =x^\intercal\Sigma x$. Besides, fo
every set $S$, $|S|$ denote its cardinality.
For any integer $k$, ${\bf I}_k$ stands for the identity matrix of
size $k$. For any square matrix $A$, $\varphi_{\max}(A)$
and $\varphi_{\min}(A)$ denote respectively the maximum and minimum
eigenvalues of $A$. When the context makes it obvious, we may omit to
mention $A$ to alleviate notations and use $\varphi_{\max}$ and
$\varphi_{\min}$ instead.  Moreover,  $\Y$ refers to the size $n_1+n_2$  concatenation of $\Y^{(1)}$ and  $\Y^{(2)}$ and $\X$ refers to the size $(n_1+n_2)\times p$ the concatenation of  $\X^{(1)}$ and  $\X^{(2)}$.
To finish with, $L$ refers to a positive numerical constant that may vary from line to line.

\section{Description of the testing strategy \label{sec:test}}

Likelihood ratio statistics used to test hypotheses like $\hyp_0$
in the classical \emph{large~$n$, small~$p$} setting are intractable on
high-dimensional datasets for the mere reason that the maximum
likelihood estimator is not itself defined under high-dimensional design
proportions. Our  approach approximates the
intractable high-dimensional test by a multiple testing construction,
similarly to the strategy developped by \cite{2003_AS_Baraud} in order to
derive statistical tests against non-parametric alternatives and
adapted to one sample tests for high-dimensional linear regression in
\cite{2010_AS_Verzelen}. 

For any subset $S$ of $\{1,\ldots,p\}$ satisfying $2|S|\leq n_1\wedge n_2$, denote $X_S^{(1)}$ and $X_S^{(2)}$  the restrictions of the random vectors $X^{(1)}$ and $X^{(2)}$ to covariates indexed by $S$. Their covariance structure is noted  $\Sigma^{(1)}_S$ (resp. $\Sigma^{(2)}_S$). Consider the linear regression of $Y^{(1)}$ by $X_{S}^{(1)}$ defined by 
\[
\begin{cases}
Y^{(1)} &=  X_S^{(1)} \beta_S^{(1)} + \epsilon_S^{(1)} \\
 Y^{(2)} &=  X_S^{(2)} \beta_S^{(2)} + \epsilon_S^{(2)}, 
\end{cases}
\]
where  the  noise variables $\epsilon_S^{(1)}$ and $\epsilon_S^{(2)}$ are independent from $X_{S}^{(1)}$ and $X_S^{(2)}$ and follow centered Gaussian distributions with new unkwown conditional standard deviations $\sigma_S^{(1)}$ and $\sigma_S^{(2)}$.   We now state the test hypotheses in reduced dimension:
\[
\begin{cases}
\hyp_{0,S}:& \beta_S^{(1)}=\beta_S^{(2)}\ ,\quad \sigma_S^{(1)}=\sigma_S^{(2)}\ ,\quad  \text{and}\quad  \Sigma^{(1)}_S=
\Sigma^{(2)}_S, \\
\hyp_{1,S}:& \beta_S^{(1)}\neq\beta_S^{(2)} \quad \text{or}\quad \sigma_S^{(1)}\neq\sigma_S^{(2)}.
\end{cases}
\]

Of course, there is no reason in general for $\beta_S^{(1)}$ and
$\beta_S^{(2)}$ to coincide with the restrictions of $\beta^{(1)}$ and
$\beta^{(2)}$ to $S$, even less in high-dimension since variables in
$S$ can be in all likelihood correlated with covariates in $S^c$. Yet,
as exhibited by Lemma \ref{lemma_equivalence_hypothese}, there is
still a strong link between the collection of low dimensional hypotheses $\hyp_{0,S}$ and the global null hypothesis $\hyp_0$.

\begin{lemma}\label{lemma_equivalence_hypothese}
The hypothesis $\hyp_{0}$ implies  $\hyp_{0,S}$ for any subset $S\subset\{1,\ldots p\}$.
\end{lemma}
\begin{proof}
 Under $\hyp_0$, the random vectors of size $p+1$ $(Y^{(1)}, X^{(1)})$ and $(Y^{(2)}, X^{(2)})$ follow the same distribution. Hence, for any subset $S$, $Y^{(1)}$ follows conditionally on $X_S^{(1)}$ the same distribution as $Y^{(2)}$ conditionnally on $X_S^{(2)}$. In other words, $\beta_S^{(1)} = \beta_S^{(2)}$.
\end{proof}

By contraposition, it suffices to reject at least one of the $\hyp_{0,S}$ hypotheses to reject the global null hypothesis. This fundamental observation motivates our testing procedure. 
As summarized in Algorithm \ref{algo:overall}, the idea is to build a
well-calibrated multiple testing procedure that considers the testing problems $\hyp_{0,S}$ against $\hyp_{1,S}$ for a collection of subsets $\cS$. 
Obviously, it would be prohibitive in terms of algorithmic complexity to test  $\hyp_{0,S}$ for every $S\subset \{1,\ldots,p\}$, since there would be $2^p$ such sets. As a result, we restrain ourselves to a relatively small collection of hypotheses $\{\hyp_{0,S},S\in \widehat{\cS}\}$, where the collection of supports $\widehat{\cS}$ is potentially data-driven. If the collection $\widehat{\cS}$ is judiciously selected, then we can manage not to lose too much power compared to the exhaustive search.\\

We now turn to the description of the three major elements required by
our overall strategy (see Algorithm \ref{algo:overall}):
\begin{enumerate}
\item a well-targeted data-driven collection of models $\widehat{\cS}$
  as produced by Algorithm \ref{algo:SLasso};
\item a parametric statistic to test the hypotheses $\hyp_{0,S}$ for $S
  \in \widehat{\cS}$, we
  resort actually to a combination of three parametric statistics
  $F_{S,V}$, $F_{S,1}$ and $F_{S,2}$;
\item a calibration procedure guaranteeing the control on type I error
  as in Algorithm \ref{algo:Bonferroni} or \ref{algo:permutations}.
\end{enumerate}

\begin{algorithm}
\caption{Overall Adaptive Testing Strategy}\label{algo:overall}
\begin{algorithmic}
\Require{Data $\mathbf{X}^{(1)}$,$\mathbf{X}^{(2)}$,$\mathbf{Y}^{(1)}$,$\mathbf{Y}^{(2)}$, collection $\cS$ and desired level $\alpha$}
\AlgPhase{Step 1 - Choose a collection $\widehat{\cS}$ of low-dimensional models
(as e.g. $\SLasso$ in Algorithm \ref{algo:SLasso})}
\Procedure{ModelChoice}{$\mathbf{X}^{(1)}$,$\mathbf{X}^{(2)}$,$\mathbf{Y}^{(1)}$,$\mathbf{Y}^{(2)}$,$\cS$}
\State Define the model collection $\widehat{\cS}\subset \cS$
\EndProcedure
\AlgPhase{Step 2 - Compute p-values for each test in low dimension}
\Procedure{Test}{$\mathbf{X}_S^{(1)}$,$\mathbf{X}_S^{(2)}$,$\mathbf{Y}^{(1)}$,$\mathbf{Y}^{(2)}$,$\widehat{\cS}$}
	\For{each subset $S$ in $\widehat{\cS}$}
		\State Compute the p-values $\tilde{q}_{V,S}$, $\tilde{q}_{1,S}$, $\tilde{q}_{2,S}$ associated to the statistics $F_{S,V}$, $F_{S,1}$, $F_{S,2}$
	\EndFor
\EndProcedure
\AlgPhase{Step 3 - Calibrate decision thresholds  as in Algorithms \ref{algo:Bonferroni} (Bonferroni)
or \ref{algo:permutations} (Permutations)}
\Procedure{Calibration}{$\mathbf{X}^{(1)}$,$\mathbf{X}^{(2)}$,$\mathbf{Y}^{(1)}$,$\mathbf{Y}^{(2)}$,$\widehat{\cS}$,$\alpha$}
	\For{each subset $S$ in $\widehat{\cS}$ and each $i=V,1,2$}
		\State  Define a threshold $\alpha_{i,S}$. 
	\EndFor
\EndProcedure
\AlgPhase{Step 4 - Final Decision}
\If{there is a least one model $S$ in $\widehat{\cS}$ such that there is
  at least one  p-value for which $\tilde{q}_{i,S} < \alpha_{i,S}$} 
\State Reject the global null hypothesis $\hyp_0$
\EndIf
\end{algorithmic}
\end{algorithm}

\subsection{Choices of Test Collections (Step 1)}\label{subsec:choice_collection}

The first step of our procedure (Algorithm \ref{algo:overall}) amounts to select a collection $\widehat{\cS}$ of subsets of $\{1,\ldots,p\}$, also called models.
A good collection $\widehat{\mathcal{S}}$ of subsets  must satisfy a tradeoff between the inclusion of the maximum number of relevant subsets $S$ and a reasonable computing time for the whole testing procedure, which is linear in the size $|\widehat{\cS}|$ of the collection. The construction of $\widehat{\cS}$ proceeds in two steps: (i) One chooses a
 {\it deterministic} collection $\cS$ of models. (ii) One defines an algorithm  (called {\em ModelChoice} in Algorithm \ref{algo:overall}) mapping the raw data $(\X,\Y)$
to some  collection  $\widehat{\cS}$ satisfying $\widehat{\cS}\subset \cS$. Even though the introduction of $\cS$ as an argument of the mapping could appear artificial at this point,  this quantity will be used in the calibration step of the procedure. 
Our methodology can be applied to any fixed or data-driven collection. Still, we focus here on two particular collections. The first one is useful for undertaking the first steps of the mathematical analysis. For practical applications, we advise to use the second collection.


\paragraph{Deterministic Collections $\cS_{\leq k}$.} By deterministic, we mean the model choice step is trivial in the sense  {\em ModelChoice}$(\X,\Y,\cS)=\cS$.
Among deterministic collections, the most straightforward collections consist of all size-$k$ subsets of $\{1,\ldots, p\}$, which we denote $\cS_k$. This kind of family provides collections which are independent from the data, 
thereby reducing the risk of overfitting. However, as we allow the model size $k$ or the total number of candidate variables $p$ to grow, 
these deterministic families can rapidly reach unreasasonble sizes. Admittedly, $\cS_1$ always remains feasible, but reducing the search to models of size 1 
can be costly in terms of power. As a variation on size $k$ models, we
introduce  the collection of all models  of size smaller than $k$, denoted $\cS_{\leq k} = \bigcup_{j=1}^k
\cS_j$, which will prove useful in theoretical developments.


\paragraph{Lasso-type Collection $\SLasso$.}
Among all data-driven collections, we suggest the Lasso-type collection $\SLasso$. Before proceeding to its definition, let us informally discuss the subsets that a ``good'' collection $\widehat{\cS}$ should contain.
Let $\mathrm{supp}(\beta)$ denote the support of a vector $\beta$. 
Intuitively,   under the alternative hypothesis, good candidates for the subsets are either subsets of $S^*_\vee:= \mathrm{supp}(\beta^{(1)})\cup \mathrm{supp}(\beta^{(2)})$ or subsets of $S^*_{\Delta}:= \mathrm{Supp}(\beta^{(1)}-\beta^{(2)})$. The first model $S^*_\vee$ nicely satisfies $\beta^{(1)}_{S^*_\vee}= \beta^{(1)}$ and $\beta^{(2)}_{S^*_\vee}= \beta^{(2)}$. The second subset has a smaller size than $S^{*}_\vee$ and focuses on covariates corresponding to different parameters in the full regression. However, the divergence between effects might only appear
conditionally on other variables with similar effects, this is why the
first subset $S^*_\vee$ is also of interest.
Obviously, both subsets $S^*_\vee$ and $S^*_{\Delta}$ are unknown. This is why we consider  a Lasso methodology that amounts to estimating both $S^*_\vee$ and $S^*_{\Delta}$ in the collection $\SLasso$. Details on the construction of $\SLasso$ are postponed to Section \ref{sec:Slasso}.

\subsection{Parametric Test Statistic (Step 2)\label{subsec:parametric}}

Given a subset $S$, we consider the three following statistics to test $\mathcal{H}_{0,S}$ against $\mathcal{H}_{1,S}$:
\begin{eqnarray}\label{eq:FSV}
F_{S,V}& := &-2 +
\frac{\|\Y^{(1)}-\X^{(1)}\widehat{\beta}^{(1)}_S\|^2/n_1}{\|{\bf
    Y}^{(2)}-
  \X^{(2)}\widehat{\beta}^{(2)}_S\|^2/n_2}+\frac{\|\Y^{(2)}-\X^{(2)}\widehat{\beta}^{(2)}_S\|^2/n_2}{\|\Y^{(1)}-\X^{(1)}\widehat{\beta}_S^{(1)}\|^2/n_1}\
,\\ \label{eq:FS1}
 F_{S,1} & :=& \frac{\|{\bf
X}^{(2)}_{S}(\widehat{\beta}^{(1)}_S-\widehat{\beta}^{
(2)}_S)\|^2/n_2 }{
\|\Y^{(1)}-\X^{(1)}\widehat{\beta}^{(1)}_S\|^2/n_1}\ ,
\quad \quad F_{S,2}  := \frac{\|{\bf
X}^{(1)}_{S}(\widehat{\beta}^{(1)}_S-\widehat{\beta}^{
(2)}_S)\|^2/n_1}{\|\Y^{(2)}-\X^{(2)}\widehat{\beta}^{(2)}_S\|^2/n_2} \ . 
\end{eqnarray}
As explained in Section \ref{sec:testdetails}, these three statistics derive 
 from the
Kullback-Leibler divergence between the conditional distributions  $Y^{(1)}|X_S^{(1)}$ and $Y^{(2)}|X_S^{(2)}$.
 While the first term $F_{S,V}$ evaluates the discrepancies in terms of conditional variances, the last two terms $F_{S,1}$ and $F_{S,2}$ address the comparison of $\beta^{(1)}$ to $\beta^{(2)}$.  

Because the distributions under the null of the statistics $F_{S,i}$, for $i=V,1,2$, depend on the size of  $S$, the only way to calibrate the multiple testing step over a collection of models of various sizes is to convert the statistics to a unique common scale. 
The most natural is to convert observed $F_{S,i}$'s into $p$-values. Under $H_{0,S}$, the conditional distributions of $F_{S,i}$ for $i=V,1,2$ to ${\bf X}_S$ are parameter-free
 and explicit (see Proposition \ref{prte:distribution_parametrique} in the next section). Consequently, one can define the exact $p$-values associated
  to $F_{S,i}$, conditional on $\X_S$. However, the computation of the $p$-values require a function inversion, which can be computationally prohibitive. 
  This is why we introduce explicit upper bounds $\tilde{q}_{i,S}$ (Equations (\ref{eq:defQVS},\ref{eq:defQV12})) of the exact $p$-values.

\subsection{Combining the parametric statistics (Step 3)\label{subsec:calibration}}

The objective of this subsection is to calibrate a multiple testing
procedure     based     on      the     sequence     of     $p$-values
$\{(\tilde{q}_{V,S},\tilde{q}_{1,S},\tilde{q}_{2,S})      ,\      S\in
\widehat{\cS}\}$,  so that  the type-I  error remains  smaller  than a
chosen level  $\alpha$. In particular, when using  a data-driven model
collection, we must take good care of preventing the risk of overfitting which results from using the same
dataset both for model selection and hypothesis testing. 

For the sake of simplicity, we assume in the two following paragraphs that $\emptyset \nsubseteq \cS$, which merely means that we do not include in the collection of tests the raw comparison of $\var{(\Y^{(1)})}$ to $\var{(\Y^{(2)})}$.

\paragraph{Testing Procedure}
Given a model collection $\widehat{\cS}$ and a sequence $\widehat{\alpha} =
(\alpha_{i,S})_{ i=V,1,2,\  S\in\widehat{\cS}}$, we define the test function:
\begin{equation}\label{eq:definition_general_test}
T_{\widehat{\cS}}^{\widehat{\alpha}} = \begin{cases} 1 & \text{if } \exists S
  \in \widehat{\mathcal{S}},\ \exists i \in
  \{V,1,2\} \quad \tilde{q}_{i,S} \leq \alpha_{i,S}. 
\\ 0 & \text{otherwise.}\end{cases}
\end{equation}
In other words, the test function rejects the global null if there exists at least one model $S~\in~\widehat{\cS}$ such that at least one of the three $p$-values 
is below the corresponding threshold $\alpha_{i,S}$.
In Section \ref{sec:calibration}, we describe two calibration methods for choosing the thresholds $(\alpha_{i,S})_{S\in \widehat{\cS}}$. 
We first define a natural Bonferroni procedure, whose conceptual simplicity allows us to derive non-asymptotic type II error bounds of the corresponding tests (Section \ref{sec:power1}). 
However, this Bonferroni correction reveals too conservative in practice, in part paying the price for resorting to data-driven collections and upper bounds on the true $p$-values.
This is why we introduce as a second option the permutation calibration procedure. This second procedure controls the type I error at the nominal level and therefore outperforms
the Bonferroni calibration in practice. Nevertheless, the mathematical analysis
of the corresponding test becomes more intricate and we are not able to provide sharp type II error bounds.

\medskip 

\noindent 
{\bf Remark}: In practice, we advocate the use of the Lasso Collection $\SLasso$ (Algorithm~\ref{algo:SLasso}) combined with the permutation calibration method (Algorithm \ref{algo:permutations}).  Henceforth, the corresponding procedure is denoted $T^{P}_{\SLasso}$.

\section{Discussion of the procedure and Type I error \label{sec:testdetails}}

In this section, we provide remaining details on the three steps of the testing procedure. First, we describe the collection $\SLasso$ and provide an informal justification of its definition. Second, we explain the ideas underlying the parametric statistics $F_{S,i}$, $i=V,1,2$ and we define the corresponding $p$-values $\tilde{q}_{i,S}$. Finally, the Bonferroni and permutation calibration methods are defined, which allows us to control the type I error of the corresponding testing procedures.

\subsection{Collection $\SLasso$}\label{sec:Slasso}

We start from $\cS_{\leq D_{\max}}$, where, in practice,  $D_{\max}=\lfloor (n_1\wedge n_2)/2\rfloor$ and we consider the following reparametrized joint regression model.
\begin{equation}
\label{eq:reparam}
\left[\begin{array}{c}
       {\bf Y}^{(1)}\\{\bf Y}^{(2)}
      \end{array}
\right] = \left[\begin{array}{cc}
                 {\bf X}^{(1)}& {\bf X}^{(1)} \\
		{\bf X}^{(2)} & -{\bf X}^{(2)}
                \end{array}
\right]\left[
\begin{array}{c}
      \theta_*^{(1)}\\ \theta_*^{(2)}
      \end{array}
\right] + \left[\begin{array}{c}
                 \eps^{(1)} \\ \eps^{(2)}
                \end{array}
\right]\ .
\end{equation}
In this new model,
$\theta_*^{(1)}$ captures the mean effect $(\beta^{(1)} +
\beta^{(2)})/2$, while $\theta_*^{(2)}$ captures the discrepancy
between the sample-specific effect $\beta^{(i)}$ and the mean effect
$\theta_*^{(1)}$, that is to say $\theta_*^{(2)} = (\beta^{(1)} -
\beta^{(2)})/2$. Consequently, $S^*_{\Delta}:= \mathrm{Supp}(\beta^{(1)}-\beta^{(2)})= \mathrm{supp}(\theta_*^{(2)})$ and $S^*_{\vee}:= \mathrm{supp}(\beta^{(1)})\cup \mathrm{supp}(\beta^{(2)})= \mathrm{supp}(\theta_*^{(1)})\cup \mathrm{supp}(\theta_*^{(2)})$. To simplify notations, denote by $\bf{Y}$ the concatenation of
${\bf{Y}}^{(1)}$ and ${\bf{Y}}^{(2)}$, as well as by ${\bf W}$ the
reparametrized design matrix of \eqref{eq:reparam}.  For a given
$\lambda>0$, the Lasso estimator of $\theta_*$ is defined by 
\begin{eqnarray}
\label{eq:reparamLasso}
\widehat{\theta}_\lambda :=\left(\begin{array}{c}
\widehat{\theta}_{\lambda}^{(1)}\\                                   \widehat{\theta}_{\lambda}^{(2)}
                                 \end{array}
\right) := \arg \min_{\theta \in \mathbb{R}^{2p}} \|{\bf{Y}} - {\bf W} \theta\| + \lambda \|\theta\|_1\ , \\ \hat{V}_{\lambda}:=\mathrm{supp}(\widehat{\theta}_{\lambda}),\quad   \hat{V}^{(i)}_{\lambda}:=\mathrm{supp}(\widehat{\theta}^{(i)}_{\lambda}),\ i=1,2\ . \label{eq:reparamLasso2}
\end{eqnarray}
For a suitable choice of the tuning parameter $\lambda$ and under assumptions of the designs, it is proved~\cite{2009_AS_Bickel,2008_AS_Meinshausen} that $\widehat{\theta}_{\lambda}$ estimates well $\theta_*$ and $\hat{V}_{\lambda}$ is a good estimator of $\mathrm{supp}(\theta_*)$. The Lasso parameter $\lambda$ tunes the amount of sparsity of
$\widehat{\theta}_\lambda$: the larger the parameter $\lambda$, the smaller the
support $\hat{V}_{\lambda}$. As the optimal choice of $\lambda$ is unknown, the collection $\SLasso$ is built using the collection of all estimators $(\hat{V}_{\lambda})_{\lambda>0}$, also called the Lasso regularization path of $\theta_*$. Below we provide an algorithm for computing $\SLasso$ along with some  additional justifications.



\begin{algorithm}
\caption{Construction of the Lasso-type Collection $\SLasso$}\label{algo:SLasso}
\begin{algorithmic}
\Require{Data $\mathbf{X}^{(1)}$,$\mathbf{X}^{(2)}$,$\mathbf{Y}^{(1)}$,$\mathbf{Y}^{(2)}$, Collection $\cS_{\leq D_{max}}$}
\State $\bf{Y} \gets \left[\begin{array}{c}{\bf Y}^{(1)}\\{\bf Y}^{(2)}\end{array}\right]$
\State $\bf{W} \gets \left[\begin{array}{cc}
        {\bf X}^{(1)}& {\bf X}^{(1)} \\
 		{\bf X}^{(2)} & - {\bf X}^{(2)}
                 \end{array}\right]$
\State Compute the function $f:\lambda\ \mapsto \hat{V}_{\lambda}
$
 (defined in (\ref{eq:reparamLasso},\ref{eq:reparamLasso2}))
using Lars-Lasso Algorithm~\cite{lars}
\State Compute the decreasing sequences $(\lambda_k)_{1\leq k\leq q}$ of jumps in $f$
\State $k \gets 1$,\quad  $\widehat{\mathcal{S}}_{L}^{(1)}\gets \emptyset$,\quad  $\widehat{\mathcal{S}}_{L}^{(2)}\gets \emptyset$
 \While{$|\hat{V}^{(1)}_{\lambda_k}\cup\hat{V}^{(2)}_{\lambda_k} | <D_{\max}$}
\State $\widehat{\mathcal{S}}_{L}^{(1)}\gets \widehat{\mathcal{S}}_{L}^{(1)}\cup \{\hat{V}^{(1)}_{\lambda_k}\cup\hat{V}^{(2)}_{\lambda_k}\} $
\State $\widehat{\mathcal{S}}_{L}^{(2)}\gets \widehat{\mathcal{S}}_{L}^{(2)} \cup \{\hat{V}^{(2)}_{\lambda_k}\}$
\State  $k \gets k+1$
\EndWhile
\State$\SLasso \gets \widehat{\mathcal{S}}^{(1)}_L \cup \widehat{\mathcal{S}}_L^{(2)} \cup \cS_1$

\end{algorithmic}
\end{algorithm}

It is known~\cite{lars} that the function $f:\lambda \mapsto \hat{V}_{\lambda}$ is piecewise constant. Consequently, there exist thresholds $\lambda_1>\lambda_2>\ldots$ such that $\hat{V}_{\lambda}$  changes on $\lambda_k$'s only. The function $f$ and the collection $(\lambda_k)$ are computed efficiently using the Lars-Lasso Algorithm~\cite{lars}. We build two collections of models using $(\hat{V}^{(1)}_{\lambda_k})_{k\geq 1}$ and $(\hat{V}^{(2)}_{\lambda_k})_{k\geq 1}$. Following the intuition described above, for a fixed $\lambda_k$, $\hat{V}^{(2)}_{\lambda_k}$ is an estimator of  $\mathrm{supp}(\beta^{(1)}-\beta^{(2)})$ while $\hat{V}^{(1)}_{\lambda_k}\cup \hat{V}^{(2)}_{\lambda_k}$ is an estimator of  $\mathrm{supp}(\beta^{(1)})\cup \mathrm{supp}(\beta^{(2)})$. This is why we define 
\[
\widehat{\mathcal{S}}^{(1)}_L:= \bigcup_{k=1}^{k_{\max}}\left\{\hat{V}^{(1)}_{\lambda_k}\cup \hat{V}^{(2)}_{\lambda_k} \right\}\ , \quad \quad 
\quad \widehat{\mathcal{S}}^{(2)}_L:= \bigcup_{k=1}^{k_{\max}}\left\{\hat{V}^{(2)}_{\lambda_k} \right\}\ ,
\]
where $k_{\max}$ is the smallest integer $q$ such that $|\hat{V}^{(1)}_{\lambda_{q+1}}\cup \hat{V}^{(2)}_{\lambda_{q+1}}|>D_{\max}$. In the end, we consider the following $\SLasso$ data-driven family,
\begin{equation}\label{eq:definition_family_lasso}
\SLasso: = \widehat{\mathcal{S}}^{(1)}_L \cup \widehat{\mathcal{S}}_L^{(2)} \cup \cS_1. 
\end{equation}
 Recall that $\cS_1$ is the collection of the $p$ models of size 1. Recently, data-driven procedures have been proposed to tune the Lasso and find a parameter $\widehat{\lambda}$ is such a way that $\widehat{\theta}_{\widehat{\lambda}}$ is a good estimator of $\theta_*$ (see e.g. \cite{2011_arxiv_Sun,2010_arxiv_Baraud}). We use the whole regularization path instead of the sole estimator $\widehat{\theta}_{\widehat{\lambda}}$, because our objective is to find subsets $S$ such that the statistics $F_{S,i}$ are powerful. Consider an example where $\beta^{(2)}=0$ and $\beta^{(1)}$ contains one large coefficient and many small coefficients. If the sample size is large enough, a well-tuned Lasso estimator will select several variables. In contrast, the best subset $S$ (in terms of power of $F_{S,i}$) contains only one variable. Using the whole regularization path, we hope to find the best trade-off between sparsity (small size of $S$) and differences between $\beta^{(1)}_S$ and $\beta^{(2)}_S$. This last remark is 
formalized in Section \ref{sec:power_Slasso}. 
Finally, the size of the collection $\SLasso$ is generally linear with $(n_1\wedge n_2)\vee p$, which makes the computation of $(\tilde{q}_{i,S})_{S\in \SLasso, i=V,1,2}$ reasonnable.






\subsection{Parametric statistics and $p$-values}
\subsubsection{Symmetric conditional likelihood}

In this subsection, we explain the intuition behind  the choice of the parametric statistics $(F_{S,V},F_{S,1},F_{S,2})$ defined in Equations (\ref{eq:FSV},\ref{eq:FS1}).
Let us denote by $\mathcal{L}^{(1)}$ (resp. $\mathcal{L}^{(2)}$) the  $\log$-likelihood of the first (resp. second)
sample normalized by $n_1$ (resp. $n_2$). Given a subset
$S\subset\{1,\ldots, p\}$ of size smaller than $n_1\wedge n_2$,
$(\widehat{\beta}^{(1)}_S,\widehat{\sigma}_{S}^{(1)})$ stands for the
maximum likelihood estimator of $(\beta^{(1)},\sigma^{(1)})$ among vectors $\beta$ whose supports are included
in $S$. Similarly, we note
$(\widehat{\beta}^{(2)}_S,\widehat{\sigma}_{S}^{(2)})$ for the maximum
likelihood corresponding to the second sample.

Statistics $F_{S,V}$, $F_{S,1}$ and $F_{S,2}$ appear as the
decomposition of a two-sample likelihood-ratio, measuring the
symmetrical adequacy of sample-specific estimators to the opposite sample. To do so, let us define the likelihood ratio in sample $i$ between an arbitrary pair~$(\beta,\sigma)$ and the corresponding sample-specific estimator $(\widehat{\beta}^{(i)}_S,\widehat{
 \sigma } ^ { (i) } _S)$:
\[
 \D_{n_i}^{(i)}(\beta,\sigma) := \mathcal{L}_{n_i}^{(i)}\left(\widehat{\beta}^{(i)}_S,\widehat{
 \sigma } ^ { (i) }
 _S\right)-\mathcal{L}_{n_i}^{(i)}\left(\beta, \sigma \right).
\]
With this definition,
$\D_{n_1}^{(1)}(\widehat{\beta}^{(2)},\widehat{\sigma}^{(2)})$
measures how far $(\widehat{\beta}^{(2)},\widehat{\sigma}^{(2)})$ is
from $(\widehat{\beta}^{(1)},\widehat{\sigma}^{(1)})$ in terms of
likelihood within sample 1.
The following symmetrized likelihood statistic can be decomposed into the sum of
$F_{S,V}$, $F_{S,1}$ and $F_{S,2}$:
\begin{equation}\label{eq:oldFS}
2\left[ \D_{n_1}^{(1)}(\widehat{\beta}^{(2)},\widehat{\sigma}^{(2)}) +
  \D_{n_2}^{(2)}(\widehat{\beta}^{(1)},\widehat{\sigma}^{(1)})\right]
= F_{S,V} + F_{S,1} + F_{S,2}\ .
\end{equation}
Instead of the three statistics $(F_{S,i})_{i=V,1,2}$, one could  use the symmetric likelihood \eqref{eq:oldFS} to build a testing procedure. However, we do not manage to obtain an explicit and sharp upper bound of the $p$-values associated to the statistic \eqref{eq:oldFS}, which makes the resulting procedure either computationally intensive if one estimated the $p$-values by a Monte-Carlo approach or less powerful if one uses a non-sharp upper bound of the $p$-values. In contrast, we explain below how, by considering separately $F_{S,V}$, $F_{S,1}$ and $F_{S,2}$, one upper bounds sharply the exact $p$-values.

\subsubsection{Definition of the $p$-values}\label{section:pvalues}

Denote by $g(x)=-2+x+1/x$ the non-negative function defined on $\mathbb{R}^+$. Since the restriction of $g$ to $[1;+\infty)$ is a bijection, we note $g^{-1}$ the corresponding reciprocal function. 

\begin{proposition}[Conditional distributions of $F_{S,V}$, $F_{S,1}$ and $F_{S,2}$ under $\hyp_{0,S}$]\label{prte:distribution_parametrique}
~\\
\noindent
\begin{enumerate}
\item Let $Z$ denote a Fisher random variable with $(n_1-|S|,n_2-|S|)$ degrees of freedom. Then, under the null hypothesis,
\[
F_{S,V} | \X_S \quad \underset{\hyp_{0,S}}{\sim} \quad g\left[Z\ \frac{n_2(n_1-|S|)}{n_1(n_2-|S|)}\right].
\]
\item Let $Z_1$ and $Z_2$ be two centered and independent Gaussian vectors with covariance $\X_S^{(2)}\left[(\X_S^{(1)T}\X_S^{(1)})^{-1}+(\X_S^{(2)T}\X_S^{(2)})^{-1}\right]\X_S^{(2)T}$ and ${\bf I}_{n_1-|S|}$. Then, under the null hypothesis,
\[
F_{S,1} | \X_S \quad \underset{\hyp_{0,S}}{\sim}\quad \frac{\|Z_1\|^2/n_2}{\|Z_2\|^2/n_1}.
\]
A symmetric result holds for $F_{S,2}$.
\end{enumerate}
\end{proposition}


Although the distributions identified in Proposition
\ref{prte:distribution_parametrique} are not all familiar distributions
with ready-to-use quantile tables, they all share the advantage that
they do not depend on any unknown quantity, such as design variances
$\Sigma^{(1)}$ and $\Sigma^{(2)}$, noise variances $\sigma^{(1)}$ and
$\sigma^{(2)}$, or even true signals $\beta^{(1)}$ and
$\beta^{(2)}$. For any $i=V,1,2$, we note
$\overline{Q}_{i,|S|}(u|\X_{S})$ for the conditional probability that $F_{S,i}$ is larger than $u$.

By Proposition \ref{prte:distribution_parametrique}, the exact $p$-value $\tilde{q}_{V,S}= \overline{Q}_{V,|S|}(F_{S,V}|\X_{S})$ associated to $F_{S,V}$ is easily computed 
from the distribution function of a Fisher random variable:
\begin{equation}\label{eq:defQVS}
\tilde{q}_{V,S}= \mathcal{F}_{n_1-|S|,n_2-|S|}\left[g^{-1}\left(F_{S,V}\right)\frac{n_1(n_2-|S|)}{n_2(n_1-|S|)}\right]+ \mathcal{F}_{n_2-|S|,n_1-|S|}\left[g^{-1}\left(F_{S,V}\right)\frac{n_2(n_1-|S|)}{n_1(n_2-|S|)}\right]\ ,
\end{equation}
where $\mathcal{F}_{m,n}(u)$ denotes the probability that a Fisher random variable with $(m,n)$ degrees of freedom is larger than $u$.

 Since the conditional distribution  of $F_{S,1}$ given $X_S$ only depends on $|S|$, $n_1$, $n_2$, and $\X_S$, one could compute an estimation of the $p$-value $\overline{Q}_1(u|X_S)$ associated with an observed value $u$ by Monte-Carlo simulations. However, this approach is computationally prohibitive for large collections of subsets $S$. This is why we  use instead an explicit upper bound of  $\overline{Q}_{1,
|S|}(u|\X_{S})$ based on Laplace method, as given in the definition below and justified in the proof of Proposition \ref{prte:majoration_Q}.

\noindent

\begin{definition}[Definition of $\widetilde{Q}_{1,|S|}$ and $\widetilde{Q}_{2,|S|}$]\label{defi:deviation}

Let us note $a=(a_1,\ldots , a_{|S|})$ the positive eigenvalues of 
\[\frac{n_1}{n_2(n_1-|S|)}\X_S^{(2)}\left[(\X_S^{(1)T}\X_S^{(1)})^{-1}+(\X_S^{(2)T}\X_S^{(2)})^{-1}\right]\X_S^{(2)T}\ .\] 
For any $u\leq  |a|_1$, define
$ \widetilde{Q}_{1,|S|}(u|\X_{S}):=1$.
For any $u>|a|_1$, take
\begin{eqnarray}\label{definition_Q1,S}
 \widetilde{Q}_{1,|S|}(u|\X_{S}):=  \exp\left[-\frac{1}{2}\sum_{i=1}^{|S|}\log\left(1-2\lambda^* a_i\right)-\frac{n_1-|S|}{2}\log\left(1+\frac{2\lambda^*u}{n_1-|S|}\right)\right]\ ,
\end{eqnarray}
where $\lambda_*$ is defined as follows.
 If all the components of $a$ are equal, then 
$ \lambda^*:=\frac{u-|a|_1}{2u(|a|_{\infty}+\frac{|a|_1}{n_1-|S|})}$. If $a$ is not a constant vector, then we define
\begin{eqnarray}
\notag  b &:=&\frac{|a|_1u}{|a|_{\infty}(n_1-|S|)}+u+
\frac{\|a\|^2}{|a|_{\infty}}-|a|_1,\\  
\Delta &:=& b^2-\frac{4u
  \left(u-|a|_1\right)}{(n_1-|S|)|a|_{\infty}}\left(|a|_1-\frac{\|a\|^2}{|a|_{\infty}}\right), \label{eq:def_delta}\\
\lambda^*&:=&\frac{1}{\frac{4u}{n_1-|S|}\left(|a|_1-\frac{\|a\|^2}{|a|_{\infty}}\right)} \left(b-\sqrt{\Delta} \right)\ . \label{eq:def_lambda}
\end{eqnarray}
$\widetilde{Q}_{2,|S|}$ is defined analogously by exchanging the role of $\X_S^{(1)}$ and $\X_{S}^{(2)}$.
\end{definition}

\begin{proposition}\label{prte:majoration_Q}
For any $u\geq 0$, and for $i=1,2$, 
$\overline{Q}_{i,|S|}(u|\X_{S}) \leq \widetilde{Q}_{i,|S|}(u|\X_{S})$.
\end{proposition}

Finally we define the approximate $p$-values $\tilde{q}_{1,S}$ and $\tilde{q}_{2,S}$ by
\begin{equation}\label{eq:defQV12}
\tilde{q}_{1,S}:= \widetilde{Q}_{1,|S|}(F_{S,1}|\X_{S})\ , \quad\quad \tilde{q}_{2,S}:= \widetilde{Q}_{2,|S|}(F_{S,2}|\X_{S})\ .
\end{equation}

Although we use similar notations for $\tilde{q}_{i,S}$ with $i=V,1,2$, 
this must not mask the essential difference that $\tilde{q}_{1,S}$ is the {\em exact} $p$-value of $F_{S,1}$ whereas $\tilde{q}_{1,S}$ and $\tilde{q}_{2,S}$ only are upper-bounds on $F_{S,2}$ and $F_{S,2}$ $p$-values. The consequences of this asymetry in terms of calibration of the test is discussed in the next subsection.

\subsection{Comparison of the calibration procedures and Type I error}\label{sec:calibration}

\subsubsection{Bonferroni Calibration  (B)}

Recall that a data-driven model collection $\widehat{\cS}$ is defined as
the result of a fixed algorithm mapping a deterministic collection $\cS$ and 
$(\X,\Y)$ to a subcollection $\widehat{\cS}$. 
The collection of thresholds $\widehat{\alpha}^B = \left\{\alpha_{i,S},\ S\in\widehat{\cS}\right\}$ is chosen such that
\begin{equation} \label{condition_bonf}
\sum_{S \in \cS }\sum_{i=V,1,2}\alpha_{i,S} \leq  \alpha \ .
\end{equation}
For the collection $\cS_{\leq k}$, or any data-driven collection derived from $\cS_{\leq k}$, a natural choice is  
\begin{equation}\label{eq:condition_bonf_classique}
 \alpha_{V,S}:=\frac{\alpha}{2k}\binom{p}{|S|}^{-1}\ ,\quad \alpha_{1,S}=\alpha_{2,S}:=\frac{\alpha}{4k}\binom{p}{|S|}^{-1}\ ,
\end{equation}	
which puts as much weight to the comparison of the conditional
variances ($F_{S,V}$) and the comparison of the coefficients $(F_{S,1},F_{S,2})$.
Similarly for the collection $\SLasso$, a natural choice is (\ref{eq:condition_bonf_classique}) with $k$ replaced by $D_{\max}$ (which equals $\lfloor (n_1\wedge n_2)/2\rfloor$ in practice).

\begin{algorithm}
\caption{Bonferroni Calibration for a collection $\widehat{\cS}\subset \mathcal{S}_{\leq D_{\max}}$}\label{algo:Bonferroni}
\begin{algorithmic}
\Require{maximum model dimension $D_{max}$, model collection $\widehat{\cS}$, desired level $\alpha$}
\For{each subset $S$ in $\widehat{\cS}$}
	\State $\alpha_{V,S} \gets \alpha (2D_{max})^{-1} \binom{p}{|S|}^{-1}$
	\State $\alpha_{1,S} \gets \alpha (4D_{max})^{-1} \binom{p}{|S|}^{-1}$,\quad $\alpha_{2,S} \gets \alpha_{1,S}$
\EndFor
\end{algorithmic}
\end{algorithm}

Given any data-driven collection $\widehat{\cS}$, denote by $T^{B}_{\widehat{\cS}}$ the multiple testing procedure calibrated by Bonferroni thesholds $\widehat{\alpha}^B$ \eqref{condition_bonf}. 

\begin{proposition}[Size of $T^{B}_{\widehat{\cS}}$]\label{prop:sizeB}
 The test function $T^{B}_{\widehat{\cS}}$ satisfies  $\mathbb{P}_{\hyp_0}[T^{B}_{\widehat{\cS}}= 1]\leq \alpha$. 
\end{proposition}

\begin{remark}[Bonferroni correction on $\cS$ and not on $\widehat{\cS}$]\label{rem:bonf_corr}
Note that even though we only compute the statistics $F_{S,i}$ for models $S\in \widehat{\cS}$, the Bonferroni correction  \eqref{condition_bonf} must be applied to the initial deterministic collection $\cS$ including $\widehat{\cS}$. Indeed, if we replace the condition (\ref{condition_bonf}) by the condition $\sum_{S \in \widehat{\cS} } \sum_{i=1}^3\alpha_{i,S} \leq  \alpha$, then the size of the corresponding is not constrained anymore to be smaller than $\alpha$. This is due to the fact that we use the {\it same} data set to select $\widehat{\cS}\subset\cS$ and to perform the multiple testing procedure. As a simple example, consider any deterministic collection $\cS$ and the data-driven collection
\[\widehat{\cS}=\left\{\arg\min_{S\in \cS}\min_{i=V,1,2} \tilde{q}_{i,S}\right\}\ , \]
meaning that $\widehat{\cS}$ only contains the subset $S$  that minimizes the $p$-values of the parametric tests. Thus, computing $T_{\widehat{\cS}}^B$ for this particular collection $\widehat{\cS}$ is equivalent to performing a multiple testing procedure on $\cS$. 
\end{remark}

Although procedure $T^{B}_{\widehat{\cS}}$ is computationally and conceptually simple, the size of the corresponding test can be much lower than $\alpha$ because of three difficulties:
\begin{enumerate}
\item Independently from our problem, Bonferroni corrections are known to be too conservative under dependence of the test statistics.
\item As emphasized by Remark \ref{rem:bonf_corr}, whereas the Bonferroni correction needs to be based on the whole collection $\cS$, only the subsets $S\in\widehat{\cS}$ are considered. Provided we could afford the computational cost of testing all subsets within $\cS$, this loss cannot be compensated for if we use the Bonferroni correction.
\item As underlined in the above subsection, for computational reasons we do not consider the exact $p$-values of $F_{S,1}$ and $F_{S,2}$ but only upper bounds $\tilde{q}_{1,S}$ and $\tilde{q}_{2,S}$  of them. We therefore overestimate the type I error due to $F_{S,1}$ and $F_{S,2}$.
\end{enumerate}

In fact, the three aforementionned issues are addressed by the permutation  approach.

\subsubsection{Calibration by permutation (P).}  The collection of thresholds $\widehat{\alpha}^P = \{\alpha_{i,S},\
  S\in\widehat{\cS}\}$ is chosen such that each $\alpha_{i,S}$ remains
inversely proportional to $\binom{p}{|S|}$ in order to put all subset
sizes at equal footage. 
In other words, we choose a collection of thresholds of the form
\begin{equation} \label{condition_perm}
\alpha_{i,S} = \widehat{C}_i \binom{p}{|S|}^{-1} \ ,
\end{equation}
where $\widehat{C}_i$'s are  calibrated by permutation to control the type I error of the global test.

Given a permutation $\pi$ of the set $\{1,\ldots, n_1+n_2\}$,  one gets $\Y^{\pi}$ and  $\X^{\pi}$ by permuting the components of $\Y$ and the rows of $\X$. This allows us to get a new sample ($\Y^{\pi,(1)}$, $\Y^{\pi,(2)}$, $\X^{\pi,(1)}$, $\X^{\pi,(2)})$. Using this new sample, we compute a new collection $\widehat{\cS}^{\pi}$,  parametric statistics $(F_{S,i}^{\pi})_{i=V,1,2}$ and $p$-values $(\tilde{q}_{i,S})_{i=V,1,2}$.  Denote $\mathcal{P}$ the uniform distribution over the permutations of size $n_1+n_2$.

We define $\widehat{C}_{V}$   as the $\alpha/2$-quantiles  with respect to  $\mathcal{P}$ of
\begin{equation}\label{eq:defCV}
\min_{S\in\widehat{\cS}^{\pi}}\left[\tilde{q}_{V,S}\binom{p}{|S|}\right]\ .
\end{equation}
Similarly, $\widehat{C}_{1}= \widehat{C}_{2}$ are the $\alpha/2$-quantiles with respect to $\mathcal{P}$ of 
\begin{equation}\label{eq:defC1}
\min_{S\in\widehat{\cS}^{\pi}}\left[\left(\tilde{q}_{1,S}\wedge \tilde{q}_{2,S}\right)\binom{p}{|S|}\right]\ .
\end{equation}
In practice, the quantiles $\widehat{C}_{i}$ are estimated by sampling a large number $B$ of permutations. The permutation calibration procedure for the Lasso collection $\SLasso$ is summarized in Algorithm \ref{algo:permutations}.

\begin{algorithm}
\caption{Calibration by Permutation for $\SLasso$}\label{algo:permutations}
\begin{algorithmic}
\Require{Data $\mathbf{X}^{(1)}$,$\mathbf{X}^{(2)}$,$\mathbf{Y}^{(1)}$,$\mathbf{Y}^{(2)}$, maximum model dimension $D_{\max}$, number $B$ of permutations, desired level $\alpha$}
\For{b = 1,\dots, B}
	\State Draw $\pi$ a random permutation of $\{1,\ldots, n_1+n_2\}$
	\State ${{\bf X}^{(b)},\bf{Y}}^{(b)} \gets $  $\pi$-permutation of $({\bf X},\bf{Y})$
	\Procedure{LassoModelChoice}{$\mathbf{X}^{(1,b)}$,$\mathbf{X}^{(2,b)}$,$\mathbf{Y}^{(1,b)}$,$\mathbf{Y}^{(2,b)}$,$\cS_{\leq D_{\max}}$}
		\State Define $\SLasso^{(b)}$ (as in Algorithm \ref{algo:SLasso})
	\EndProcedure
	\Procedure{Test}{$\mathbf{X}^{(1,b)}$,$\mathbf{X}^{(2,b)}$,$\mathbf{Y}^{(1,b)}$,$\mathbf{Y}^{(2,b)}$,$\SLasso^{(b)}$}
		\For{each subset $S$ in $\SLasso^{(b)}$}
			\State Compute the $p$-values $\tilde{q}_{i,S}^{(b)}$ for $i=V,1,2$.
		\EndFor
		\State  $C_{V}^{(b)}\gets  \min_{S \in \SLasso^{(b)}} \tilde{q}_{V,S}^{(b)}\binom{p}{|S|}$
		\State  $C_{1}^{(b)}\gets \min_{S \in \SLasso^{(b)}} \left(\tilde{q}_{1,S}^{(b)}\wedge \tilde{q}_{2,S}^{(b)}\right)\binom{p}{|S|}$    
	\EndProcedure
\EndFor
\State Define $\widehat{C}_{V}$ as the $\alpha/2$-quantile of the $(C_V^{(1)},\dots,C_V^{(B)})$ distribution
\State Define $\widehat{C}_{1}=\widehat{C}_{2}$ as the $\alpha/2$-quantile of the $(C_1^{(1)},\dots,C_1^{(B)})$ distribution
\For{each subset $S$ in $\SLasso$, each $i=V,1,2$,}
\State  $\alpha_{i,S} \gets \widehat{C}_i \binom{p}{|S|}^{-1}$
\EndFor
\end{algorithmic}
\end{algorithm}

Given any data-driven collection $\widehat{\cS}$, denote by $T^{P}_{\widehat{\cS}}$ the multiple testing procedure calibrated by  the permutation method \eqref{condition_perm}.


\begin{proposition}[Size of $T^P_{\widehat{\cS}}$] \label{prop:sizeP}
The test function $T^{P}_{\widehat{\cS}}$ satisfies  
\[\alpha/2\leq \mathbb{P}_{\hyp_0}\left[T^{P}_{\widehat{\cS}}=1\right]\leq \alpha\ .\] 
\end{proposition}

\begin{remark}
Through the three constants $\widehat{C}_{V}$, $\widehat{C}_{1}$ and $\widehat{C}_{2}$ (Eq. (\ref{eq:defCV},\ref{eq:defC1})), the permutation approach corrects simultaneously for the losses mentioned earlier due to the Bonferroni correction, in particular the restriction to a data-driven class $\widehat{\cS}$ and the approximate $p$-values $\tilde{q}_{1,S}$ and $\tilde{q}_{2,S}$.
Yet, the level of $T^{P}_{\widehat{\cS}}$ is not exactly $\alpha$ because we treat separately the the statistics $F_{S,V}$ and $(F_{S,1},F_{S,2})$ and apply a Bonferroni correction. It would be possible to calibrate all the statistics simultaneously in order to constrain the size of the corresponding test to be exactly $\alpha$. However, this last approach would favor the statistic $F_{S,1}$ too much, because we would put on the same level the exact $p$-value $\tilde{q}_{V,S}$ and the upper bounds $\tilde{q}_{1,S}$ and $\tilde{q}_{2,S}$.
\end{remark}

\subsection{Interpretation tools}\label{interpretation}

\paragraph{Empirical $p$-value}
When using a calibration by permutations, one can derive an empirical
$p$-value $p^{empirical}$ to assess the global significance of the test.  In contrast
with model  and statistic specific  $p$-values $\tilde{q}_{i,S}$, this
$p$-value provides a nominally accurate estimation of the type-I
error rate associated with the global multiple testing procedure, every
model  in the  collection and test statistic being  considered. It can be
directly compared  to the desired  level $\alpha$ to decide about the
rejection or not of the global null hypothesis.

This empirical p-value is obtained as the fraction of the
permuted values of the statistic that are less than the observed test
statistic. Keeping the notation of Algorithm \ref{algo:permutations},
the empirical p-value for the variance and coefficient parts are given
respectively by : 
\begin{eqnarray*}
p^{empirical}_{V} = &&\frac{1}{B} \sum_{b=1}^B \mathds{1} \left[
  C_{V}^{(b)} <  \min_{S \in \SLasso} \tilde{q}_{V,S} \binom{p}{|S|}
\right]\ , \\
p^{empirical}_{1-2} = &&\frac{1}{B} \sum_{b=1}^B \mathds{1}\left[ C_{1}^{(b)} <
  \min_{S \in \SLasso} \left( \tilde{q}_{1,S} \wedge \tilde{q}_{2,S} \right)\binom{p}{|S|} \right]\ . 
\end{eqnarray*}

The empirical p-value for the global test is then given by the following
equation.
\begin{equation}
\label{eq:neighborhood_pvalue}
p^{empirical} = 2 \min(p^{empirical}_{V},p^{empirical}_{1-2}).
\end{equation}

\paragraph{Rejected model} Moreover, one can keep track of
the model responsible for the rejection, unveiling sensible information on
which particular coefficients most likely differ between samples. The
rejected models for the variance and coefficient parts are given
respectively by : 
\begin{eqnarray*}
S^R_{V} =  &&\arg \min_{S \in \SLasso} \tilde{q}_{V,S}  \binom{p}{|S|}\\
S^R_{1-2} = && \arg \min_{S \in \SLasso}  \left( \tilde{q}_{1,S} \wedge \tilde{q}_{2,S} \right)\binom{p}{|S|} 
\end{eqnarray*}
We define the rejected model $S^R$ as model $S^R_{V}$ or
$S^R_{1-2}$ according to the smallest empirical p-value
$p^{empirical}_{V}$ or $p^{empirical}_{1-2}$. 
\label{eq:rejected_model}

\section{Power and Adaptation to Sparsity}\label{sec:power1}
Let us fix some number $\delta\in(0,1)$. The objective is to investigate the set of parameters $(\beta^{(1)},\sigma^{(1)},\beta^{(2)},\sigma^{(2)})$ that enforce the power of the test to exceed $1-\delta$.  We focus here on the Bonferroni calibration (B) procedure because the analysis is easier. Section  \ref{sec:num} will illustrate that the permutation calibration (P) outperforms the Bonferroni calibration (B) in practice. In the sequel, $A \lesssim B$ (resp. $A\gtrsim B$) means that for some  positive constant $L(\alpha,\delta)$ that only depends on $\alpha$ and $\delta$, $A\leq L(\alpha,\delta)B$ (resp. $A\geq L(\alpha,\delta)B$).

We first define the symmetrized Kullback-Leibler divergence as a way to measure the discrepancies between 
$(\beta^{(1)},\sigma^{(1)})$ and $(\beta^{(2)},\sigma^{(2)})$. Then, we consider tests with deterministic collections
 in Sections \ref{sec:power_Sk}--\ref{sec:power_arbitrary}. We prove that the corresponding tests are minimax adaptive to the sparsity of
  the parameters or to the sparsity of the difference $\beta^{(1)}-\beta^{(2)}$. Sections \ref{sec:power_Slasso}--\ref{sec:power_Slasso2} are
   devoted to the analysis $T_{\SLasso}^B$. Under stronger assumptions on the population covariances than for deterministic collections, we prove that the performances of $T_{\SLasso}^B$ are nearly optimal.

\subsection{Symmetrized Kullback-Leibler divergence}

 Intuitively, the test $T_{\cS}^{B}$ should reject $\hyp_0$ with large probability when $(\beta^{(1)},\sigma^{(1)})$ is far from $(\beta^{(2)},\sigma^{(2)})$ in some sense. A classical way of measuring the divergence between two distributions is the Kullback-Leibler discrepancy. In the sequel, we note 
 $\mathcal{K}\left[\mathbb{P}_{Y^{(1)}|X};\mathbb{P}_{Y^{(2)}|X}\right]$
the Kullback discrepancy  between the conditional distribution of $Y^{(1)}$ given $X^{(1)}=X$ and  conditional distribution of $Y^{(2)}$ given $X^{(2)}=X$. Then, we denote $\mathcal{K}_{1}$ the expectation of this Kullback divergence when  $X\sim \mathcal{N}(0_p,\Sigma^{(1)})$. Exchanging the roles of $\Sigma^{(1)}$ and $\Sigma^{(2)}$, we also define $\mathcal{K}_{2}$:
\begin{equation*}
 \mathcal{K}_{1} := 
\mathbb{E}_{X^{(1)}}\left\{\mathcal{K}\left[\mathbb{P}_{Y^{(1)}|X};\mathbb{P}_{Y^{(2)}|X}\right]\right\}\ , \quad
\mathcal{K}_{2} := 
\mathbb{E}_{X^{(2)}}\left\{\mathcal{K}\left[\mathbb{P}_{Y^{(2)}|X};\mathbb{P}_{Y^{(1)}|X}\right]\right\}\ .
\end{equation*}
The sum $\mathcal{K}_{1}+ \mathcal{K}_{2}$ forms a semidistance with respect to $(\beta^{(1)},\sigma^{(1)})$ and $(\beta^{(2)},\sigma^{(2)})$ as proved by the following decomposition
\begin{equation*}
 2\left(\mathcal{K}_{1}+ \mathcal{K}_{2}\right)= \left(\frac{\sigma^{(1)}}{\sigma^{(2)}}\right)^2+\left(\frac{\sigma^{(2)}}{\sigma^{(1)}}\right)^2-2 + \frac{\|\beta^{(2)}-\beta^{(1)}\|_{\Sigma^{(2)}}^2}{(\sigma^{(1)})^2} + \frac{\|\beta^{(2)}-\beta^{(1)}\|_{\Sigma^{(1)}}^2}{(\sigma^{(2)})^2}.
\end{equation*} 
When $\Sigma^{(1)}\neq \Sigma^{(2)}$, we quantify the discrepancy between these covariance matrices by
\begin{equation*}
\varphi_{\Sigma^{(1)},\Sigma^{(2)}}:= \varphi_{\max}\left\{\sqrt{\Sigma^{(2)}}(\Sigma^{(1)})^{-1}\sqrt{\Sigma^{(2)}}+\sqrt{\Sigma^{(1)}}(\Sigma^{(2)})^{-1}\sqrt{\Sigma^{(1)}}\right\}\ .
\end{equation*}	
Observe that the quantity $\varphi_{\Sigma^{(1)},\Sigma^{(2)}}$ can be considered as a constant if we assume that the smallest and largest eigenvalues of $\Sigma^{(i)}$ are bounded away from zero and infinity.

\subsection{Power of $T_{\cS_{\leq k}}^B$}\label{sec:power_Sk}

First, we control the power of $T_{\cS}^B$ for a deterministic collection $\cS=\cS_{\leq k}$ (with some $k \leq (n_1\wedge n_2)/2$) and the Bonferroni calibration weights $\widehat{\alpha}_{i,S}$ as in (\ref{eq:condition_bonf_classique})).  For any $\beta\in\mathbb{R}^p$,  $|\beta|_0$ refers to the size of its support and $|\beta|$ stands for the vector  $(|\beta_i|), i=1,\ldots, p$. We consider the two following assumptions
\[{\bf A.1}: \hspace{5cm} \log(1/(\alpha\delta))\lesssim n_1\wedge n_2\ .\hspace{5cm}\]
\[{\bf A.2}: \hspace{2cm}|\beta^{(1)}|_0+ |\beta^{(2)}|_0	\lesssim k\wedge \left(\frac{n_1\wedge n_2}{\log(p)}\right)\ , \quad \quad  \log(p)\leq n_1\wedge n_2\ .\hspace{2cm}\]

\begin{remark} Condition {\bf A.1} requires that the type I and type II errors under consideration are not exponentially smaller than the sample size. Condition {\bf A.2} tells us that the number of non-zero components of $\beta^{(1)}$ and $\beta^{(2)}$ has to be smaller than $(n_1\wedge n_2)/\log(p)$. This requirement has been shown~\cite{Vminimax} to be minimal to obtain fast rates of testing of the form (\ref{eq:rates_testing}) in the specific case $\beta^{(2)}=0$, $\sigma^{(1)}=\sigma^{(2)}$ and $n_2=\infty$.
\end{remark}    

\begin{thrm}[Power of $T_{\cS_{\leq k}}^B$]\label{cor:section_complete}
Assuming that {\bf A.1} and {\bf A.2} hold,  $\mathbb{P}[T_{\cS_{\leq k}}^B=1]\geq 1-\delta$ as long as 
\begin{equation}\label{eq:rates_testing}
\mathcal{K}_1+\mathcal{K}_2\gtrsim \varphi_{\Sigma^{(1)},\Sigma^{(2)}} 
\frac{\left\{|\beta^{(1)}|_0\vee |\beta^{(2)}|_0\vee 1\right\}\log\left(p\right)+\log\left(\frac{1}{\alpha\delta}\right)}{n_1\wedge n_2}\ .
\end{equation}
If we further assume that $\Sigma^{(1)}= \Sigma^{(2)}:=\Sigma$, then $\mathbb{P}[T_{\cS_{\leq k}}^B=1]\geq 1-\delta$ as long as 
\begin{equation}\label{eq:rates_testing2}
\frac{\|\beta^{(1)}-\beta^{(2)}\|^2_{\Sigma}}{\var[Y^{(1)}]\wedge \var[Y^{(2)}]} \gtrsim
\frac{|\beta^{(1)}-\beta^{(2)}|_0\log\left(p\right)+\log\left(\frac{1}{\alpha\delta}\right)}{n_1\wedge n_2}\ .
\end{equation}
\end{thrm}

\begin{remark} The condition $\Sigma^{(1)}= \Sigma^{(2)}$ is not necessary to control the power of $T_{\cS_{\leq k}}^B$ in terms of $|\beta^{(1)}-\beta^{(2)}|_0$ as in (\ref{eq:rates_testing2}). However, the expression (\ref{eq:rates_testing2}) would become far more involved.
\end{remark}

\begin{remark} Before assessing the optimality of Theorem \ref{cor:section_complete}, let us briefly compare the two rates of detection (\ref{eq:rates_testing}) and (\ref{eq:rates_testing2}). 
According to (\ref{eq:rates_testing}), $T_{\cS_{\leq k}}^B$ is powerful
as soon as the symmetrized Kullback distance is large compared to $\{|\beta^{(1)}|_0\vee |\beta^{(2)}|_0\}\log\left(p\right)/(n_1\wedge n_2)$. In contrast, (\ref{eq:rates_testing2}) tells us that  $T_{\cS_{\leq k}}^B$ is powerful when  $\|\beta^{(1)}-\beta^{(2)}\|^2_{\Sigma}/(\var[Y^{(1)}]\wedge \var[Y^{(2)}])$ is large compared to the sparsity of the difference: $|\beta^{(1)}-\beta^{(2)}|_0\log\left(p\right)/(n_1\wedge n_2)$.

When $\beta^{(1)}$ and $\beta^{(2)}$ have many non-zero coefficients in common, $|\beta^{(1)}-\beta^{(2)}|_0$ is much smaller than $|\beta^{(1)}|_0\vee |\beta^{(2)}|_0$. Furthermore, the left-hand side of \eqref{eq:rates_testing2} is of the same order as $\mathcal{K}_1+\mathcal{K}_2$  when $\Sigma^{(1)}=\Sigma^{(2)}$, $\sigma^{(1)}=\sigma^{(2)}$ and $\|\beta^{(i)}\|_{\Sigma}/\sigma^{(i)}\lesssim 1$ for $i=1,2$, that is when the conditional variances are equal and when the signals $\|\beta^{(i)}\|_{\Sigma}$ are at most at the same order as the noises levels $\sigma^{(i)}$.
In such a case, \eqref{eq:rates_testing2} outperforms \eqref{eq:rates_testing} and only the sparsity of the difference $\beta^{(1)}-\beta^{(2)}$ plays a role in the detection rates. Below, we prove that \eqref{eq:rates_testing} and \eqref{eq:rates_testing2} are both 
optimal from a minimax point of view but on different sets.
\end{remark}

\begin{proposition}[Minimax lower bounds]\label{prte:minoration_minimax}
Assume that   $p\geq 5$, $\Sigma^{(1)}=\Sigma^{(2)}=I_p$, fix some $\gamma>0$, and fix $(\alpha,\delta)$ such that $\alpha+\delta<53\%$.
There exist two constants $L(\alpha,\delta,\gamma)$ and $L'(\alpha,\delta,\gamma)$ such that the following holds. 

\begin{itemize}
 \item 
For all $1\leq s\leq p^{1/2-\gamma}$  
{\bf no} level-$\alpha$ test has a power larger than $1-\delta$ simultaneously over all $s$-sparse vectors $(\beta^{(1)},\beta^{(2)})$ satisfying ${\bf A.2}$ and
\begin{equation}\label{eq:rates_testing_minoration}
\mathcal{K}_1+\mathcal{K}_2 \geq L(\alpha,\delta,\gamma)
\frac{s}{n_1\wedge n_2}\log\left(p\right)\ .
\end{equation}
\item For all $1\leq s\leq p^{1/2-\gamma}$, 
{\bf no} level-$\alpha$ test has a power larger than $1-\delta$ simultaneously over all sparse vectors $(\beta^{(1)},\beta^{(2)})$ satisfying ${\bf A.2}$, $|\beta^{(1)}-\beta^{(2)}|_0\leq s$ and
\begin{equation}\label{eq:rates_testing_minoration2}
\frac{\|\beta^{(1)}-\beta^{(2)}\|^2_{I_p}}{\var[Y^{(1)}]\wedge \var[Y^{(2)}]} \geq L'(\alpha,\delta,\gamma)
\frac{s}{n_1\wedge n_2}\log\left(p\right)\ .
\end{equation}

\end{itemize}

\end{proposition}

The proof (in Section \ref{sec:proofs}) is a straightforward application of minimax lower bounds obtained for the one-sample testing problem~\cite{2010_AS_Verzelen,2011_AS_Arias-Castro}.

\begin{remark} Equation (\ref{eq:rates_testing}) together with \eqref{eq:rates_testing_minoration} tell us that $T_{\cS_{\leq k}}^B$ simultaneously achieves (up to a constant) the optimal rates of detection over $s$-sparse vectors $\beta^{(1)}$ and $\beta^{(2)}$ for all 
\[s\lesssim k \wedge p^{1/2-\gamma}\wedge \frac{n_1\wedge n_2}{\log(p)}\ ,\]
for any $\gamma>0$. Nevertheless, we only managed to prove the minimax lower bound for $\Sigma^{(1)}=\Sigma^{(2)}=I_p$, implying that, even though the detection rate \eqref{eq:rates_testing} is unimprovable uniformly over all $(\Sigma^{(1)},\Sigma^{(2)})$, some improvement is perhaps possible for specific covariance matrices. Up to our knowledge, there exist no such results of adaptation to the population covariance of the design even in the one sample problem. 
\end{remark}

\begin{remark} Equation \eqref{eq:rates_testing2} together with \eqref{eq:rates_testing_minoration2} tells us that $T_{\cS_{\leq k}}^B$ simultaneously achieves (up to a constant)  the
optimal rates of detection over $s$-sparse differences $\beta^{(1)}-\beta^{(2)}$ satisfying $\frac{\|\beta^{(1)}\|_{\Sigma}}{\sigma^{(1)}}\vee \frac{\|\beta^{(2)}\|_{\Sigma}}{\sigma^{(2)}}\leq 1$  for all $s\lesssim k \wedge p^{1/2-\gamma}\wedge \frac{n_1\wedge n_2}{\log(p)}$.
\end{remark}

\begin{remark}[Informal justification of the introduction of the collection $\SLasso$] If we look at the proof of Theorem \ref{cor:section_complete}, we observe that the power \eqref{eq:rates_testing} is achieved by the statistics $(F_{S_{\vee},V},F_{S_{\vee},1},F_{S_{\vee},2})$ where $S_\vee$ is the union of the support of $\beta^{(1)}$ and $\beta^{(2)}$. In contrast, 
\eqref{eq:rates_testing2} is achieved by the statistics $(F_{S_{\Delta},V},F_{S_{\Delta},1},F_{S_{\Delta},2})$ where $S_\Delta$ is the support of $\beta^{(1)}-\beta^{(2)}$. Intuitively, the idea underlying the  collection $\widehat{\cS}_L^{(1)}$ in the definition \eqref{eq:definition_family_lasso} of $\SLasso$ is to estimate $S_{\vee}$, while the idea underlying the collection $\widehat{\cS}_L^{(2)}$ is to estimate $S_{\Delta}$.
\end{remark}

\subsection{Power of  $T_{\cS}^B$ for any deterministic $\cS$}\label{sec:power_arbitrary}

 Theorem \ref{thrm_puissance} belows extends Theorem
\ref{cor:section_complete} from deterministic collections of the form $\cS_{\leq k}$ to any
deterministic collection $\cS$, unveiling a bias/variance-like
trade-off linked to the cardinality of subsets $S$ of collection $\cS$.  
To do so, we need to consider the Kullback discrepancy between the conditional distribution of $Y^{(1)}$ given  $X_S^{(1)}=X_S$ and the conditional distribution of $Y^{(2)}$ given $X_S^{(2)}=X_S$, which we denote $\mathcal{K}\left[\mathbb{P}_{Y^{(1)}|X_S};\mathbb{P}_{Y^{(2)}|X_S}\right]$. For short, we respectively note $\K_1(S)$ and $\K_2(S)$ \begin{eqnarray*}
 \mathcal{K}_1(S) &:= &
\mathbb{E}_{X_S^{(1)}}\left\{\mathcal{K}\left[\mathbb{P}_{Y^{(1)}|X_S};\mathbb{P}_{Y^{(2)}|X_S}\right]\right\}\ ,\\
\mathcal{K}_2(S) &:= &
\mathbb{E}_{X_S^{(2)}}\left\{\mathcal{K}\left[\mathbb{P}_{Y^{(2)}|X_S};\mathbb{P}_{Y^{(1)}|X_S}\right]\right\}\ .
\end{eqnarray*}
Intuitively, $\mathcal{K}_1(S)+\mathcal{K}_2(S)$ corresponds to some distance between the regression of $Y^{(1)}$ given $X_S^{(1)}$ and of $Y^{(2)}$ given $X_S^{(2)}$. Noting $\Sigma_S^{(1)}$ (resp. $\Sigma_S^{(2)}$) the restriction of $\Sigma^{(1)}$ (resp. $\Sigma^{(2)}$) to indices in $S$, we define
\begin{equation}\label{eq:definition_varphiS}
\varphi_{S}:= \varphi_{\max}\left\{\sqrt{\Sigma_S^{(2)}}(\Sigma_S^{(1)})^{-1}\sqrt{\Sigma_S^{(2)}}+\sqrt{\Sigma_S^{(1)}}(\Sigma_S^{(2)})^{-1}\sqrt{\Sigma_S^{(1)}}\right\}\ .
\end{equation}

\begin{thrm}[Power of $T_{\cS}^B$ for any deterministic $\cS$]\label{thrm_puissance} For any $S\in\cS$, we note $\alpha_S=\min_{i=V,1,2}\alpha_{i,S}$.
The power of $T_{\cS}^B$ is larger than $1-\delta$ as long as there exists $S\in\cS$ such that $|S|\lesssim n_1\wedge n_2$ and 
\begin{eqnarray}\label{eq:H1}
1+ \log[1/(\delta \alpha_S)] &\lesssim&n_1\wedge n_2\ ,
\end{eqnarray}
and 
\begin{equation} 
\mathcal{K}_1(S)+\mathcal{K}_2(S)\gtrsim \varphi_S\left(\frac{1}{n_1}+ \frac{1}{n_2}\right)\left[|S|+ \log\left(\frac{1}{\alpha_S\delta}\right)\right]\ .
\label{inegalite_puissance_hm}
\end{equation} 
\end{thrm}

\begin{remark}
Let us note $\Delta(S)$ the right hand side of
\eqref{inegalite_puissance_hm}. According to Theorem
\ref{thrm_puissance},  The term $\Delta(S)$ plays the role of a variance
term and therefore increases with the cardinality of $S$. Furthermore, the term $\mathcal{K}_1-\mathcal{K}_1(S)+\mathcal{K}_2-\mathcal{K}_2(S)$ plays the role of a bias. Let us note $\mathcal{S}^*$ the subcollection of $\mathcal{S}$ made of sets $S$ satisfying \eqref{eq:H1}.
According to theorem \ref{thrm_puissance}, $T_{\cS}^B$ is powerful as
long as $\mathcal{K}_1+\mathcal{K}_2$ is larger (up to constants) to
\begin{equation}\label{eq:tradeoff}
\inf_{S\in\mathcal{S}^*}\left\{\mathcal{K}_1-\mathcal{K}_1(S)+\mathcal{K}_2-\mathcal{K}_2(S)\right\}+ \Delta(S)
\end{equation}
Such a result is comparable to oracle inequalities obtained in estimation since the test $T_{\cS}^B$ is powerful when the Kullback loss $\mathcal{K}_1+\mathcal{K}_2$
is larger than the  trade-off \eqref{eq:tradeoff} between a bias-like term and a variance-like term without requiring the knowledge of this trade-off in advance. We refer to \cite{2003_AS_Baraud} for a thorough comparison between oracle inequalities in model selection and second type error terms of this form.
\end{remark}

\subsection{Power of $T_{\SLasso}^B$}\label{sec:power_Slasso}

For the sake of simplicity, we restrict in this subsection to the case
$n_1=n_2:=n$, more general results being postponed to the next subsection. 
The test $T_{\cS_{\leq n/2}}^{B}$ is computationally expensive (non polynomial with respect to $p$). The collection $\SLasso$  has been introduced to fix this burden. We consider $T_{\SLasso}^B$ with  the prescribed Bonferroni calibration weights $\widehat{\alpha}_{i,S}$ (as in (\ref{eq:condition_bonf_classique}) with $k$ replaced by $\lfloor (n_1\wedge n_2)/2\rfloor$.
In the statements below, $\psi^{(1)}_{\Sigma^{(1)},\Sigma^{(2)}}$,
$\psi^{(2)}_{\Sigma^{(1)},\Sigma^{(2)}}$,\dots\ refer to  positive quantities that only depend on the largest and the smallest eigenvalues of $\Sigma^{(1)}$ and $\Sigma^{(2)}$. Consider the additional assumptions
\[ {\bf A.3}: \hspace{3cm}|\beta^{(1)}|_0\vee |\beta^{(2)}|_0	\lesssim \psi^{(1)}_{\Sigma^{(1)},\Sigma^{(2)}} \frac{n}{\log(p)}\ . \hspace{4.5cm}\]
\[ {\bf A.4}: \hspace{3cm}|\beta^{(1)}|_0\vee |\beta^{(2)}|_0	\lesssim \psi^{(2)}_{\Sigma^{(1)},\Sigma^{(2)}} \sqrt{\frac{n}{\log(p)}}\ . \hspace{4cm}\]

\begin{thrm}\label{thrm:lasso}
Assuming that {\bf A.1} and {\bf A.3} hold, we have  $\mathbb{P}[T_{\SLasso}^B=1]\geq 1-\delta$ as long as 
\begin{equation}\label{eq:rates_test_lasso}
 \mathcal{K}_1+\mathcal{K}_2\gtrsim \psi^{(3)}_{\Sigma^{(1)},\Sigma^{(2)}} \frac{\left\{|\beta^{(1)}|_0\vee |\beta^{(2)}|_0\vee 1\right\}\log\left(p\right)+\log\left(\frac{1}{\alpha\delta}\right)}{n}\ .
\end{equation}
If $\Sigma^{(1)}=\Sigma^{(2)}=\Sigma$ and if ${\bf A.1}$ and ${\bf A.4}$ hold, then $\mathbb{P}[T_{\SLasso}^B=1]\geq 1-\delta$ as long as 
\begin{equation}\label{eq:rates_test_lasso2}
\frac{\|\beta^{(1)}-\beta^{(2)}\|^2_{\Sigma}}{\var[Y^{(1)}]\wedge \var[Y^{(2)}]} \gtrsim \psi^{(4)}_{\Sigma,\Sigma}
\frac{|\beta^{(1)}-\beta^{(2)}|_0\log\left(p\right)+\log\left(\frac{1}{\alpha\delta}\right)}{n}\ .
\end{equation}

\end{thrm}

\begin{remark} 
 The rates of detection \eqref{eq:rates_test_lasso} and the sparsity condition {\bf A.3} are analogous to \eqref{eq:rates_testing} and Condition {\bf A.2} in Theorem  \ref{cor:section_complete} for $T_{\cS_{\leq (n_1\wedge n_2)/2}}^{B}$. The second result \eqref{eq:rates_test_lasso2} is also similar to \eqref{eq:rates_testing2}. As a consequence, $T_{\SLasso}^B$ is minimax adaptive to the sparsity of $(\beta^{(1)},\beta^{(2)})$ and of $\beta^{(1)}-\beta^{(2)}$.
\end{remark}

\begin{remark}
Dependencies of {\bf A.3}, {\bf A.4},  \eqref{eq:rates_test_lasso} and \eqref{eq:rates_test_lasso2} on $\Sigma^{(1)}$ and $\Sigma^{(2)}$ are unavoidable because the collection $\SLasso$ is based on the Lasso estimator which require design assumptions to work well \cite{2007_AS_Candes}. Nevertheless, one can improve all these dependencies  using restricted eigenvalues instead of largest eigenvalues. This and other extensions are considered in next subsection. 
\end{remark}

\subsection{Sharper analysis of  $T^B_{\SLasso}$}\label{sec:power_Slasso2}

Given a matrix $\X$, an integer $k$, and a number $M$, one respectively defines  the largest
and smallest eigenvalues of order $k$,  the compatibility constants
$\kappa[M,k,\X ]$ and $\eta[M,k,\X]$ (see \cite{2009_EJS_Geer})  by
\begin{eqnarray}\nonumber
\Phi_{k,+}(\X)&=&\sup_{\theta, 1\leq |\theta|_0\leq k}\frac{\|\X\theta\|^2}{\|\theta\|^2}\ , \quad  \Phi_{k,-}(\X)=\inf_{\theta, 1\leq |\theta|_0\leq k}\frac{\|\X\theta\|^2}{\|\theta\|^2}\ ,\\ \nonumber
 \kappa[M,k,\X]&=& \min_{T,\theta:\ |T|\leq k,\ \theta\in \mathcal{C}(M,T)} \left\{\frac{\|\X \theta\|}{\|\theta\|}\right\}\ , \\ \label{eq:definition_eta}
\eta[M,k,\X] & =& \min_{T,\theta:\ |T|\leq k,\ \theta\in \mathcal{C}(M,T)} \left\{\sqrt{k}\frac{\|\X \theta\|}{|\theta|_1}\right\}\ ,
\end{eqnarray}
where $\mathcal{C}(M,T)=\{\theta:\ |\theta_{T^c}|_1< M |\theta_T|_1\}$. 
Given an integer $k$, define 
\begin{eqnarray*}
\gamma_{\Sigma^{(1)},\Sigma^{(2)},k}&:=&\frac{\bigwedge_{i=1,2} \kappa^2\left[6,k_* ,\sqrt{\Sigma^{(i)}}\right] }{\bigvee_{i=1,2}\Phi_{k_*,+}(\sqrt{\Sigma^{(i)}})}\ , \\
\gamma'_{\Sigma^{(1)},\Sigma^{(2)},k}&:= &\frac{\bigvee_{i=1,2}\Phi^2_{k,+}(\sqrt{\Sigma^{(i)}})}{\bigwedge_{i=1,2}\Phi_{k,-}(\sqrt{\Sigma^{(i)}}) \bigwedge_{i=1,2}\kappa^2[6,k,\sqrt{\Sigma^{(i)}]} }\ ,
\end{eqnarray*}
that measure the closeness to orthogonality of $\Sigma^{(1)}$ and $\Sigma^{(2)}$. 
Theorem \ref{thrm:lasso} is straightforward consequence of the two following results.

\begin{proposition}\label{prte:lasso_precis}
There exist four positive constants $L^*$, $L^*_1$, $L^*_2$, and $L^*_3$ such that following holds.
Define $k_*$ as the largest integer that satisfies
\begin{equation}\label{eq:definition_k_*}
 (k_*+1)\log(p)\leq  L^*(n_1\wedge n_2)\ ,
\end{equation}
and assume that 
\begin{equation}\label{eq:lasso_hypothese_alpha_delta}
 1+\log\left[1/(\alpha\delta)\right]< L^*_1(n_1\wedge n_2)\ .
\end{equation}
 The hypothesis  $\hyp_0$ is rejected by $T_{\SLasso}^B$ with probability larger than $1-\delta$ for  any $(\beta^{(1)},\beta^{(2)})$ satisfying
\begin{equation}
 |\beta^{(1)}|_0+|\beta^{(2)}|_0 \leq L^*_2 \gamma_{\Sigma^{(1)},\Sigma^{(2)},k} k_* \left(\frac{n_1}{n_2}\wedge \frac{n_2}{n_1}\right)
\ .\label{eq:hypothese_sparsite} 
\end{equation}
 and  
\begin{eqnarray*}
 \mathcal{K}_1+\mathcal{K}_2\geq L^*_3 \gamma'_{\Sigma^{(1)},\Sigma^{(2)},k_*} \frac{\left(|\beta^{(1)}|_0\vee |\beta^{(2)}|_0\vee 1\right)\log(p)+ \log\{1/(\alpha\delta)\}}{n_1\wedge n_2}\left(\frac{n_1}{n_2}\vee \frac{n_2}{n_1}\right)\ .
\end{eqnarray*}
\end{proposition}

This proposition tells us that $T_{\SLasso}^B$ behaves nearly as well as what has been obtained in \eqref{eq:rates_testing} for $T^B_{\cS_{\leq (n_1\wedge n_2)/2}}$, at least when $n_1$ and $n_2$ are of the same order.

In the next proposition, we assume that $\Sigma^{(1)}=\Sigma^{(2)}:=\Sigma$. Given an integer $k$, define 
\[ \tilde{\gamma}_{\Sigma, k} := \frac{\kappa[6,k,\sqrt{\Sigma}] \Phi^{1/2}_{k,-}(\sqrt{\Sigma})} {\Phi_{1,+}(\sqrt{\Sigma})}\ ,\quad   \tilde{\gamma}^{(2)}_{\Sigma, k} :=\frac{ \kappa^2\left[6,k ,\sqrt{\Sigma}\right] }{\Phi_{k,+}(\sqrt{\Sigma})}\ , \quad \tilde{\gamma}^{(3)}_{\Sigma, k} :=\frac{\Phi^2_{1,+}(\sqrt{\Sigma})}{ \kappa^2[6,k,\sqrt{\Sigma}]}\ .  \ \]

\begin{proposition}\label{prte:lasso_precis2}
Let us assume that $\Sigma^{(1)}=\Sigma^{(2)}:=\Sigma$. There exist five positive constants $L^*$, $\tilde{L}^{*}$,  $L^*_1$, $L^*_2$, and $L^*_3$ such that following holds.
 Define $k_*$ and $\tilde{k}_*$ as the largest positive integers that satisfy
\begin{eqnarray}\nonumber 
 (k_*+1)\log(p)&\leq&  L^*(n_1\wedge n_2)\ , \\
\label{eq:definition_k'*}
 \tilde{k}_* &\leq &\tilde{L}^{*} \tilde{\gamma}_{\Sigma, k_*} \left[\frac{n_1\wedge n_2}{|n_1-n_2|} \wedge  \sqrt{\frac{n_1\wedge n_2}{\log(p)}}\right] \ ,
\end{eqnarray}
with the convention $x/0=\infty$. Assume that 
\begin{equation*}
 1+\log\left[1/(\alpha\delta)\right]< L^*_1(n_1\wedge n_2)\ .
\end{equation*}
 The hypothesis  $\hyp_0$ is rejected by $T_{\SLasso}^B$ with probability larger than $1-\delta$ for  any $(\beta^{(1)},\beta^{(2)})$ satisfying
\begin{equation}
 |\beta^{(1)}|_0+|\beta^{(2)}|_0 \leq L^*_2 \tilde{\gamma}^{(2)}_{\Sigma, \tilde{k}_*} \tilde{k}_* 
\ .\label{eq:hypothese_sparsite2} 
\end{equation}
 and  
\begin{eqnarray*}
 \frac{\|\beta^{(1)}-\beta^{(2)}\|_{\Sigma}^2}{\var(Y^{(1)})\wedge \var(Y^{(2)})}\geq L^*_3 \tilde{\gamma}^{(3)}_{\Sigma, k_*} \left[\left(|\beta^{(1)}-\beta^{(2)}|_0\vee 1\right)\log(p)+ \log\{1/(\alpha\delta)\}\right]\ .
\end{eqnarray*}
\end{proposition}

\begin{remark}
The definition \eqref{eq:definition_k'*} of $\tilde{k}_*$ together with Condition \eqref{eq:hypothese_sparsite2} restrict the number of non-zero components $|\beta^{(1)}|_0+|\beta^{(2)}|_0$ to be  small  in front of $(n_1\wedge n_2)/|n_1-n_2|$. This technical assumption enforces the design matrix in the reparametrized model \eqref{eq:reparam} to be almost block-diagonal and allows us to control efficiently the Lasso estimator $\widehat{\theta}^{(2)}_{\lambda}$ of $\theta_*^{(2)}$ for some $\lambda>0$ (see the proof in Section \ref{sec:proofs} for further details). Still, this is not clear to what extent this assumption is necessary.
\end{remark}

\section{Numerical Experiments}\label{sec:num}

This section evaluates the performances of the suggested test
statistics along with aforementioned test collections and calibrations
on simulated linear regression datasets. 


\subsection{Synthetic Linear Regression Data}


In order to calibrate the difficulty of the testing task, we simulate
our data according to the rare and weak parametrization adopted in \cite{2011_AS_Arias-Castro}. We choose a large but
still reasonable number of variables $p= 200$, and restrict ourselves
to cases where the number of observations $n=n_1=n_2$ in each sample remains smaller than
$p$. The sparsity of sample-specific coefficients $\beta^{(1)}$ and $\beta^{(2)}$ is parametrized by the
number of non zero common coefficients $p^{1-\eta}$ and the number of
non zero coefficients $p^{1-\eta_2}$ which are specific to $\beta^{(2)}$. The magnitude $\mu_r$ of all  non zero
coefficients is set to a common value of $\sqrt{2 r \log{p}}$, where
we let the magnitude parameter range from $r=0$ to $r=0.5$:
\[
\begin{array}{rccc}
\beta^{(1)} =& (\mu_r\ \mu_r \dots \mu_r & 0  \dots  0 &0 \dots 0)\\ 
\beta^{(2)} =& \underbrace{(\mu_r\ \mu_r \dots \mu_r}_{p^{1-\eta}
  \text{ common coefficients}}&
\underbrace{\mu_r \dots  \mu_r}_{p^{1-\eta_2} \text{ sample-2-specific coefficients }} & 0 \dots  0 )\\ 
\end{array}
\]

We consider three sample sizes 
$n=25,\ 50,\ 100$, and generate two sub-samples of equal size
$n_1=n_2=n$ according to the following sample specific linear
regression models:
\[
\left\{
\begin{array}{rcl}
\mathbf{Y}^{(1)} &=& \X^{(1)} \beta^{(1)} + \varepsilon^{(1)},\\
\mathbf{Y}^{(2)} &=& \X^{(2)} \beta^{(2)} + \varepsilon^{(2)}.
\end{array}
\right.
\]

Design matrices $\X^{(1)}$ and
$\X^{(2)}$ are generated by multivariate Gaussian distributions,
$\X_i^{(j)} \sim \mathcal{N}(0,\Sigma^{(j)})$ with varying choices of
$\Sigma^{(j)}$, as detailed below. Noise components $\varepsilon_i^{(1)}$ and
$\varepsilon_i^{(2)}$ are generated independantly from $\X^{(1)}$ and
$\X^{(2)}$ according to a standard centered Gaussian distribution. 

The next two paragraphs detail the different design scenarios under
study as well as test statistics, collections and calibrations in competition. Each experiment is repeated 1000 times.

\paragraph{Design Scenarios Under Study.}

\subparagraph{Sparsity Patterns.}
We study six different sparsity patterns as summarized in Table
\ref{tab:scen}. The first two are meant to validate type I error control. The last
four allow us to compare the performances of the various test
statistics, collections and calibrations 
under different sparsity levels and proportions of shared
coefficients. In all cases, the choices of sparsity parameters $\eta$ and
$\eta_2$  lead to strong to very strong levels of sparsity. 
 The last column of Table \ref{tab:scen} illustrates the signal sparsity patterns of
$\beta^{(1)}$ and $\beta^{(2)}$ associated with each scenario. In scenarios 1 and 2,
sample-specific signals share little, if not none, non zero
coefficient. In scenarios 3 and 4, sample-specific coefficients
show some overlap. Scenario~4 is the most difficult one since the number
of sample-2-specific coefficients is much smaller than the number
of common non zero coefficients: the sparsity of the difference
between $\beta^{(1)}$ and $\beta^{(2)}$ is
much smaller than the global sparsity of $\beta^{(2)}$. This explains why the illustration in
the last column might be misleading: the two patterns are not equal
but do actually differ by only one covariate.

\medskip

Beyond those six varying sparsity patterns, we consider three
different correlation structures $\Sigma^{(1)}$ and $\Sigma^{(2)}$ for
the generation of the design matrix. In all three cases, we assume
that $\Sigma^{(1)} = \Sigma^{(2)} = \Sigma$. On top of the basic orthogonal
matrix $\Sigma^{(1)}=\Sigma^{(2)} = I_p$, we investigate two randomly
generated correlation structures.

\subparagraph{Power Decay Correlation Structure.}First, we consider a power
decay correlation structure such that $\Sigma_{i,j} =
\rho^{|i-j|}$. Since the sparsity pattern of $\beta^{(1)}$ and
$\beta^{(2)}$ is linked to the order of the covariates, we randomly
permute at each run the columns and rows of $\Sigma$ in order to make sure that the
correlation structure is independent from the sparsity pattern.

\subparagraph{Gaussian Graphical Model Structure.} Second, we simulate
correlation structures with the \texttt{R} package
\texttt{GGMselect}. The function \texttt{simulateGraph}  generates
covariance matrices corresponding to  Gaussian graphical model
structure made of clusters with some intra-cluster and extra-cluster
connectivity coefficients. See Section 4 of \cite{2009_Giraud_Huet} for more
details. A new structure is generated at each run. 


Both random correlation structures are calibrated such that, on
average, each covariate is correlated with 10 other covariates with
correlations above 0.2 in absolute value. This corresponds to fixing
$\rho$ at a value of $0.75$ in the power decay correlation structure and the
intra-cluster connectivity coefficient to 5\% in the Gaussian
graphical model structure. With the default option of the function
\texttt{simulateGraph} the extra-cluster connectivity coefficient is taken
five times smaller.

 \begin{table}
 \centering
 \begin{tabular}{cccccc}
 \hline
 Setting & $\eta$ & $\sharp$ common & $\eta_2$ & $\sharp$ $\beta^{(2)}$
 specific & Signals\\
 \hline
 $\hyp_{00}$ & - & 0 & - & 0
 & \includegraphics[width=.1\textwidth]{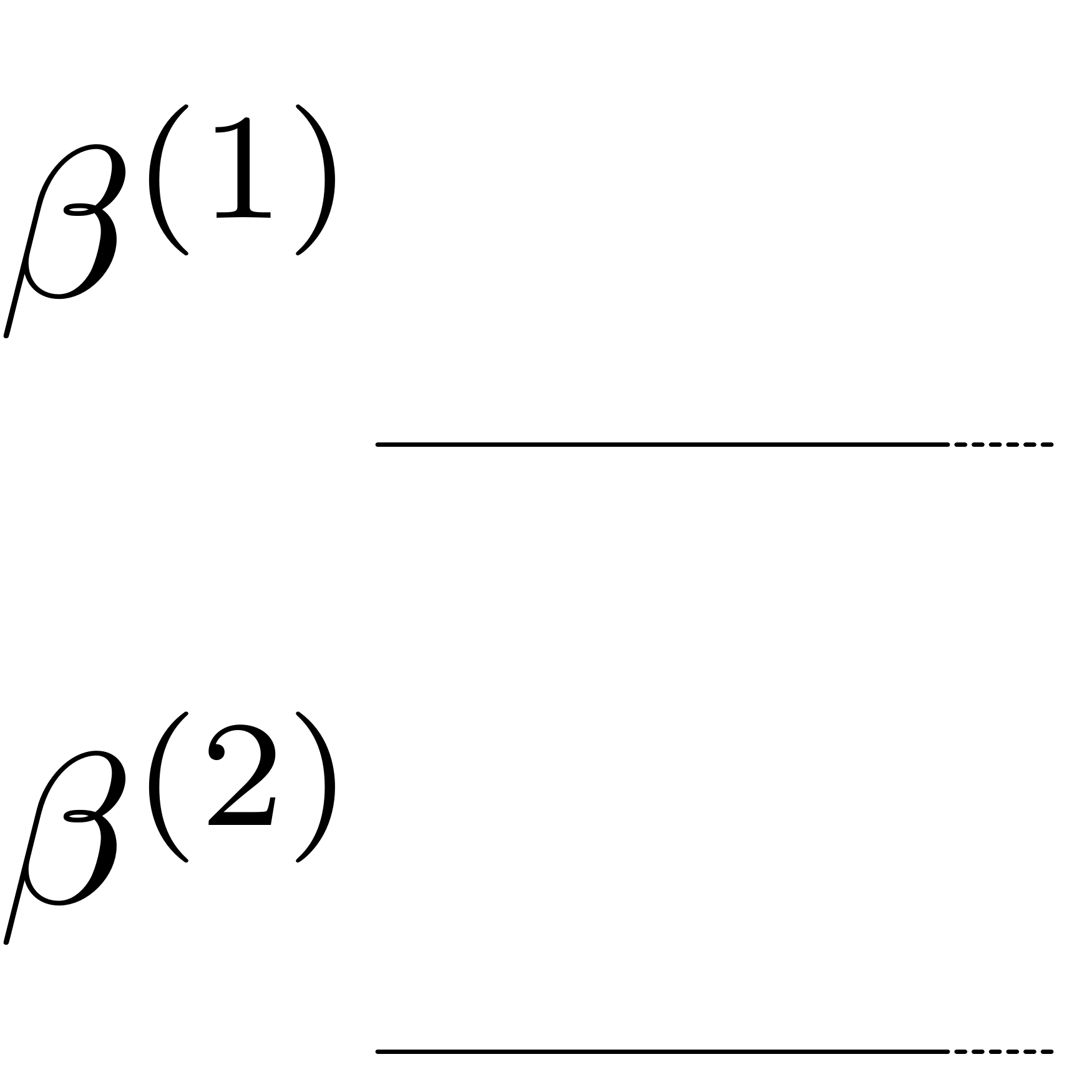}\\[-2ex] \\
 $\hyp_{0}$ & 5/8 & 7 & -& 0
 & \includegraphics[width=.1\textwidth]{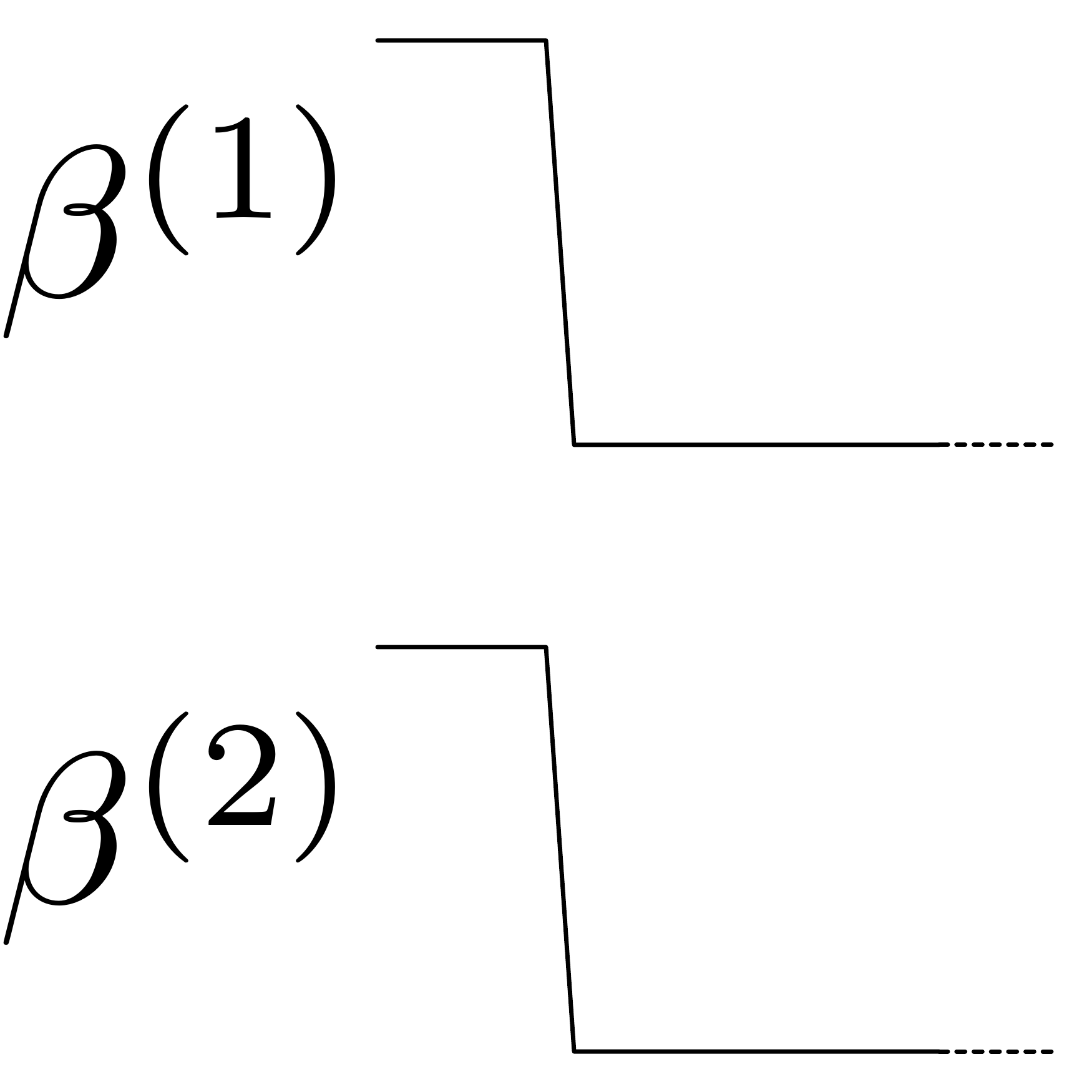}\\[-2ex]\\
 \hline
 1 &  - & 0 & 5/8 & 7 & \includegraphics[width=.1\textwidth]{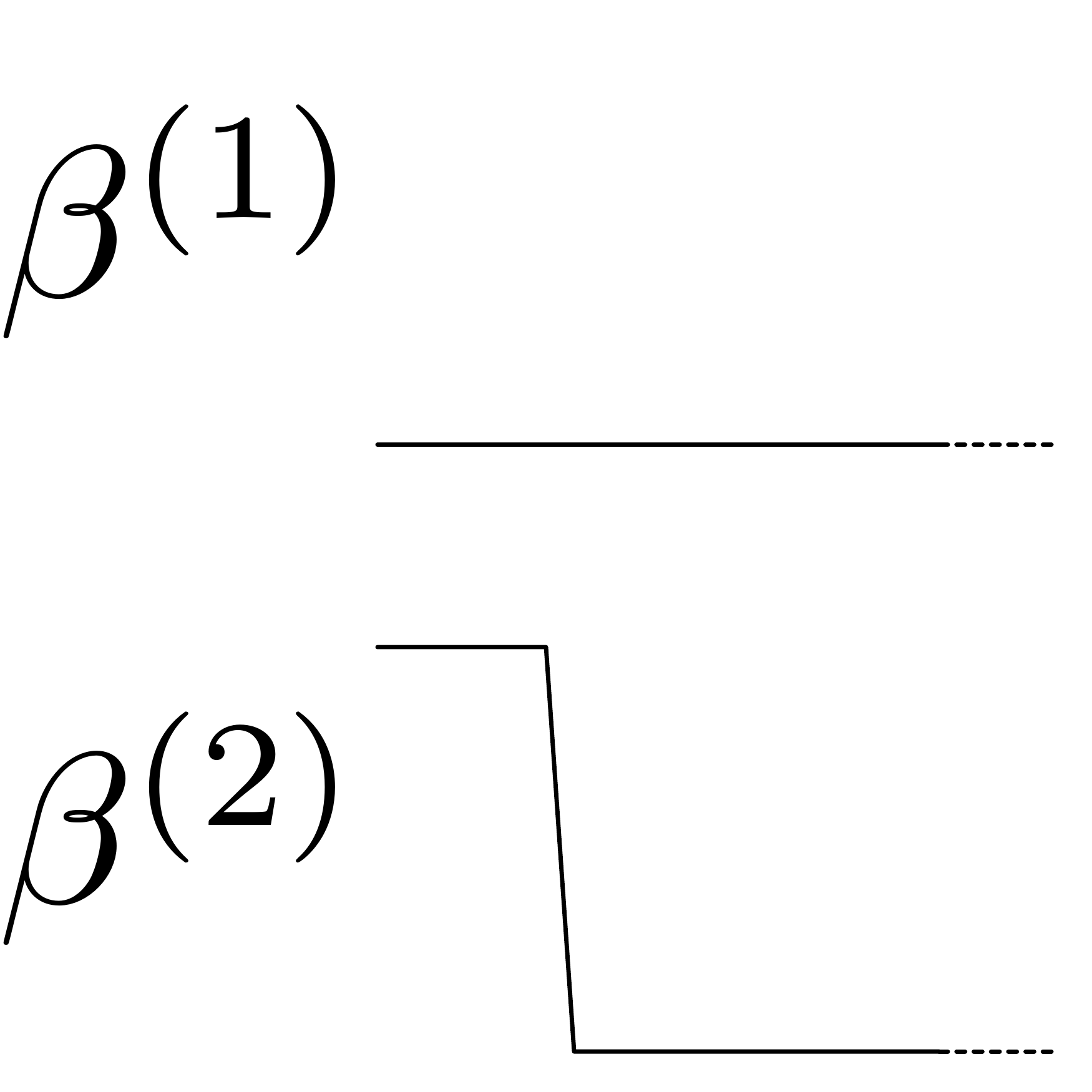}\\[-2ex]\\
 \hline 
 2 &  7/8 & 1 & 5/8 & 7 & \includegraphics[width=.1\textwidth]{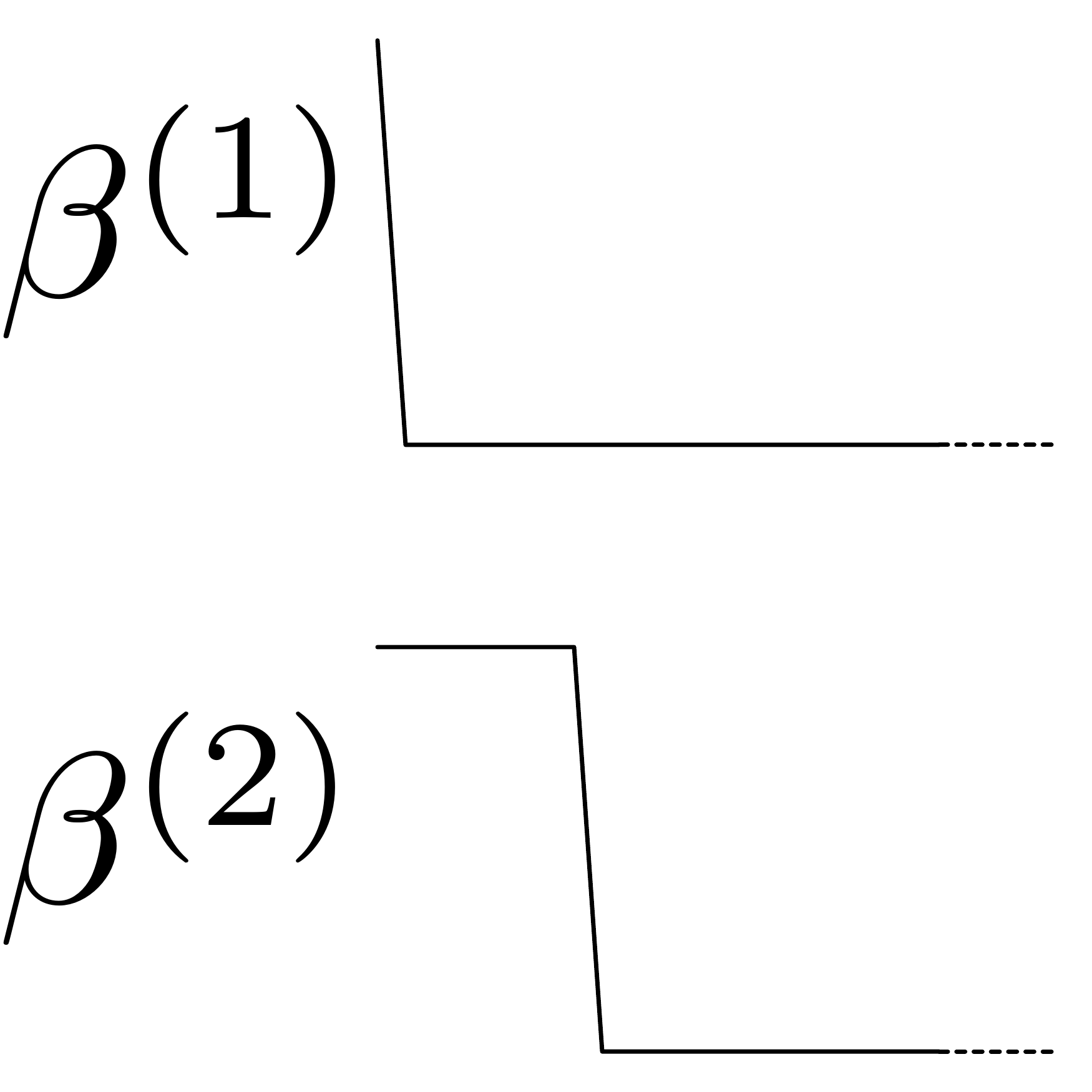}\\[-2ex]\\
 \hline 
 3 &  5/8 & 7 & 5/8 & 7 & \includegraphics[width=.1\textwidth]{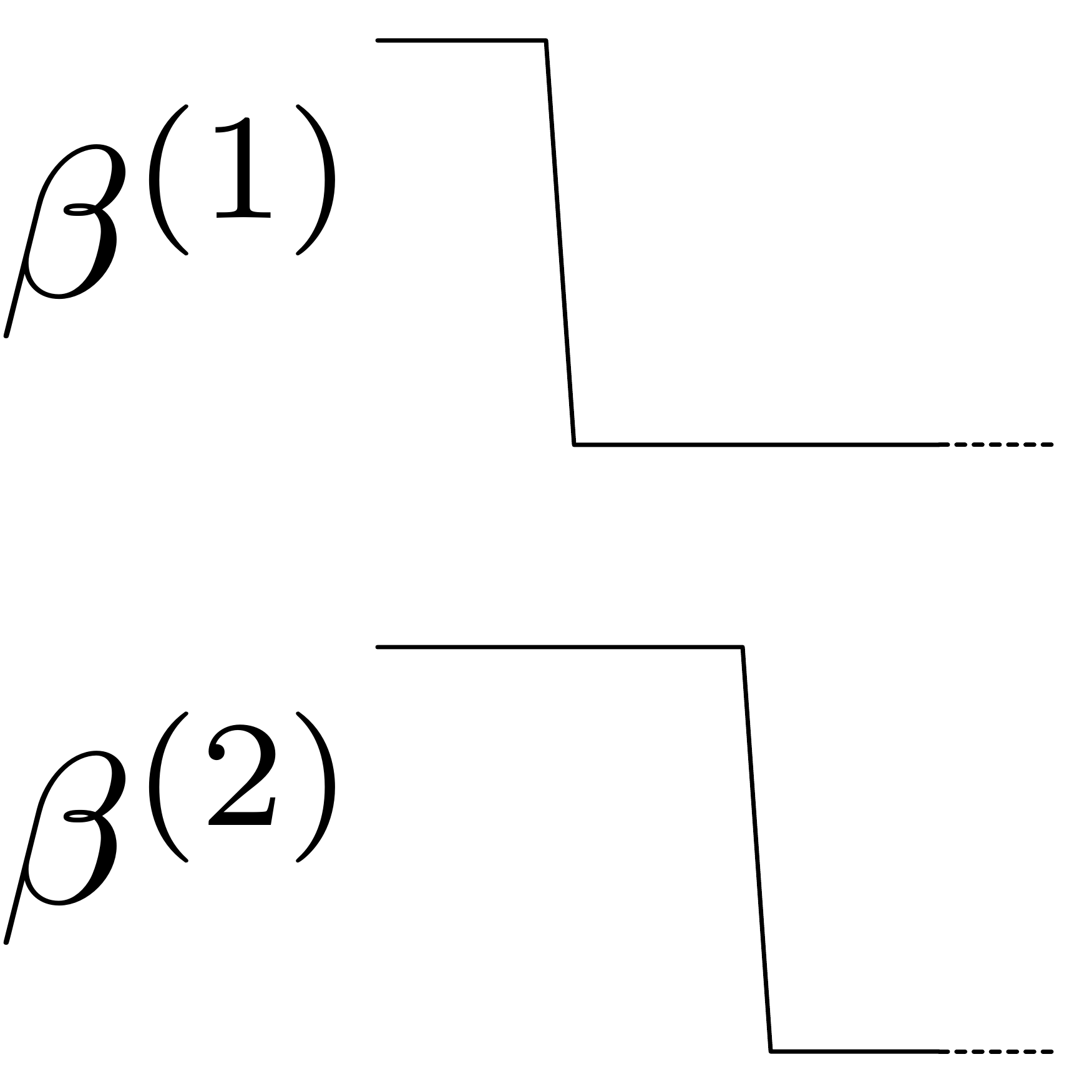}\\[-2ex]\\
 \hline 
 4 &  5/8 & 7 & 7/8 & 1 & \includegraphics[width=.1\textwidth]{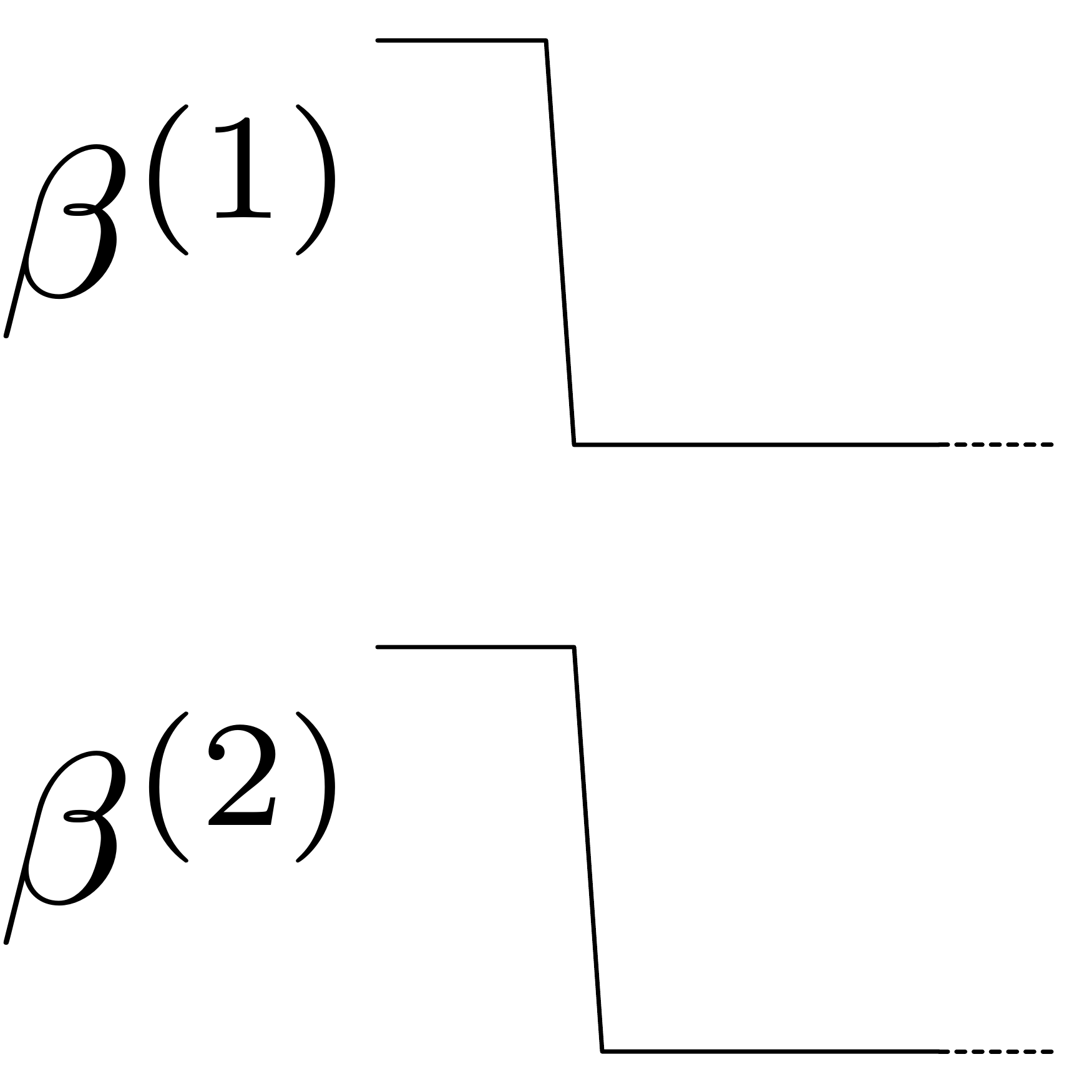}\\[-2ex]\\
  \hline
 \end{tabular}
 \caption{\label{tab:scen} Summary of the six different sparsity
   patterns under study.}
 \end{table}

\paragraph{Test statistics, collections and calibrations in competition}

In the following, we present the results of the proposed test statistics
combined with two test collections, namely a deterministic and data-driven
model collection, respectively $\cS_1$ and
$\SLasso$, as well as with a Bonferroni ${\bf (B)}$ or
Permutation ${\bf (P)}$ calibration (computed with 100 random permutations). 

Furthermore, to put those results in perspective, we compare 
our suggested test statistic to the usual Fisher statistic and we
compare our approach with the parallel work of  \cite{2013_ArXiv_Stadler}.

 \textit{Fisher statistic}. 
For a given support $|S|$ of reduced dimension the usual likelihood
ratio statistic 
for the equality of $\beta_S^{(1)}$ and $\beta_S^{(2)}$ follows a
Fisher distribution with $|S|$ and $n_1+n_2 -2 |S|$ degrees of
freedom:
\begin{equation}
\mathrm{Fi}_S = \frac{\|\Y-\X_S\widehat{\beta}_S\|^2 - \|\Y^{(1)}-\X_S^{(1)}\widehat{\beta}^{(1)}_S\|^2-\|\Y^{(2)}-\X_S^{(2)}\widehat{\beta}^{(2)}_S\|^2}{\|\Y^{(1)}-\X_S^{(1)}\widehat{\beta}^{(1)}_S\|^2+\|\Y^{(2)}-\X_S^{(2)}\widehat{\beta}^{(2)}_S\|^2} \frac{n_1+n_2-2|S|}{|S|},
\end{equation}
where $\widehat{\beta}_S$ is the maximum likelihood estimator
restricted to covariates in support $S$ on the
concatenated sample $(\X,\Y)$. While this statistic $\mathrm{Fi}_S$ is able to detect differences between $\beta^{(1)}$ and $\beta^{(2)}$, it is not really suited for detecting differences between the standard deviations $\sigma^{(1)}$ and $\sigma^{(2)}$.

The Fisher statistic $\mathrm{Fi}_S$ is adapted to the high-dimensional framework
similarly as the suggested statistics $(F_{S,V},F_{S,1},F_{S,2}) $, except that exact $p$-values are
available. The corresponding test with a collection
$\widehat{\mathcal{S}}$ and a Bonferroni (resp. permutation) calibration
is denotes $T^{B,\mathrm{Fisher}}_{\widehat{\mathcal{S}}}$
($T^{P,\mathrm{Fisher}}_{\widehat{\mathcal{S}}}$).

 \textit{Procedure of St\"adler and Mukherjee \cite{2013_ArXiv_Stadler}}.
The DiffRegr procedure of St\"adler and Mukherjee performs two-sample
testing between high-dimensional regression models.  The procedure is
based on sample-splitting: the data is split into two parts, the
first one allowing to reduce dimensionality (\emph{screening} step) and the second being used  to compute p-values based on
a restricted log-likelihood-ratio statistic (\emph{cleaning} step). To increase the stability of
the results the splitting step is repeated multiple times and the
resulting p-values must be aggregated.
We choose to use the p-value calculations based on permutations as it
remains computationally reasonable for the regression case, see
\cite{2013_ArXiv_Stadler}.  
The single-splitting and multi-splitting procedures are denoted
respectively SS(perm) and MS(perm).


\paragraph{Validation of Type I Error Control}
\subparagraph{Control Under the Global Null Hypothesis $\hyp_{00}$.}
Table \ref{tab:level} presents estimated 
type I error rates, that is the percentage of simulations for which the null
hypothesis is rejected,   based upon $1000$ simulations under the
restricted null hypothesis $\hyp_{00}$, where $\beta^{(1)} = \beta^{(2)}
= 0$ and under orthogonal correlation structure. The
desired level is $\alpha=5\%$, and the estimated levels are given with  
a $95\%$ Gaussian confidence interval. 

As expected under independence, the combination of the $\cS_1$ collection with Bonferroni
correction gives accurate alpha-level  when applied  to the usual
Fisher  statistic. On the contrary when applied to the
suggested statistics,  the use  of upper bounds  on p-values  leads to
a strong decrease in observed type-I error. This decrease is exacerbated when using the
$\SLasso$ collection,  since we are  accounting for many more models
than the number actually tested in order to prevent overfitting. This
effect can be seen both on the Fisher statistic and our suggested statistic.
Even with the usual Fisher statistic, for which we know the
exact $p$-value, it is unthinkable to use Bonferroni calibration as
soon as we adopt data-driven collections instead of deterministic
ones.

On the contrary, a calibration by permutations restores a control of
type-I error at the desired nominal level, whatever the test statistic
or model collection.

As noted by \cite{2013_ArXiv_Stadler}, the multi-splitting procedure
yields conservative results in terms of type I error control at level $5\%$.

\begin{table}[!ht]
\begin{center}
\subtable[Tests  $T^{*}_{\widehat{\mathcal{S}}}$]{
\begin{tabular}{l*{5}{c}} 
\hline
Model collection &\multicolumn{2}{c}{$\mathcal{S}_1$} & \multicolumn{2}{c}{$\widehat{\mathcal{S}}_{Lasso}$} \\ 
Calibration  &({\bf B}) & ({\bf P}) &({\bf B}) & ({\bf P}) \\ 
n= 25 &$ 1 \pm 0.6 $&$ 6.9 \pm 1.6 $&$ 0 \pm 0 $&$ 6.9 \pm 1.6 $\\
n= 50 &$ 1.8 \pm 0.8 $&$ 5.8 \pm 1.4 $&$ 0 \pm 0 $&$ 6 \pm 1.5 $\\
n= 100 &$ 1 \pm 0.6 $&$ 7.4 \pm 1.6 $&$ 0.1 \pm 0.2 $&$ 7.3 \pm 1.6 $\\
\hline
\end{tabular} 
}
\subtable[Tests  $T^{*,\mathrm{Fisher}}_{\widehat{\mathcal{S}}}$]{
\begin{tabular}{l*{4}{c}} 
\hline
Model collection &\multicolumn{2}{c}{$\mathcal{S}_1$} & \multicolumn{2}{c}{$\widehat{\mathcal{S}}_{Lasso}$} \\ 
Calibration  &({\bf B}) & ({\bf P}) &({\bf B}) & ({\bf P}) \\ 
n= 25 &$ 5.5 \pm 1.4 $&$ 6.8 \pm 1.6 $&$ 0.5 \pm 0.4 $&$ 6.5 \pm 1.5 $\\
n= 50 &$ 4.5 \pm 1.3 $&$ 5.5 \pm 1.4 $&$ 0.1 \pm 0.2 $&$ 5.3 \pm 1.4 $\\
n= 100 &$ 4.8 \pm 1.3 $&$ 6.6 \pm 1.5 $&$ 0.1 \pm 0.2 $&$ 6.5 \pm 1.5
$\\
\hline
\end{tabular} 
}
\subtable[ DiffRegr procedure]{
\begin{tabular}{l*{2}{c}} 
\hline
Model collection &\multicolumn{1}{c}{SS (perm)} & \multicolumn{1}{c}{MS (perm)} \\ 
n= 25 &$ 4.3 \pm 1.3 $&$ 0.1 \pm 0.2 $\\
n= 50 &$ 4.1 \pm 1.2 $&$ 0.2 \pm 0.3 $\\
n= 100 &$ 3.5 \pm 1.1 $&$ 0.1 \pm 0.2 $\\
\hline
\end{tabular} 
}
\caption{\label{tab:level}Estimated test levels in percentage along
  with 95\% Gaussian confidence interval (in percentage) under $\hyp_{00}$ based upon 1000 simulations. }
\end{center}
\end{table}

\subparagraph{Control Under the Global Equality of Non Null
  Coefficients $\hyp_{0}$.} Figures \ref{fig:level} and \ref{fig:levelcorr}
present level checks under $\hyp_0$ but with non null $\beta^{(1)}=
\beta^{(2)}\neq 0$, under respectively orthogonal and non-orthogonal
correlation structures. Conclusions are perfectly similar to the case
$\hyp_{00}$: all methods behave well, except the multi-split DiffRegr
procedure and the Bonferroni calibration-based procedures
 $T^{B}_{\widehat{\mathcal{S}}}$       (for       any       collection
 $\widehat{\mathcal{S}}$)                                           and
 $T^{B,\mathrm{Fisher}}_{\mathcal{\SLasso}}$.  In particular, the Fisher
 statistic combined  with $\cS_1$  and Bonferroni calibration  is more
 conservative than the desired nominal level under correlated designs.

\begin{figure}[!ht]
\centering
\begin{tabular}{ccc}
$n=25$ & $n=50$ & $n=100$ \\
\includegraphics[width=.3\textwidth]{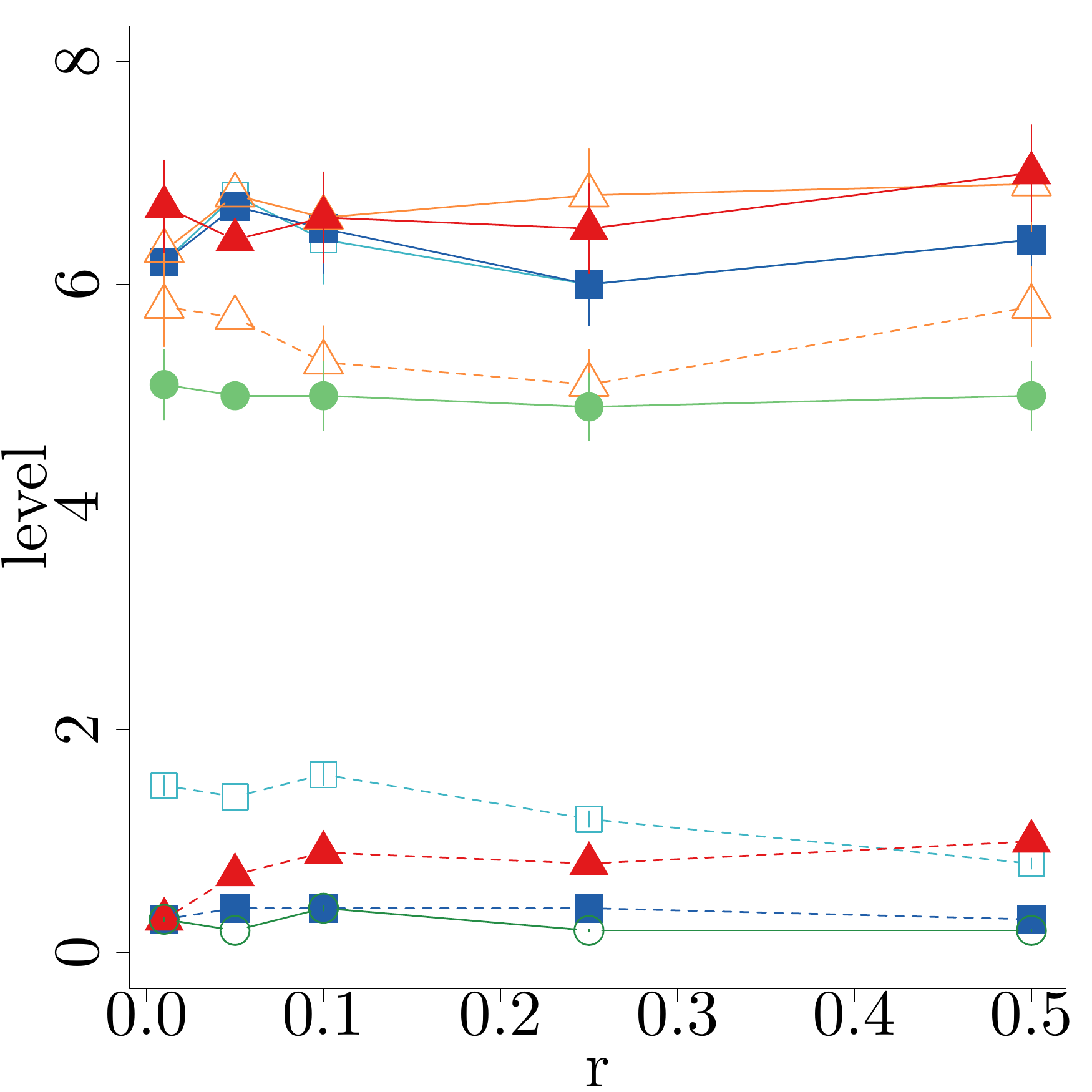}
&
\includegraphics[width=.3\textwidth]{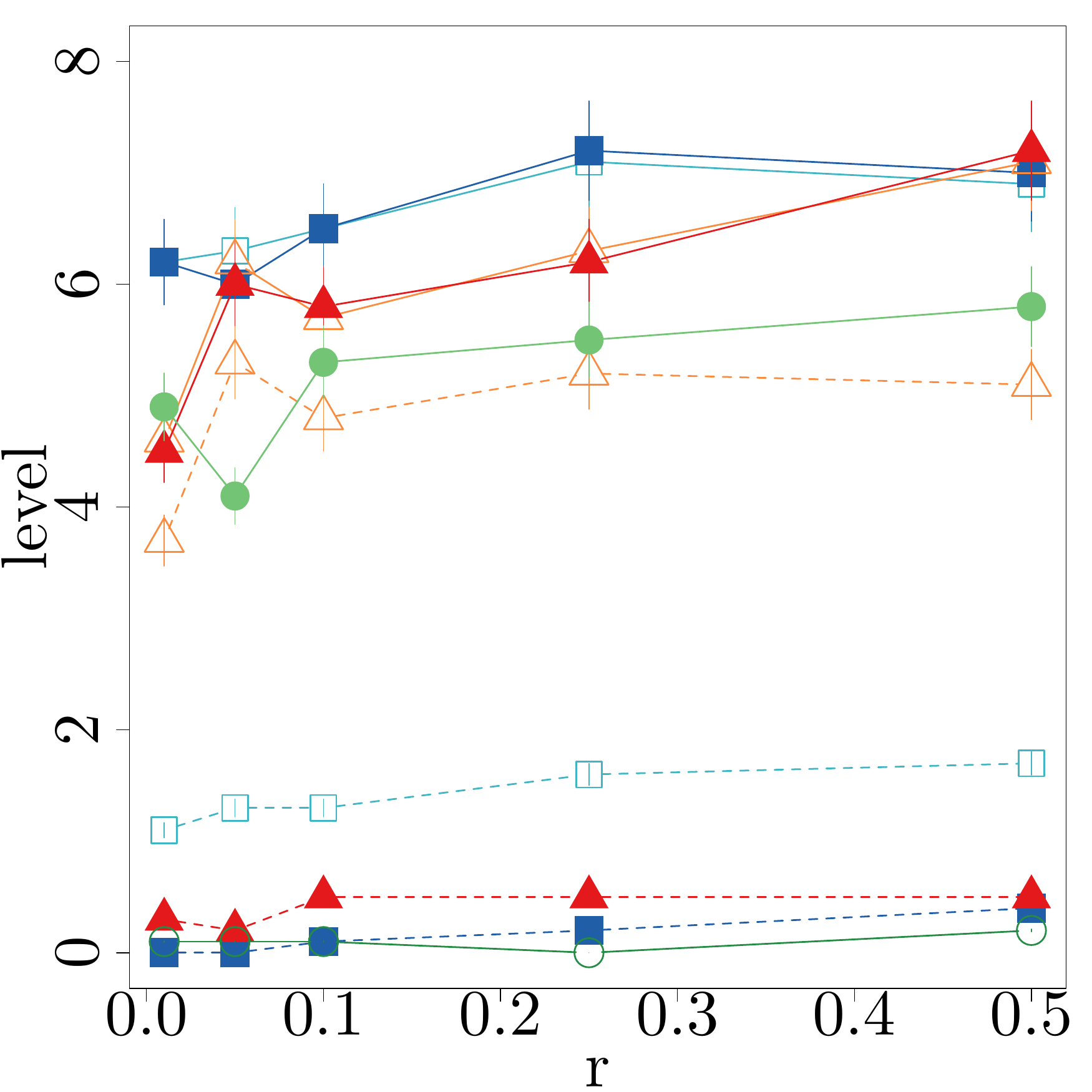}
&
\includegraphics[width=.3\textwidth]{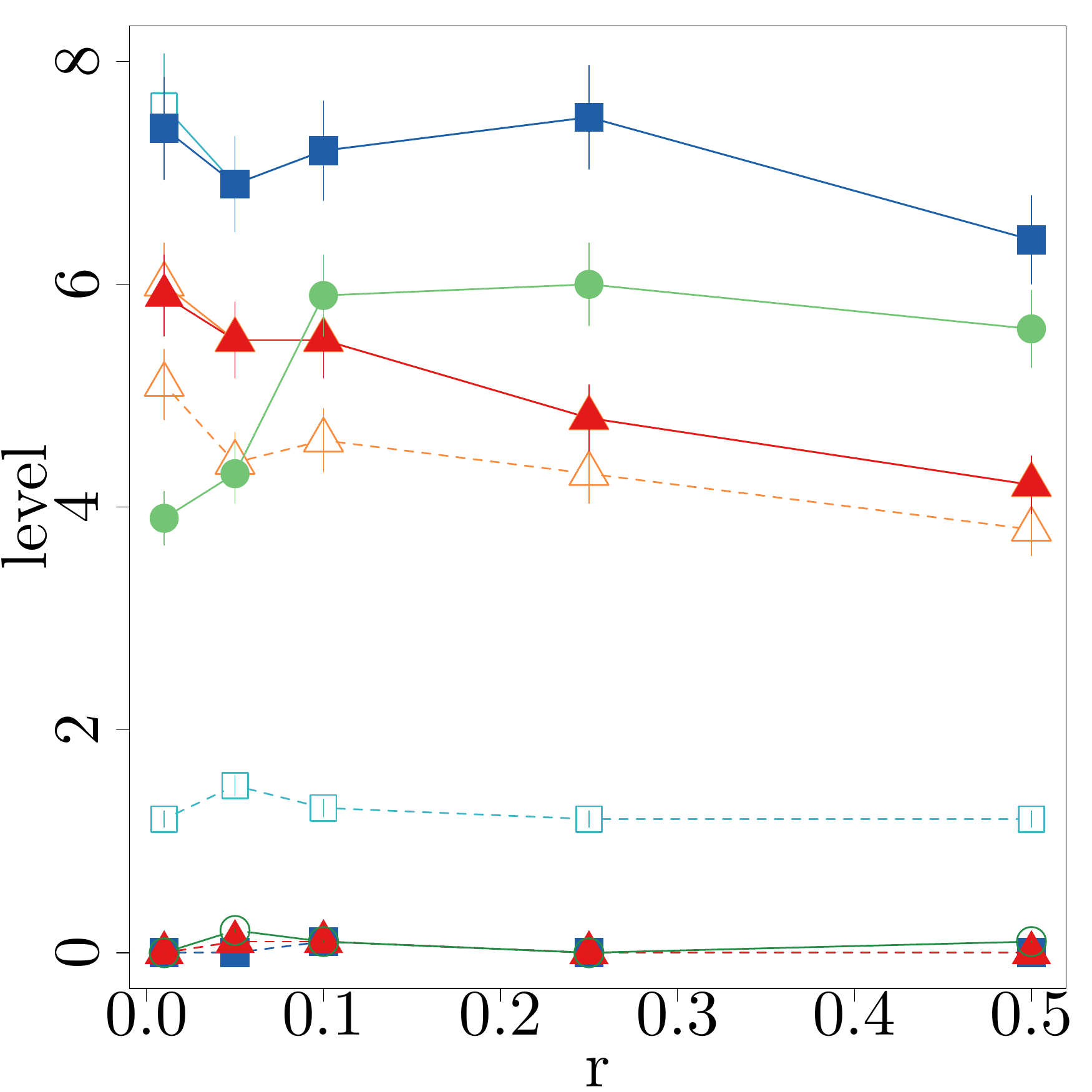} \\
\end{tabular}
\caption{\label{fig:level}Estimated test levels in percentage under $\hyp_{0}$ 
  for varying magnitudes of common non
  null coefficients, based upon 1000 simulations. Bonferroni
  calibration in dotted lines, calibration by permutation in plain
  lines. Blue squares represent the suggested test
  $T^{*}_{\widehat{\mathcal{S}}}$, red triangles
  stand for the Fisher test  $T^{*,\mathrm{Fisher}}_{\widehat{\mathcal{S}}}$. The deterministic collection
  $\cS_1$ is drawn in empty points, while the data-driven collection
  $\SLasso$ is in plain points. Green circles represent the
  DiffRegr procedure, respectively plain and empty for  single-splitting and
multi-splitting. }
\end{figure}

\begin{figure}[!ht]
\centering
\includegraphics[width=.3\textwidth]{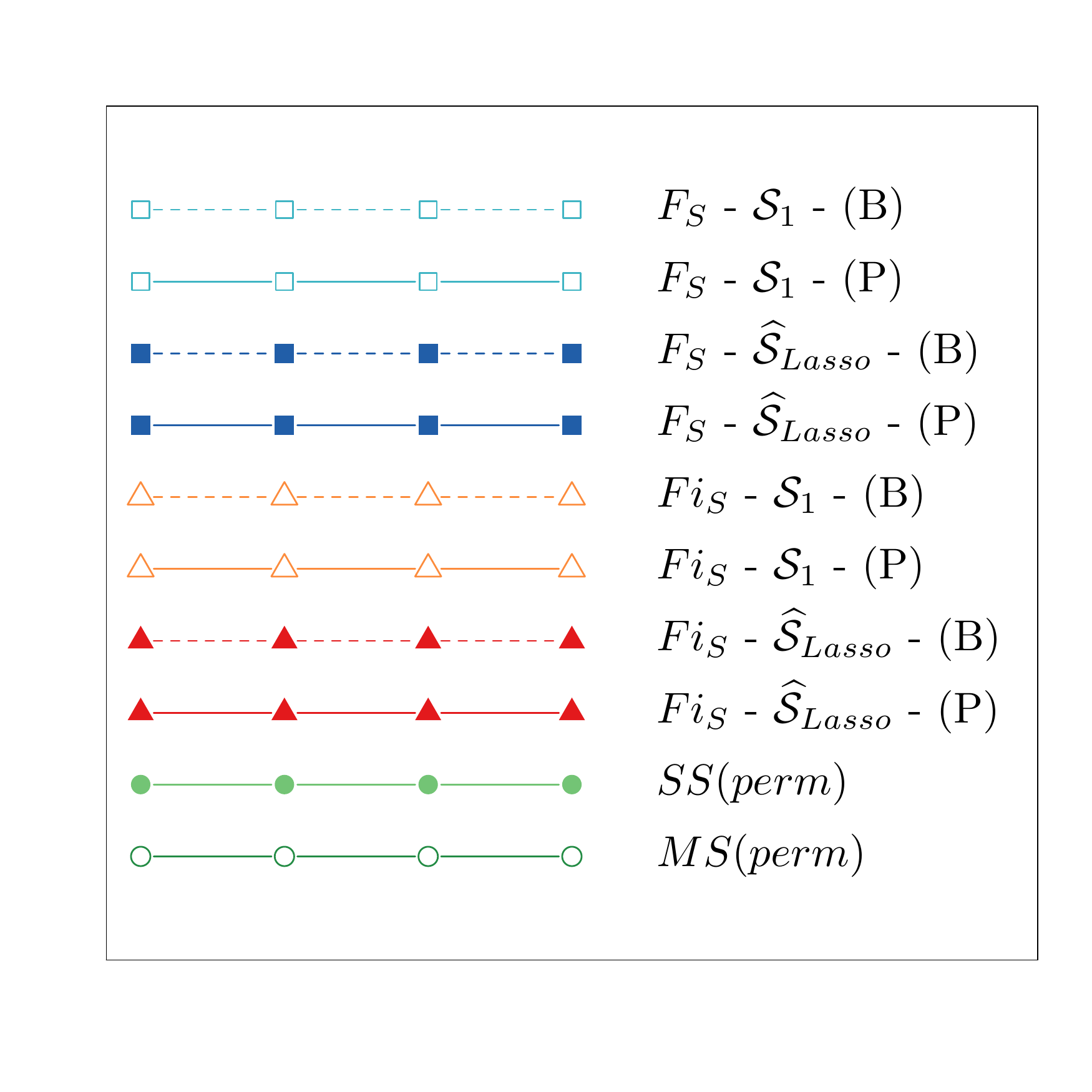}
\caption{Legend of the procedures under study}
\end{figure}

\begin{figure}[!ht]
\centering
\begin{tabular}{cc}
Power decay & GGM \\
\includegraphics[width=.3\textwidth]{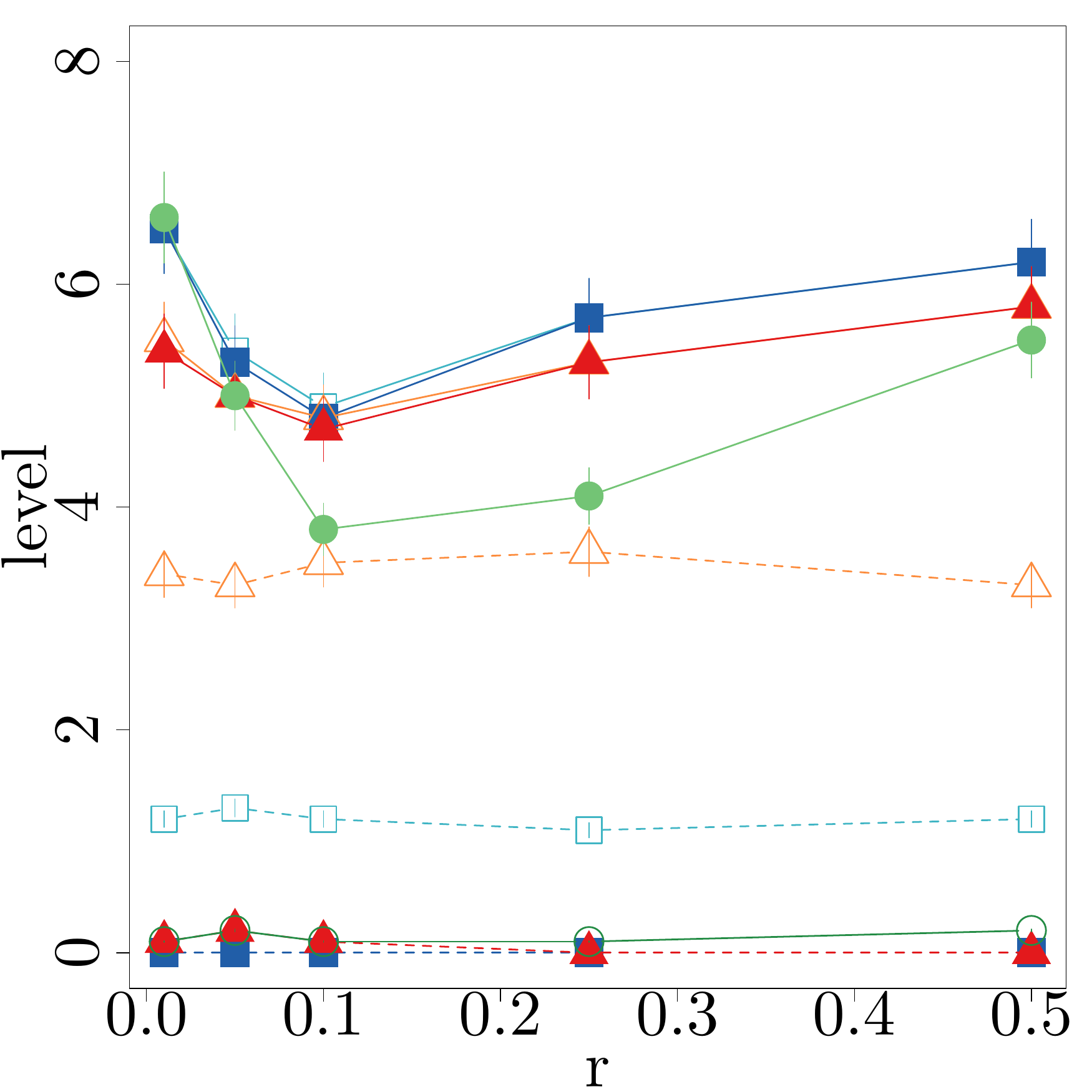}
&
\includegraphics[width=.3\textwidth]{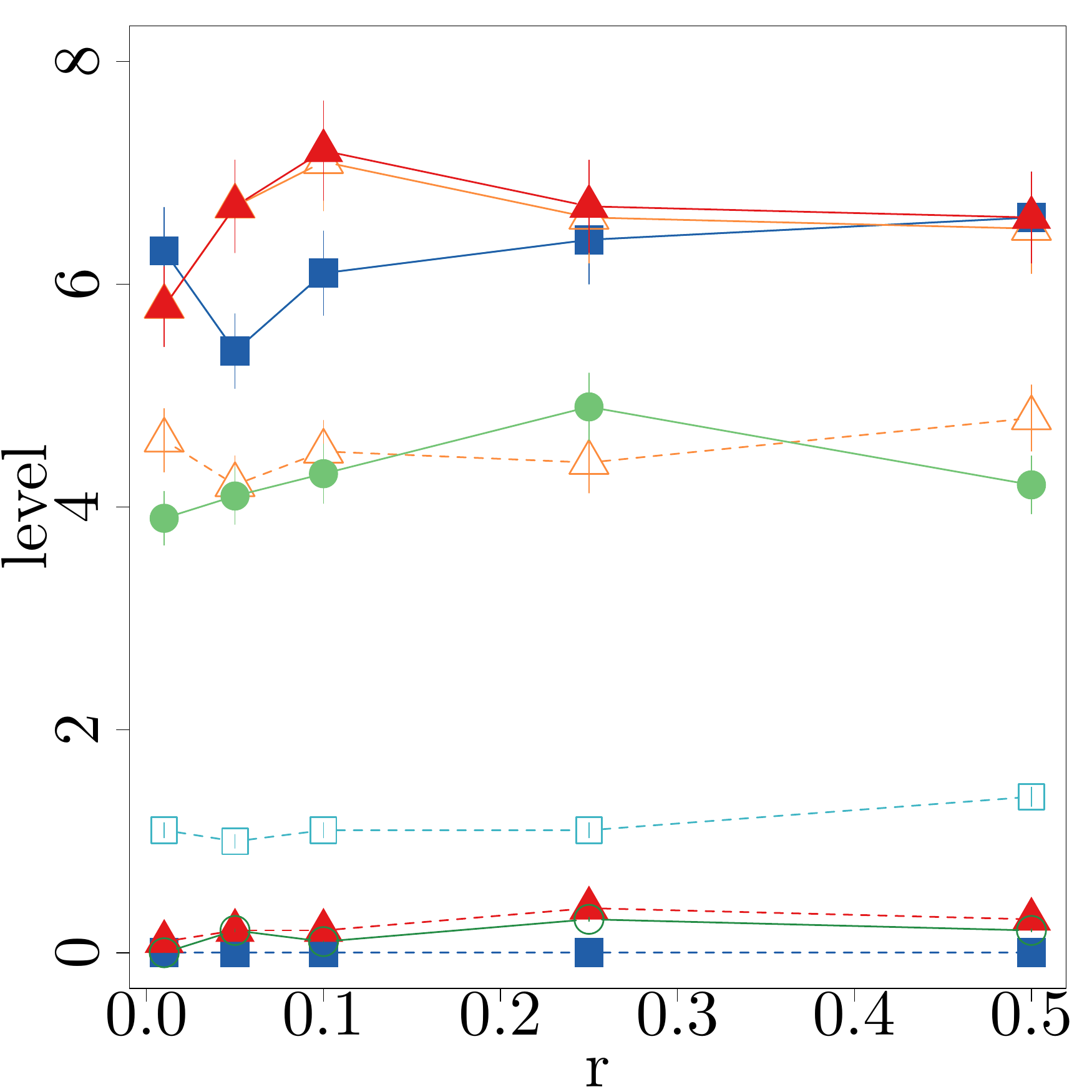} \\
\end{tabular}
\caption{\label{fig:levelcorr}Estimated test levels in percentage under
  $\hyp_{0}$ for varying magnitudes of common non
  null coefficients, based upon 1000 simulations, under power decay
  and GGM correlation structures when $n=50$.  Bonferroni
  calibration in dotted lines, calibration by permutation in plain
  lines. Blue squares represent the suggested test
  $T^{*}_{\widehat{\mathcal{S}}}$, red triangles
  stand for the Fisher test  $T^{*,\mathrm{Fisher}}_{\widehat{\mathcal{S}}}$. The deterministic collection
  $\cS_1$ is drawn in empty points, while the data-driven collection
  $\SLasso$ is in plain points. Green circles represent the
  DiffRegr procedure, respectively plain and empty for  single-splitting and
multi-splitting. }
\end{figure}

\paragraph{Power Analysis.} We do not investigate the power of the Bonferroni-based procedures $T^{B}_{\widehat{\cS}}$ and $T^{B,\mathrm{Fisher}}_{\widehat{\cS}}$ as they have been shown to be too conservative in the above Type I error analysis. Figure \ref{fig:powerCLRvsFisher}
represents power performances for the  test
$T^{P}_{\widehat{\mathcal{S}}}$ and 
the usual likelihood ratio test $T^{P,\mathrm{Fisher}}_{\widehat{\mathcal{S}}}$ combined with either $\cS_1$
or $\SLasso$ test collections using a calibration by permutation under
an orthogonal covariance matrix $\Sigma$, as well as power performance
for the DiffRegr procedure. Figure \ref{fig:correlated} represents
equivalent results for power decay and GGM 
covariance structures when $n=50$. 

In the absence of common coefficients (scenarios 1 and 2), the test $T^{P}_{\widehat{\mathcal{S}}}$
reaches 100\% power from very low signal magnitudes and small sample
sizes. Compared to the test based on usual likelihood ratio statistics, which
does not reach more than 40\% power when $n=25$ given the signal
magnitudes under consideration, the suggested statistics proves itself
extremely efficient. Under these settings as well, any subset of
size 1 containing one of the variables activated in only $\beta^{(2)}$
can suffice to reject the null, which is why collection $\cS_1$
performs actually very well when associated with $(F_{S,V},F_{S,1},F_{S,2})$ and not so
badly when associated with $\mathrm{Fi}_S$.

However, in more complex settings 3 and 4, where larger subsets are
required to correct for strong and numerous common effects,
subset collection $\SLasso$ yields a higher power than
the collection $\cS_1$. 

For small $n$, the test $T^{P}_{\SLasso}$
outperforms the procedure DiffRegr, whose limitation likely stems from the half sampling step. This limitation of sample
splitting approaches has already been noticed by \cite{2013_ArXiv_Stadler}.
However, for $n=100$ the  procedure DiffRegr performs better than our procedure
in the highly challenging setting  4.

\begin{figure}[!htbp]
 \begin{adjustwidth}{-1in}{-1in}
 \centering
 \begin{tabular}{m{1.5cm}ccc}
 Settings &$n=25$ & $n=50$ & $n=100$ \\
 \raisebox{.3\textwidth}{1. \includegraphics[width=.1\textwidth]{data/settings_1.pdf}}& \includegraphics[width=.3\textwidth]{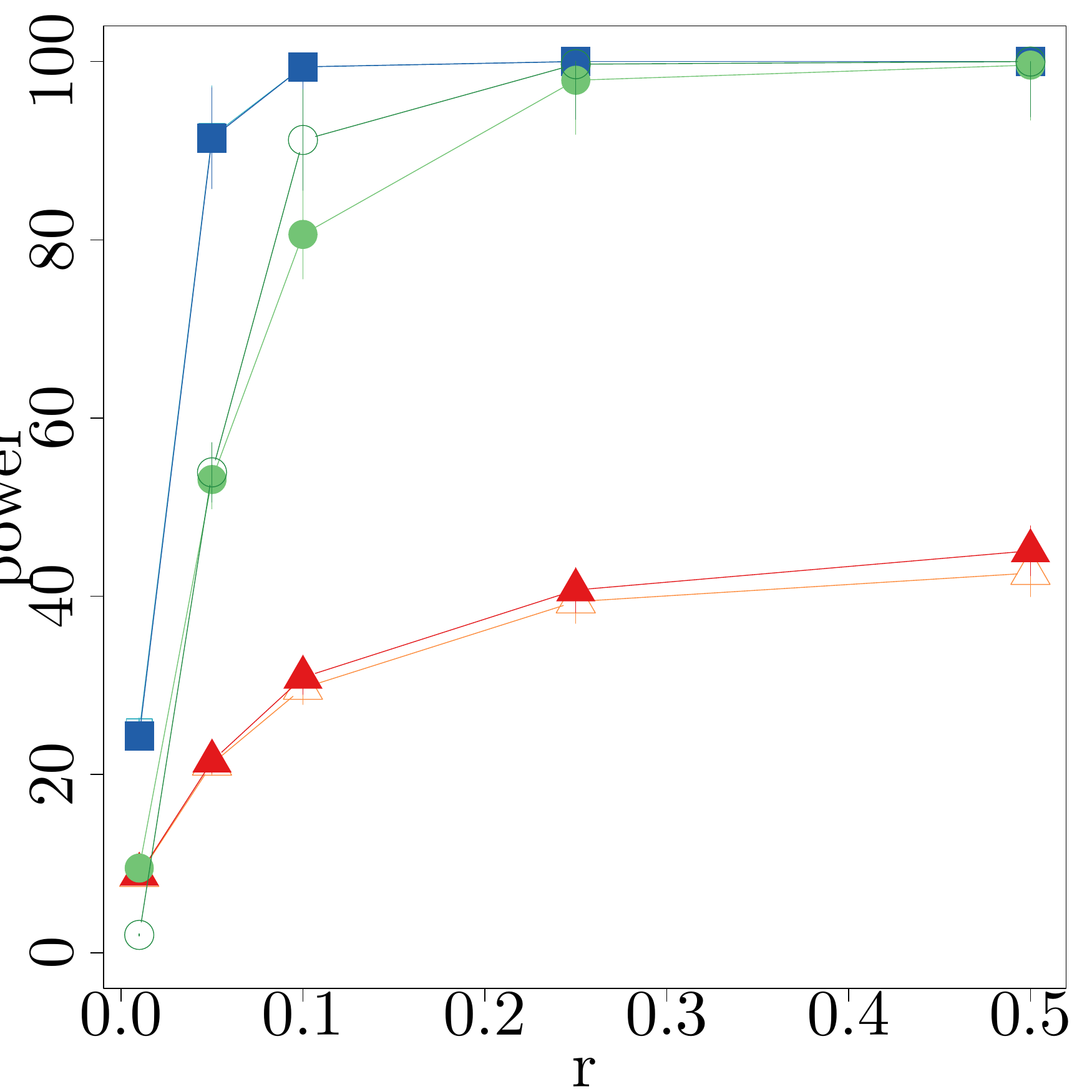}
 &
 \includegraphics[width=.3\textwidth]{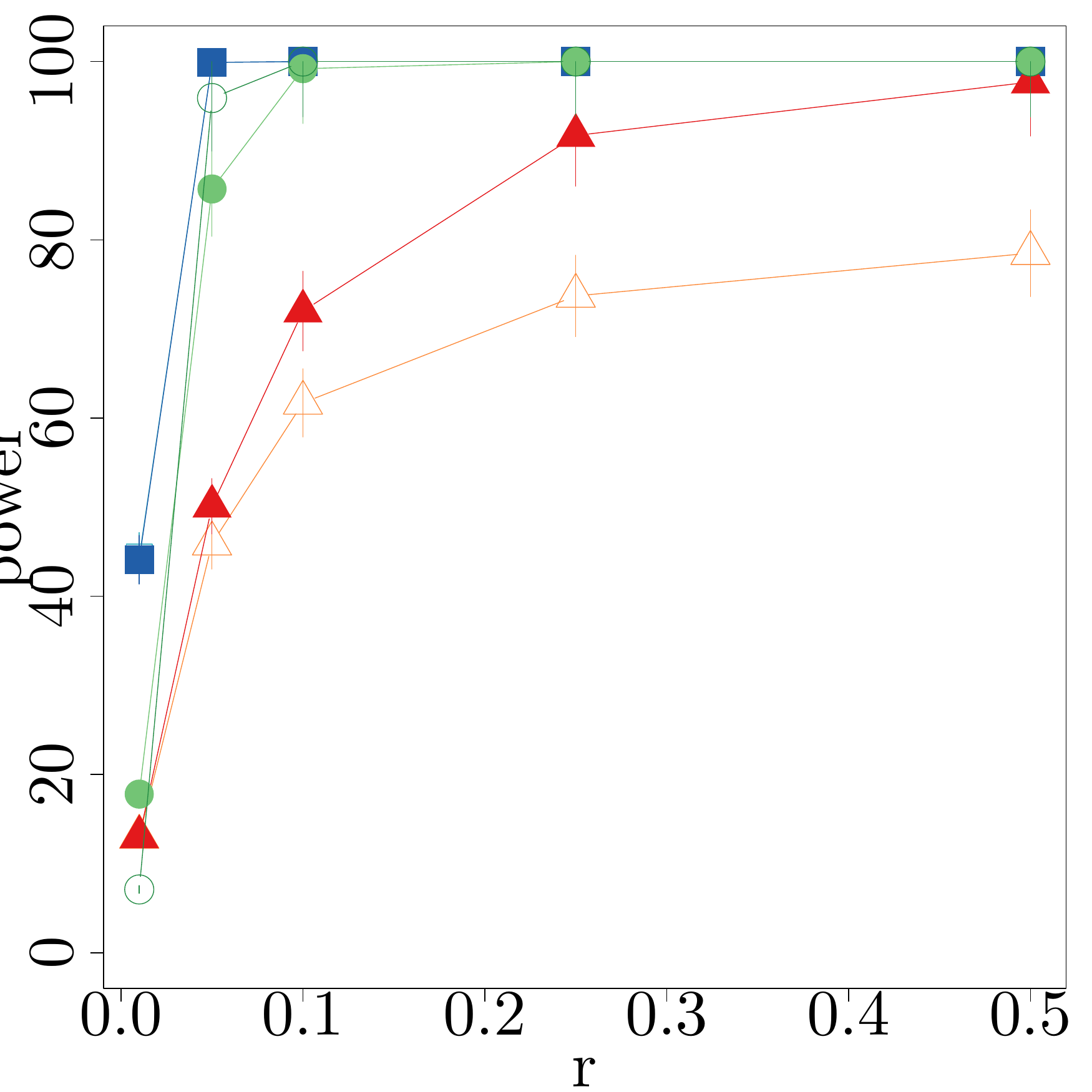}
 &
 \includegraphics[width=.3\textwidth]{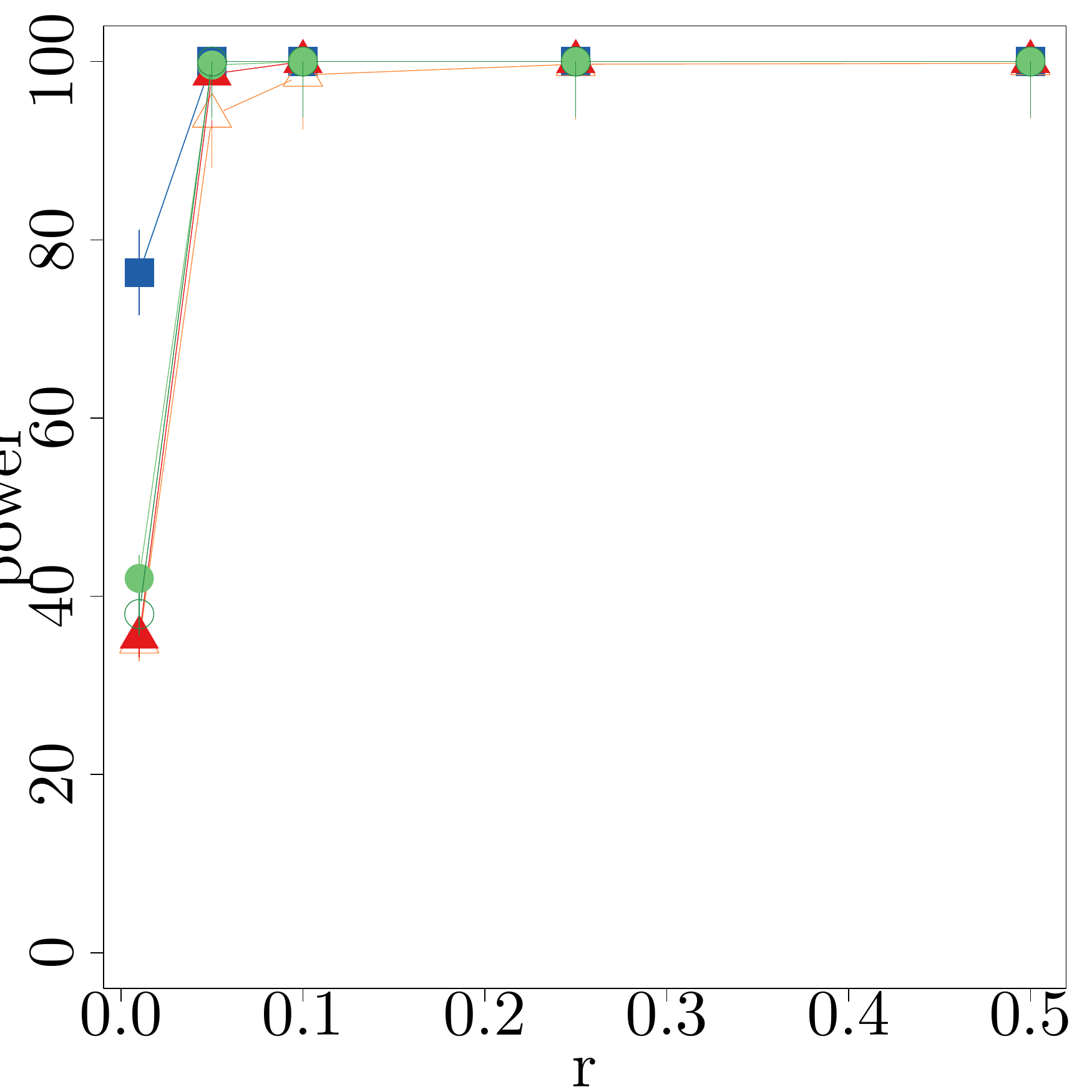} \\[-20ex]
 \raisebox{.3\textwidth}{2. \includegraphics[width=.1\textwidth]{data/settings_2.pdf}}&\includegraphics[width=.3\textwidth]{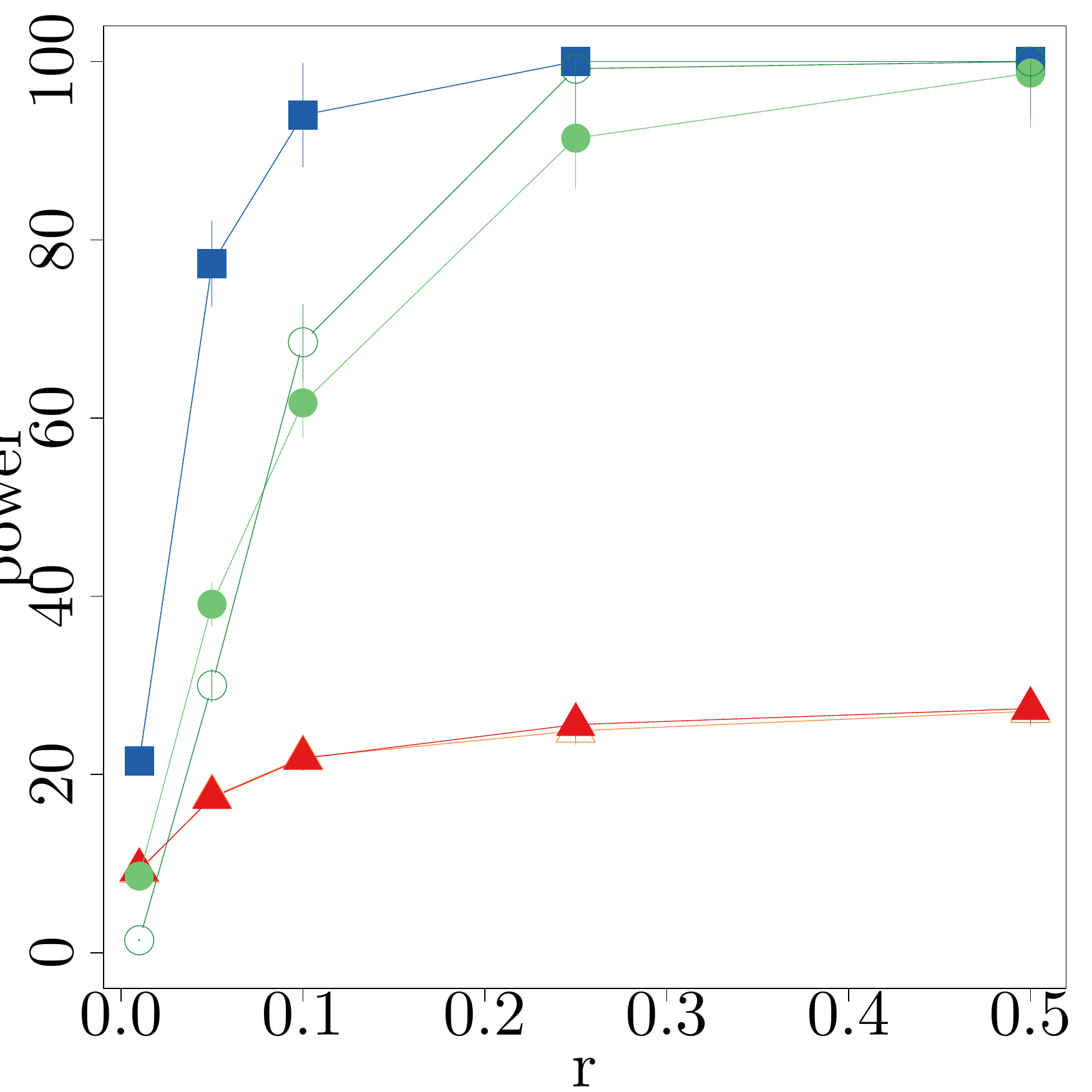}
 &
 \includegraphics[width=.3\textwidth]{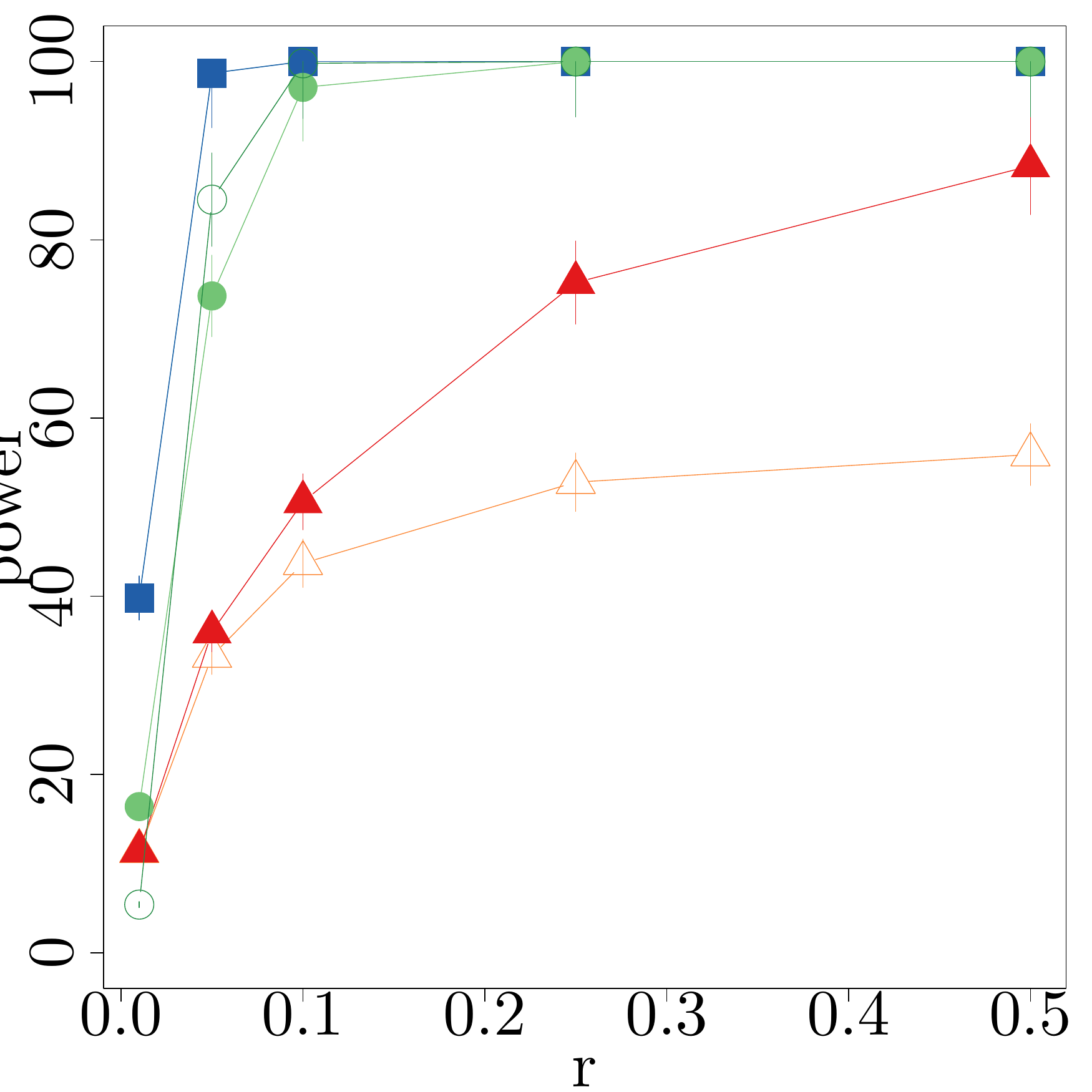}
 &
 \includegraphics[width=.3\textwidth]{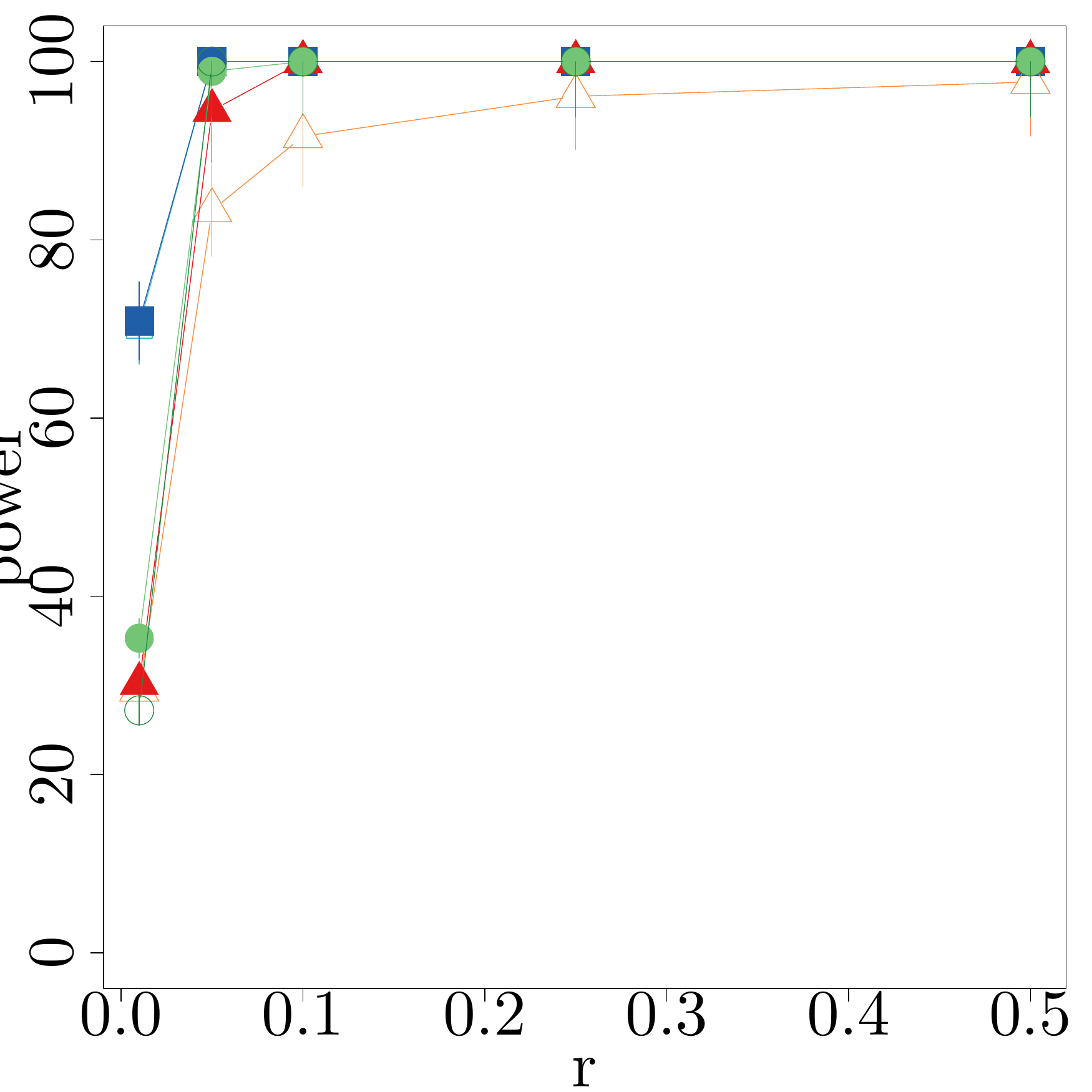} \\[-20ex]
 \raisebox{.3\textwidth}{3. \includegraphics[width=.1\textwidth]{data/settings_3.pdf}}&\includegraphics[width=.3\textwidth]{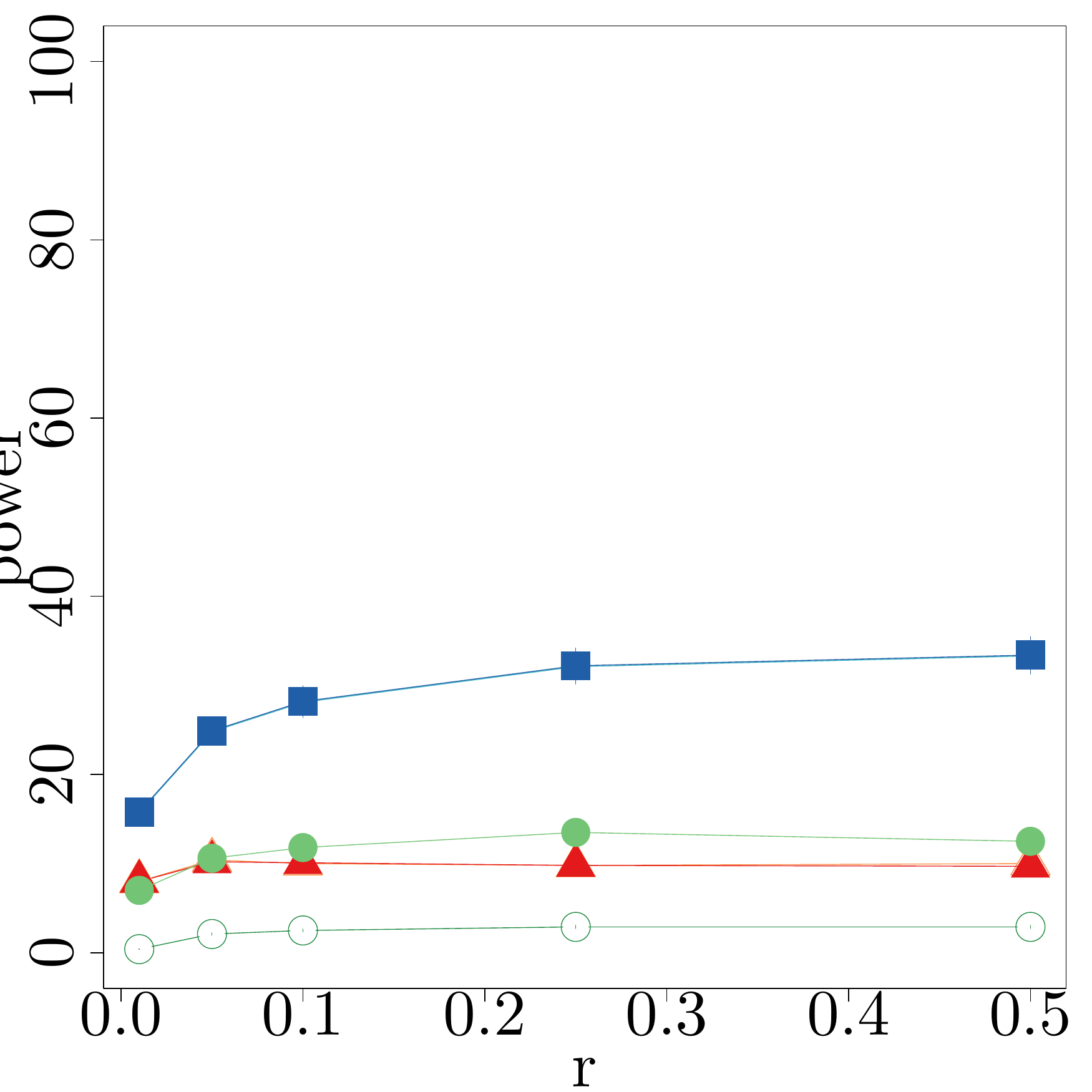}
 &
 \includegraphics[width=.3\textwidth]{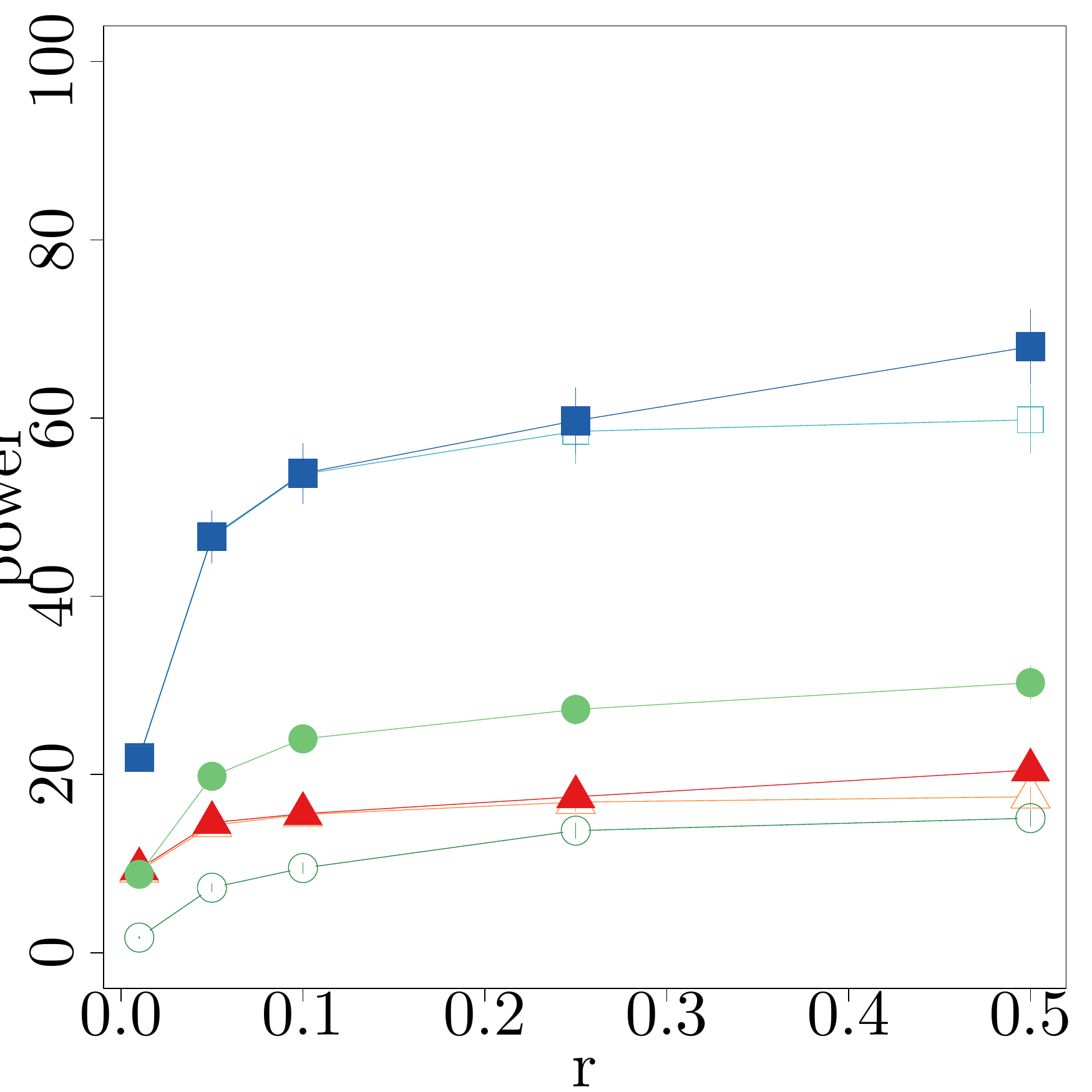}
  &
\includegraphics[width=.3\textwidth]{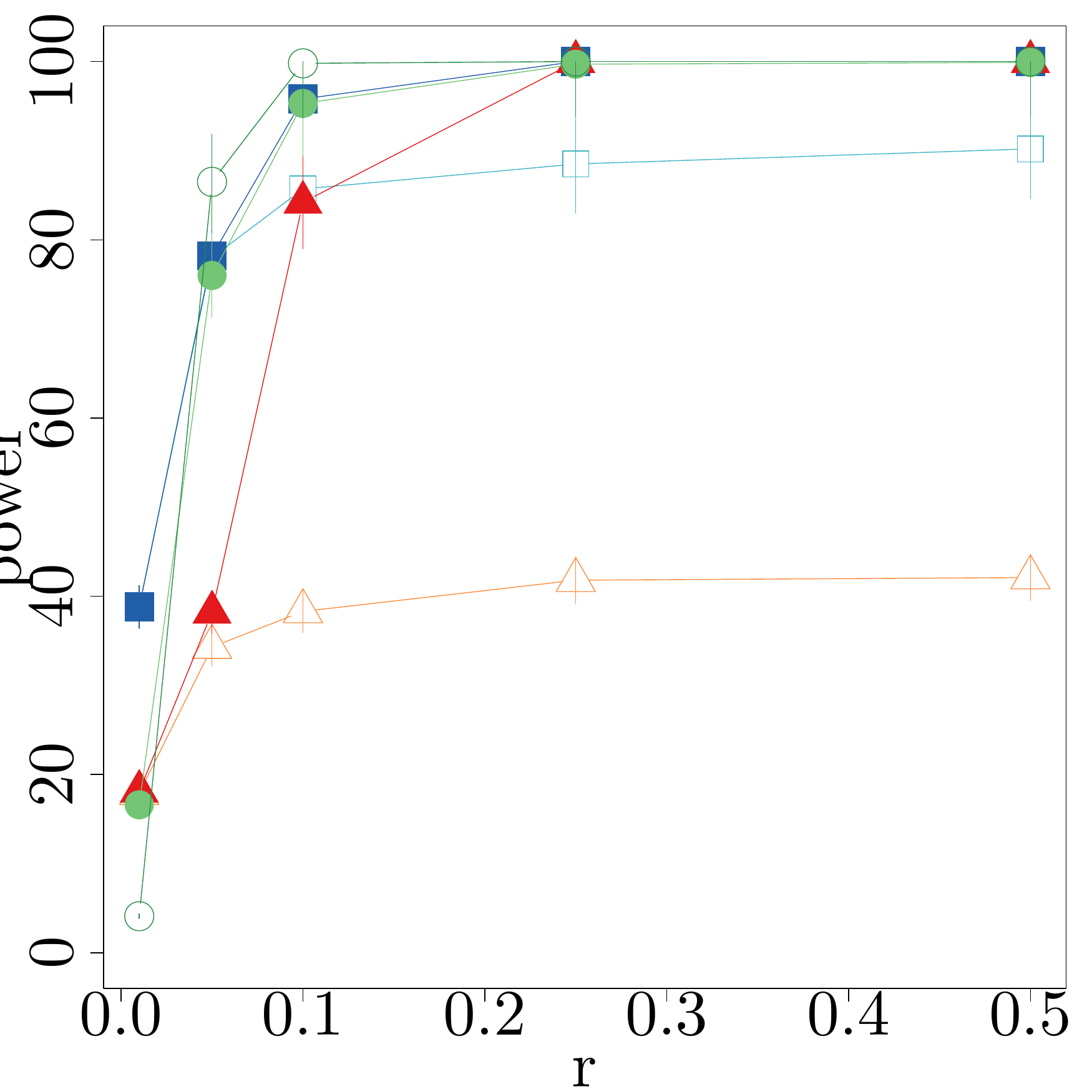} \\[-20ex]
 \raisebox{.3\textwidth}{4. \includegraphics[width=.1\textwidth]{data/settings_4.pdf}}&\includegraphics[width=.3\textwidth]{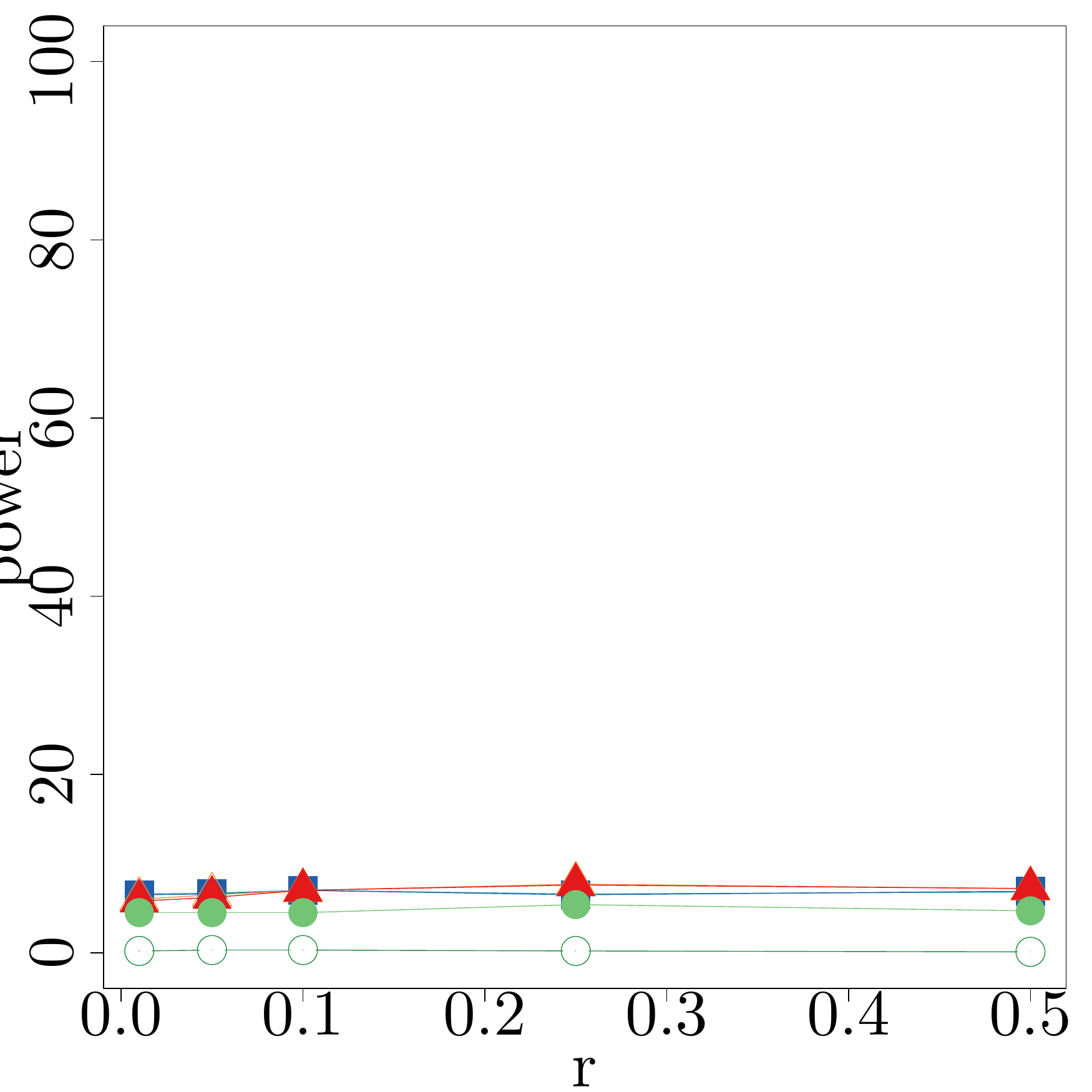}
 &
 \includegraphics[width=.3\textwidth]{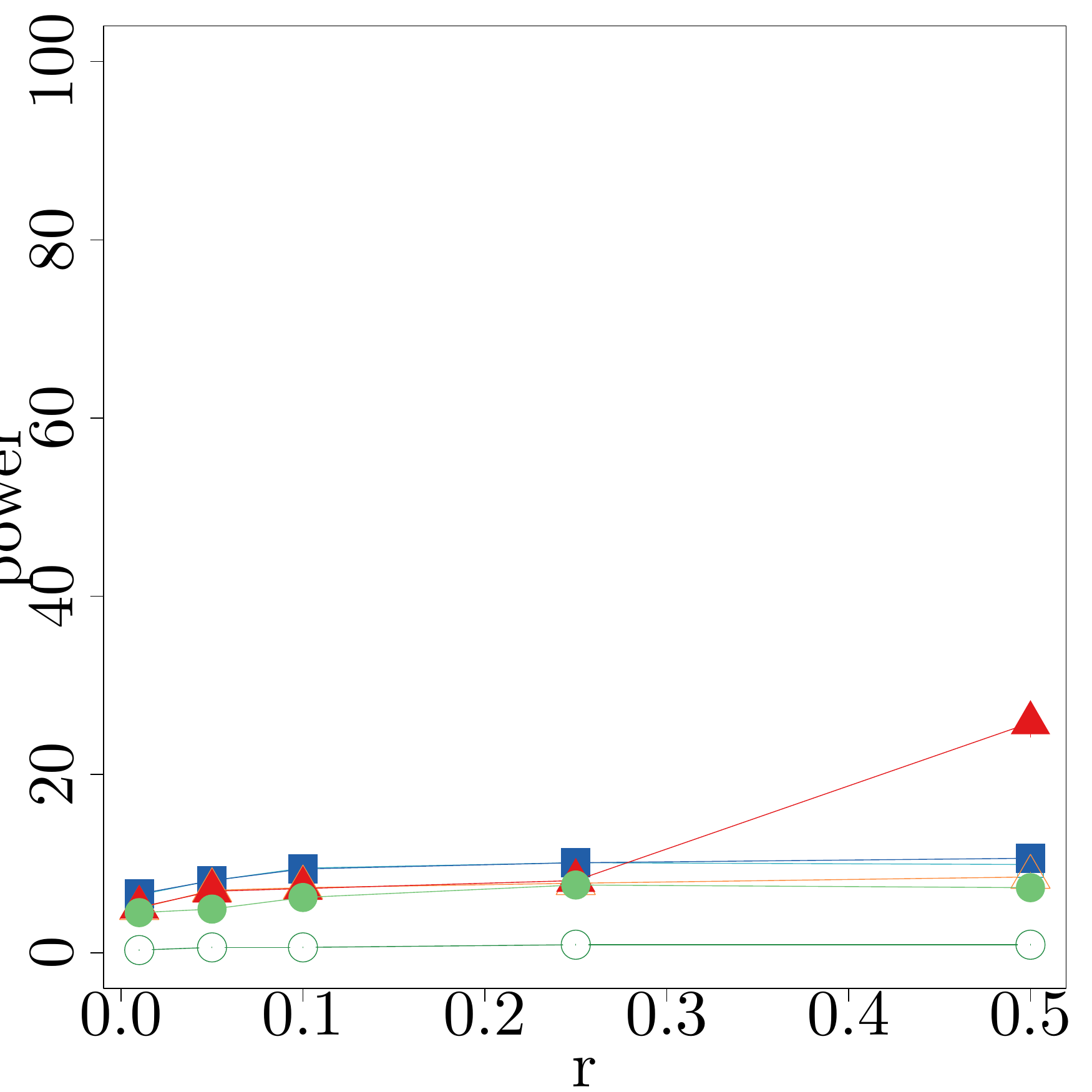}
 &
 \includegraphics[width=.3\textwidth]{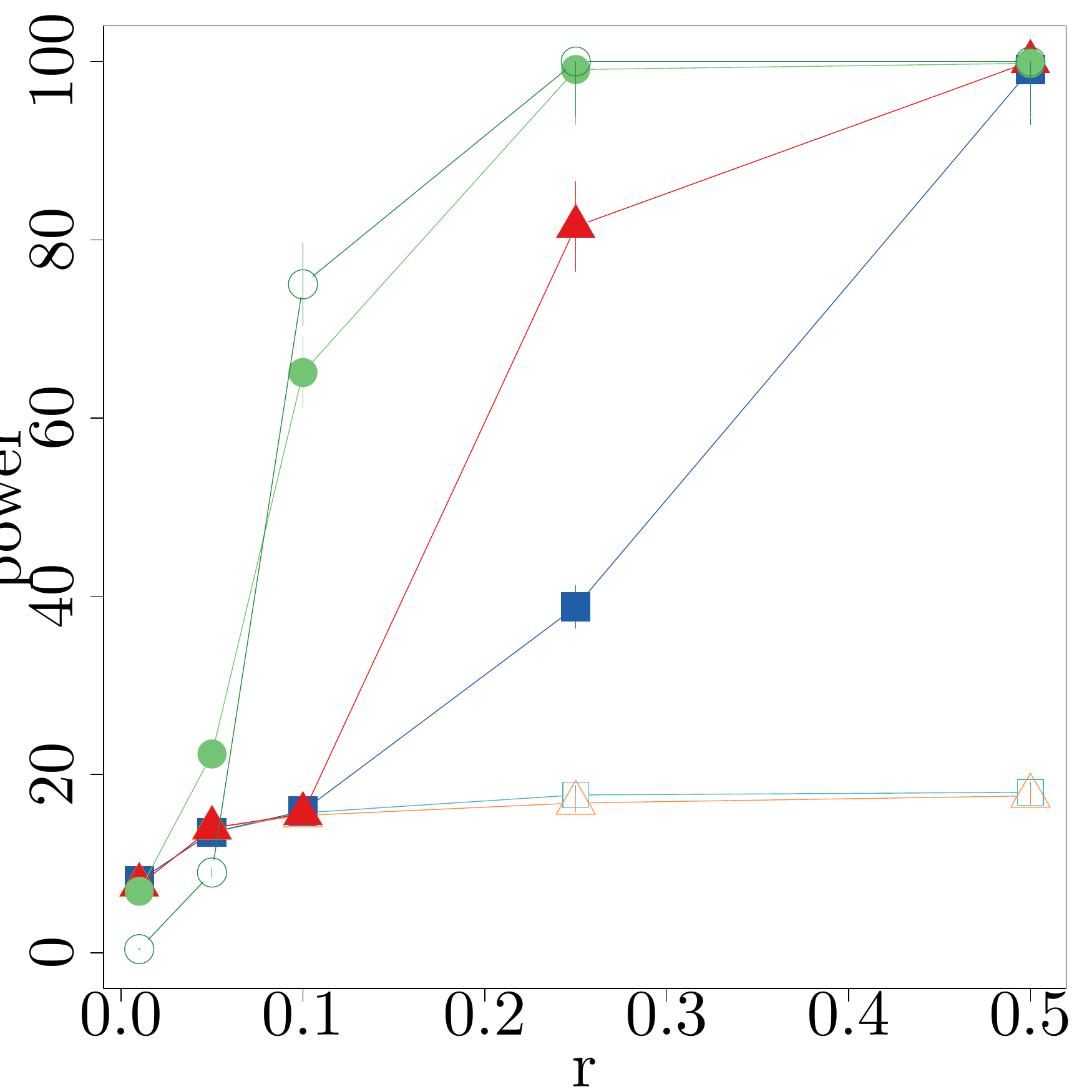}
 \\[-20ex]
 \end{tabular}\end{adjustwidth}
\caption{\label{fig:powerCLRvsFisher}Power (in percentage) as a
  function of signal magnitude parameter $r$ for
  various sparsity pattern under the assumption of uncorrelated
  designs $\Sigma^{(1)} = \Sigma^{(2)} = I_p$.
 Results for the suggested test  $T^{P}_{\widehat{\mathcal{S}}}$ and
  the test $T^{P,\mathrm{Fisher}}_{\widehat{\mathcal{S}}}$
, combined with $\cS_1$  or $\SLasso$ test collections.  Blue squares represent the suggested test
  $T^{P}_{\widehat{\mathcal{S}}}$, red triangles  stand for the Fisher test  $T^{P,\mathrm{Fisher}}_{\widehat{\mathcal{S}}}$. The deterministic collection
  $\cS_1$ is drawn in empty points, while the data-driven collection
  $\SLasso$ is in  plain points. 
 Results for the DiffRegr procedure are represented by green circles,
 respectively plain and empty for  single-splitting and 
multi-splitting approaches.}
\end{figure}

Figure \ref{fig:correlated} provide similar results
under respectively power decay correlated designs and GGM-like
correlated designs for a sample size of $n=50$, leading to similar
conclusions as in the uncorrelated case. 

\begin{figure}[!htbp]
 \begin{adjustwidth}{-1in}{-1in}
 \centering
 \begin{tabular}{m{1.5cm}ccc}
 Settings & Power decay & GGM \\
 \raisebox{.3\textwidth}{1. \includegraphics[width=.1\textwidth]{data/settings_1.pdf}}& \includegraphics[width=.3\textwidth]{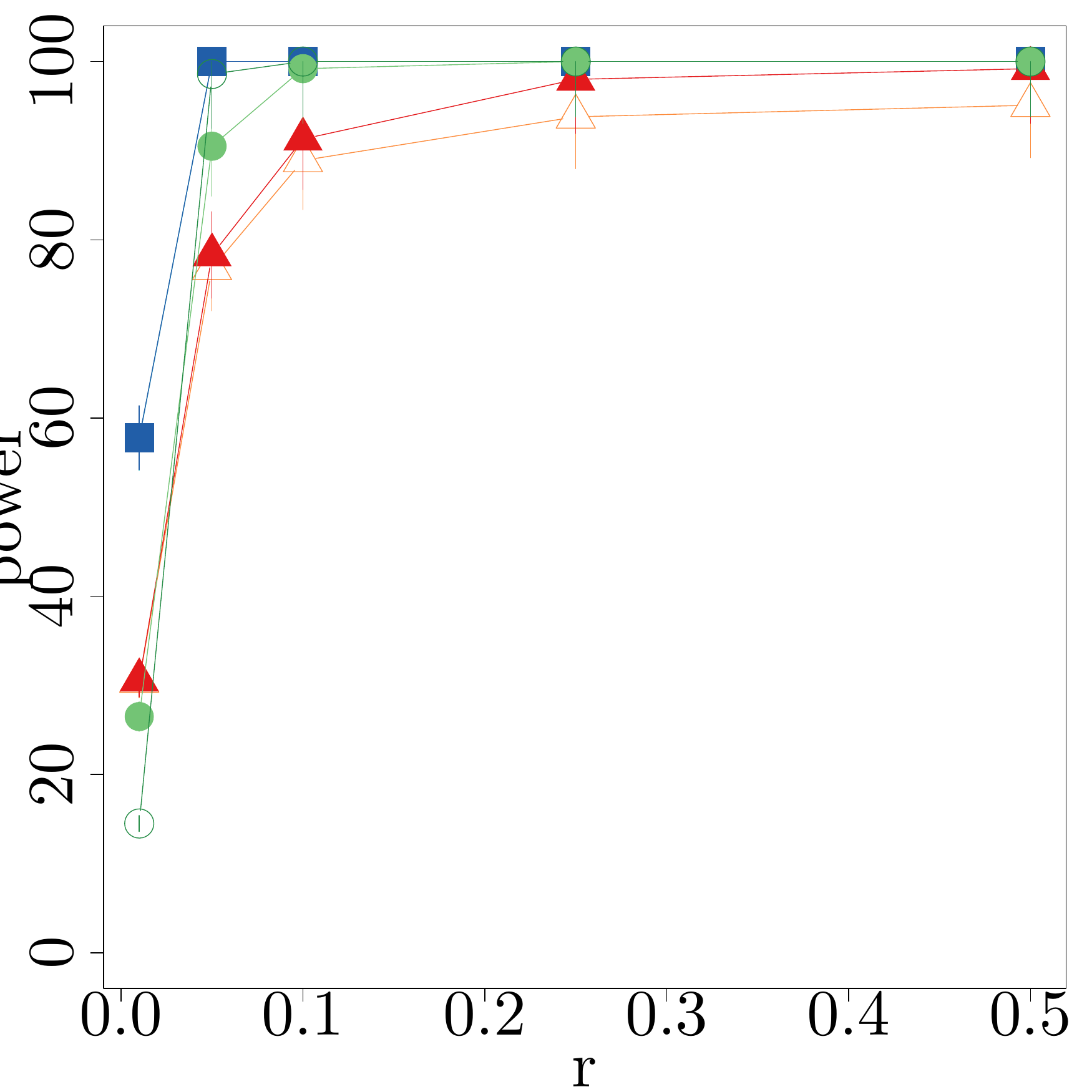}
 &
 \includegraphics[width=.3\textwidth]{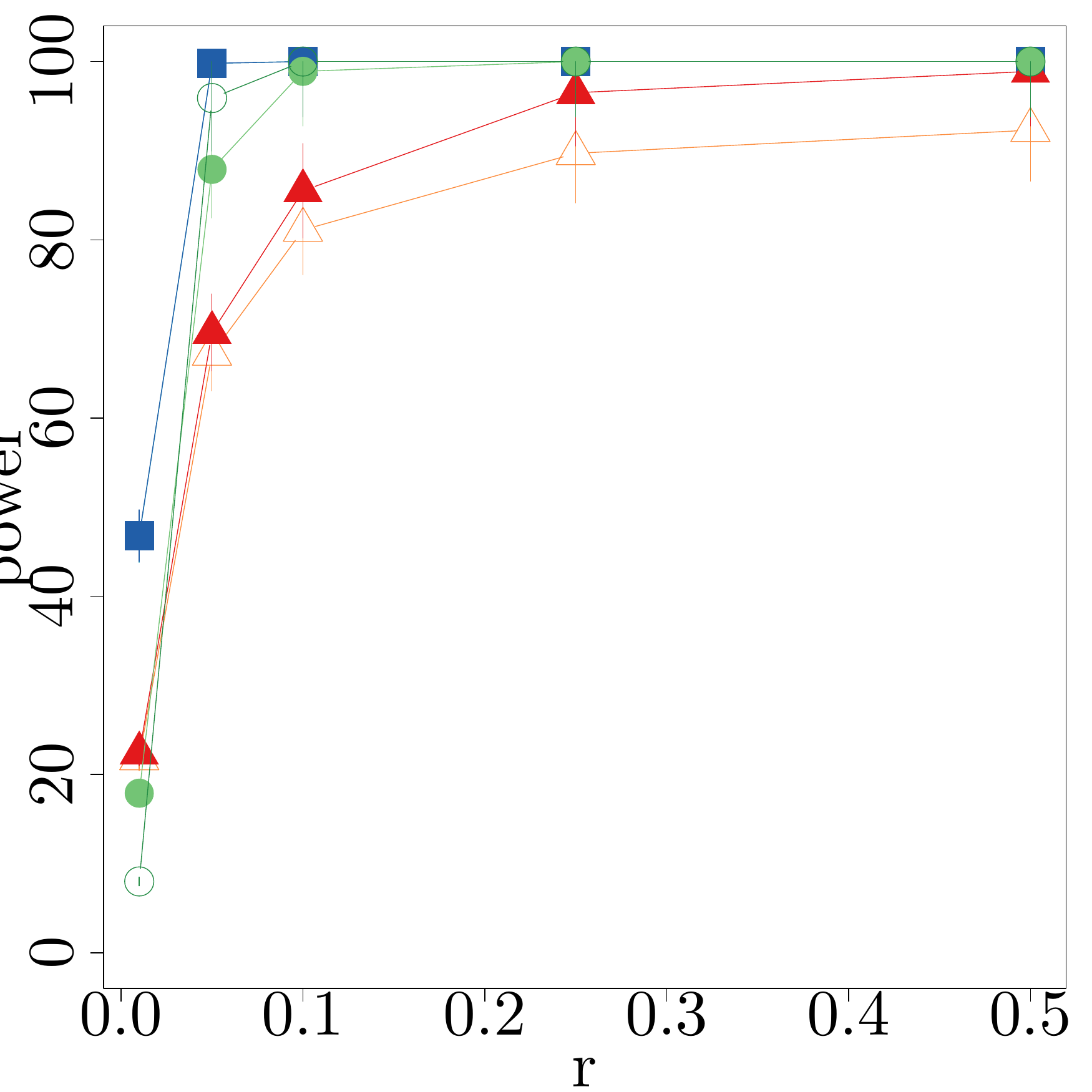} \\[-20ex]
 \raisebox{.3\textwidth}{2. \includegraphics[width=.1\textwidth]{data/settings_2.pdf}}& \includegraphics[width=.3\textwidth]{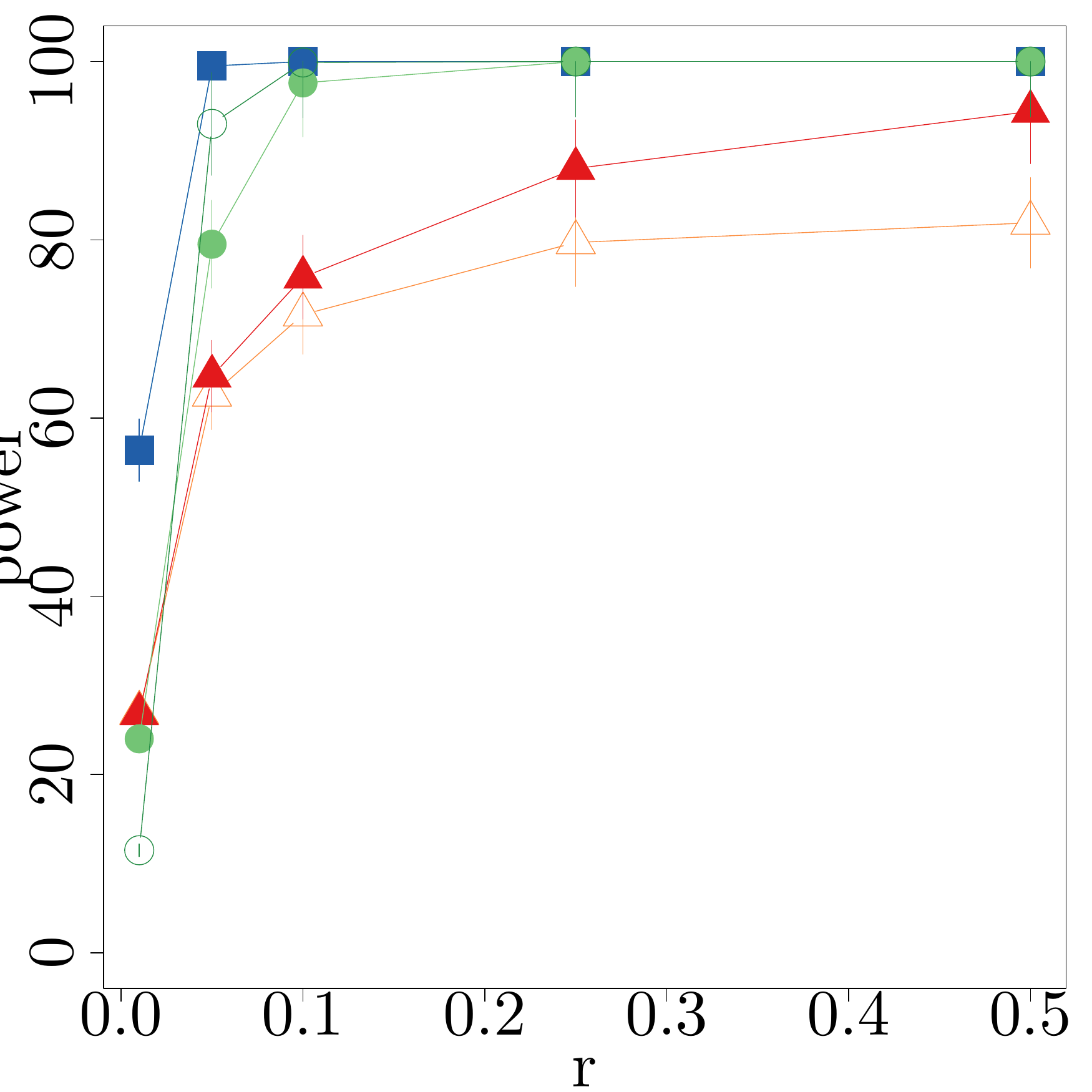}
 &
 \includegraphics[width=.3\textwidth]{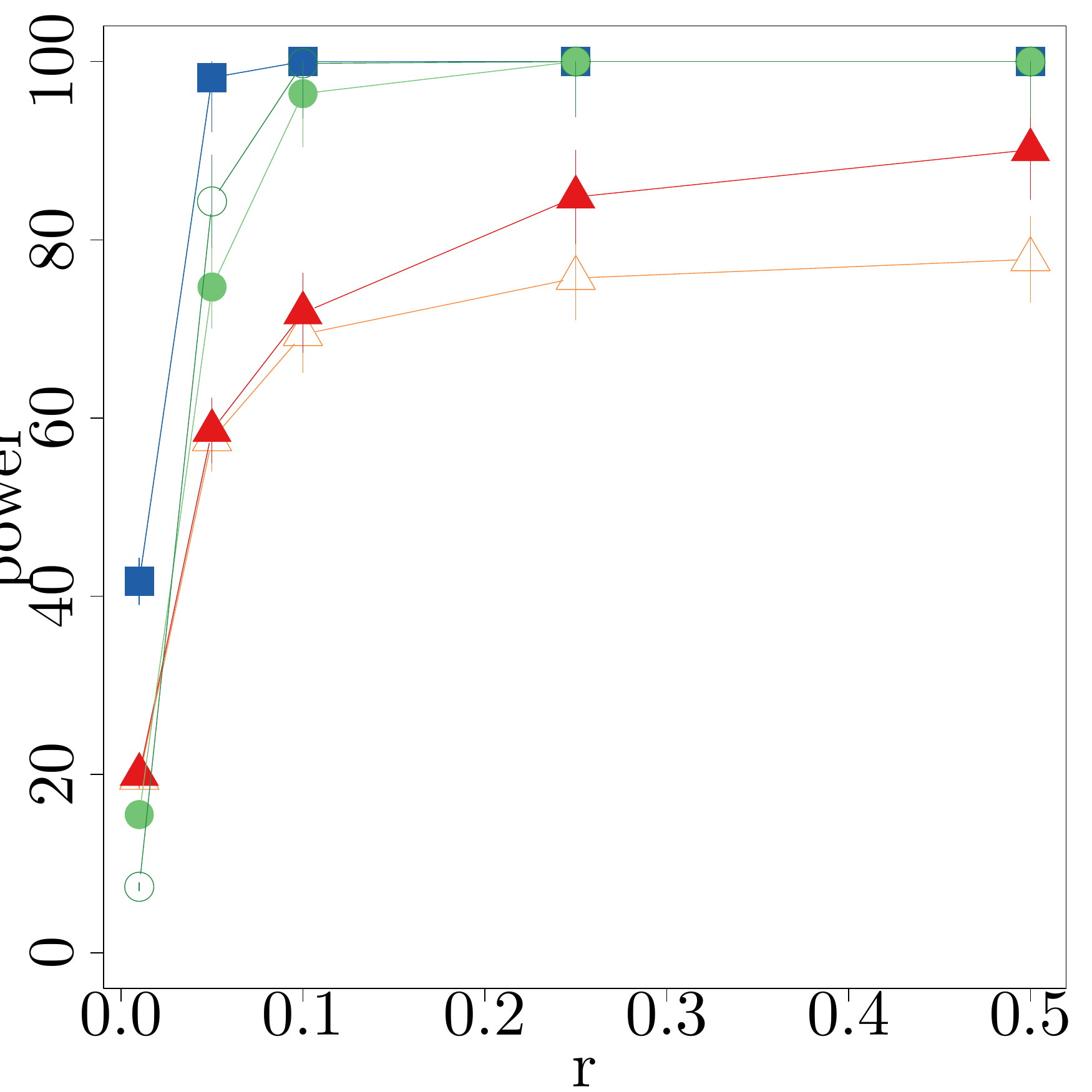} \\[-20ex]
 \raisebox{.3\textwidth}{3. \includegraphics[width=.1\textwidth]{data/settings_3.pdf}}& \includegraphics[width=.3\textwidth]{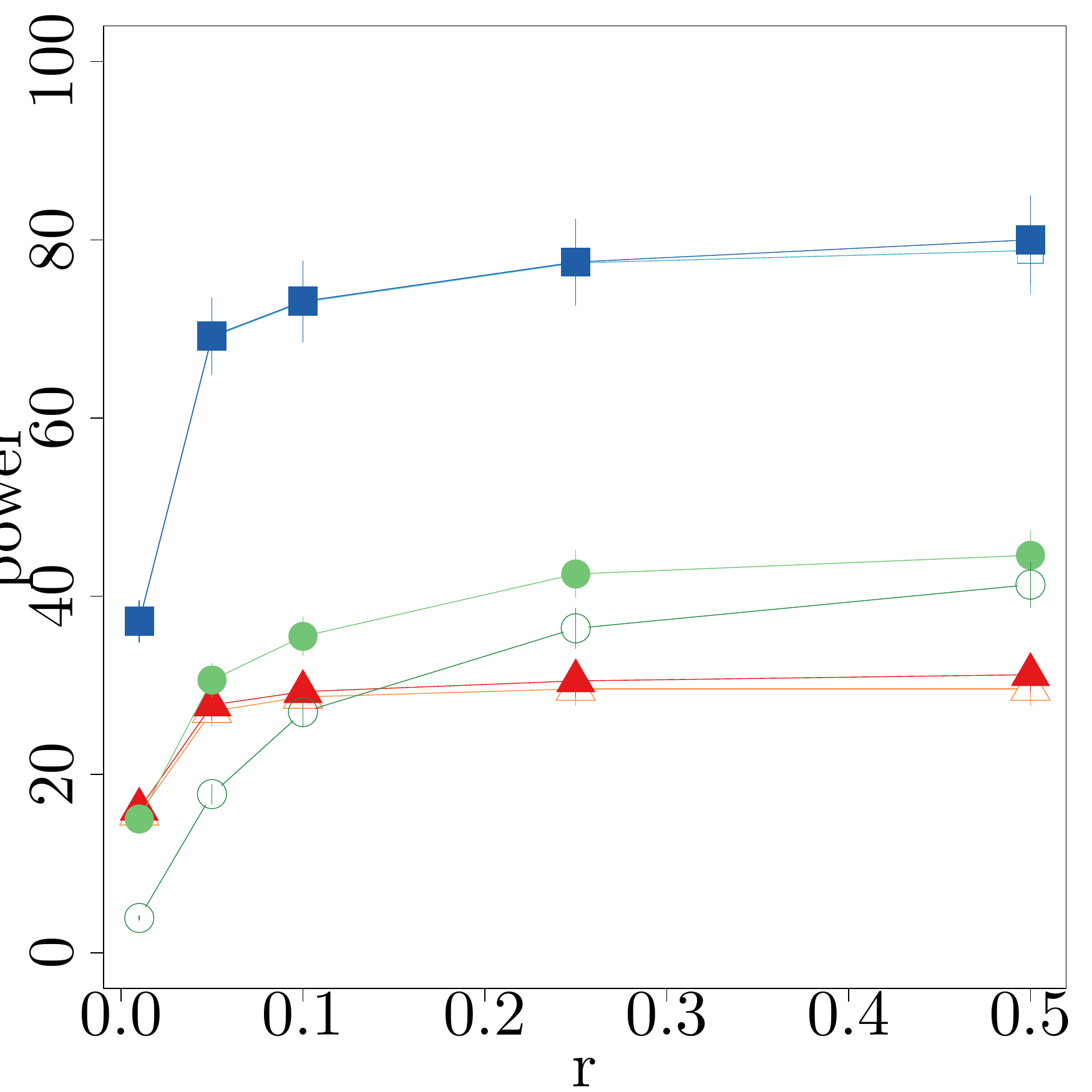}
 &
 \includegraphics[width=.3\textwidth]{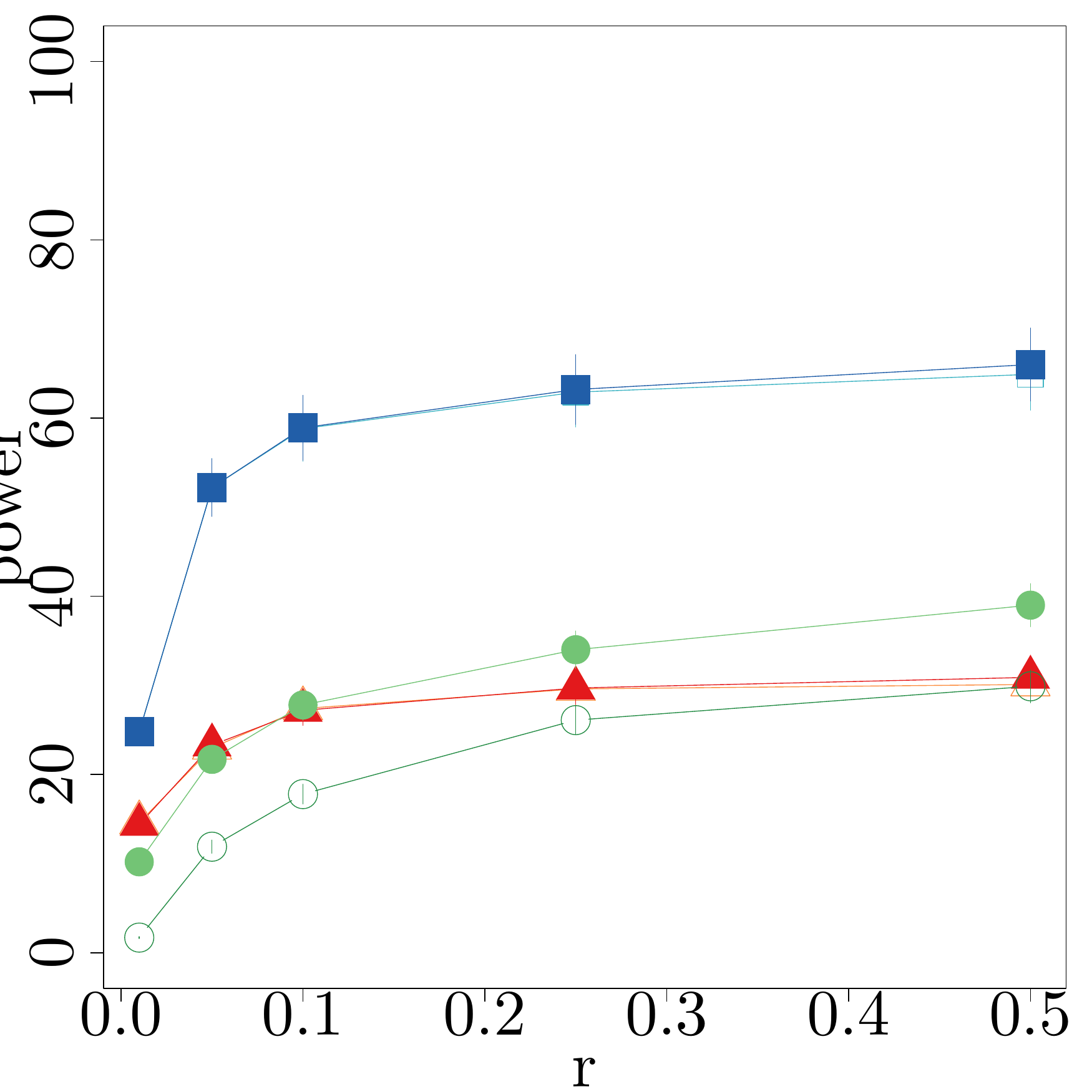} \\[-20ex]
 \raisebox{.3\textwidth}{4. \includegraphics[width=.1\textwidth]{data/settings_4.pdf}}& \includegraphics[width=.3\textwidth]{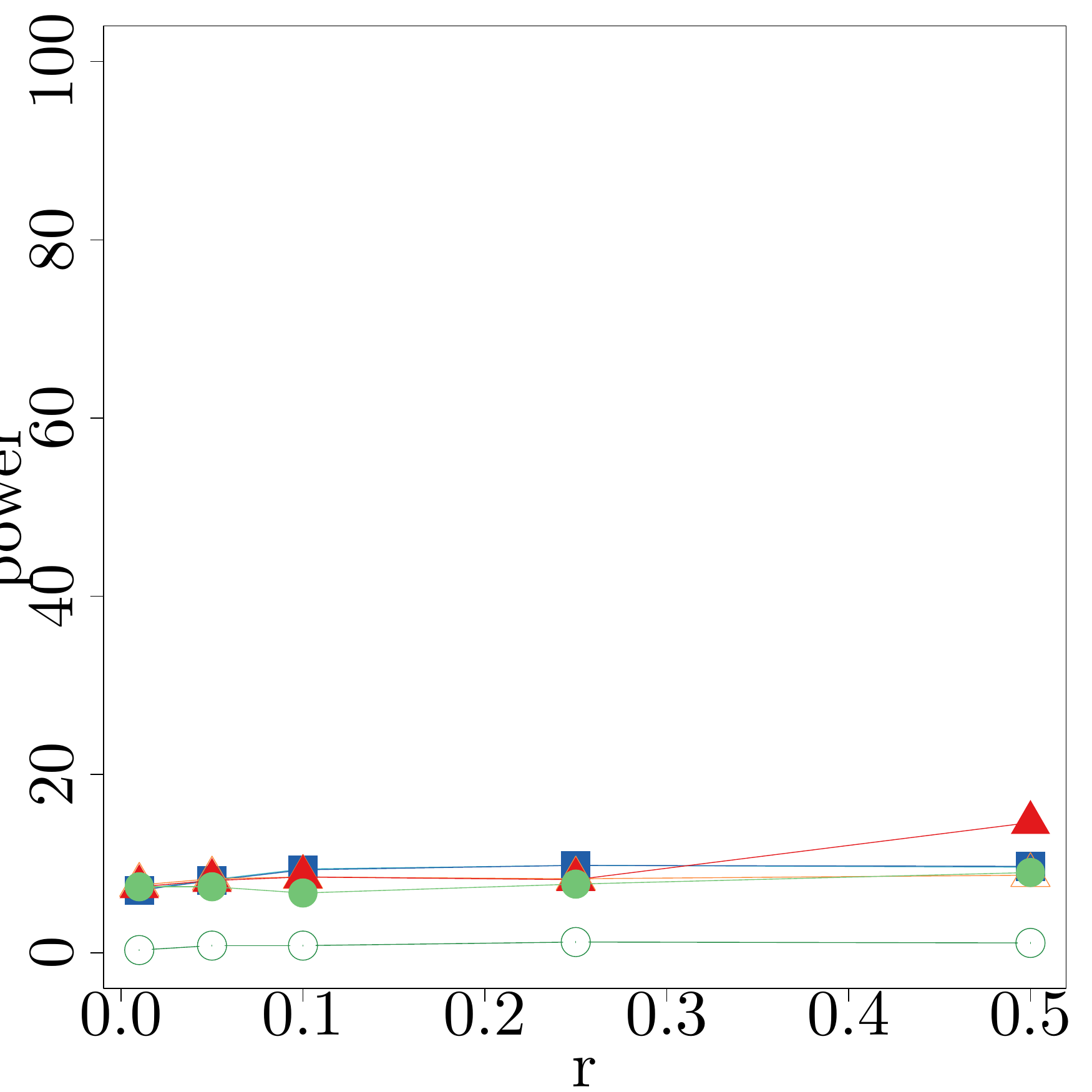}
 &
 \includegraphics[width=.3\textwidth]{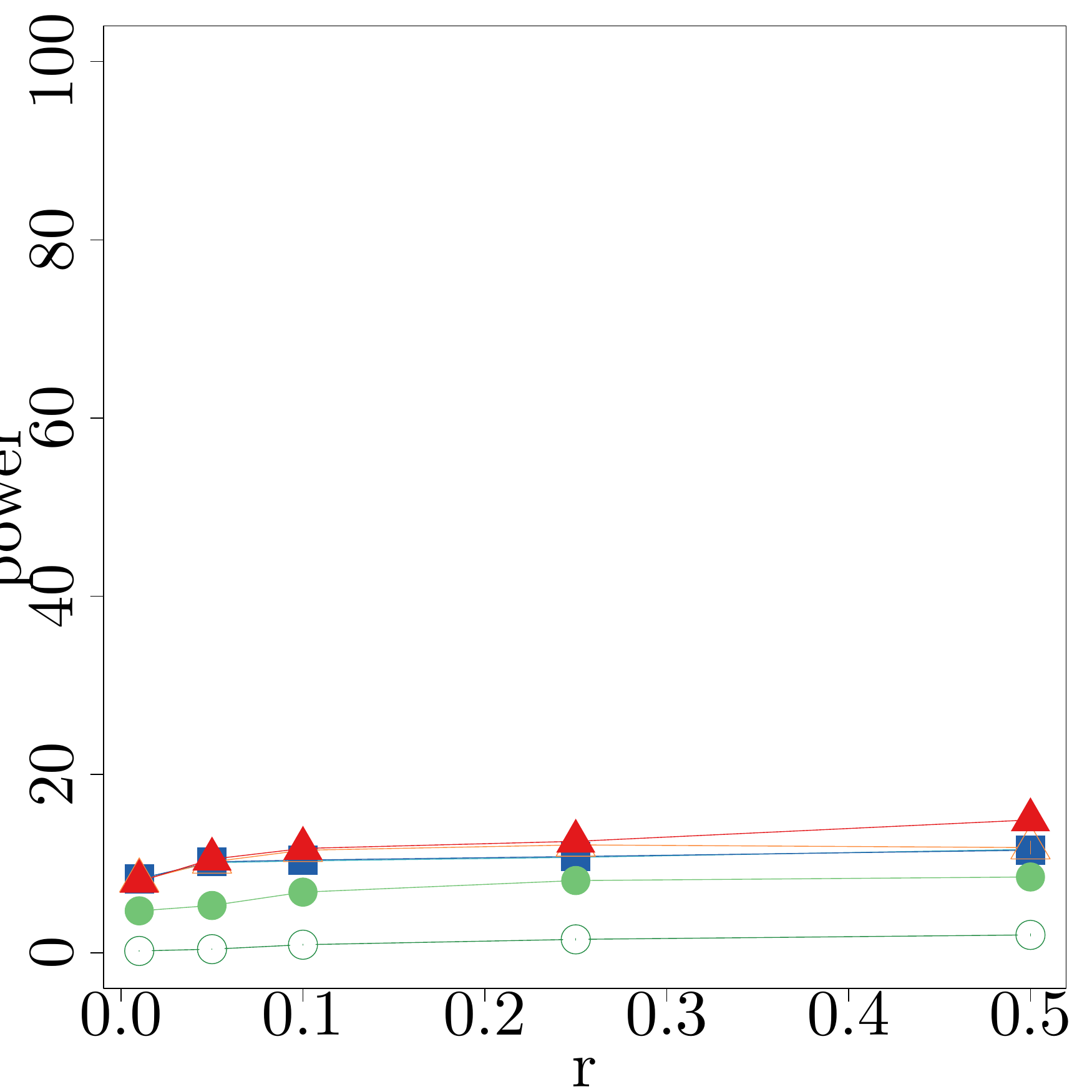} \\[-20ex]
 \end{tabular}\end{adjustwidth}
\caption{\label{fig:correlated}Power (in percentage) as a
  function of signal magnitude parameter $r$ for
  various sparsity patterns under power decay and GGM correlated
  designs, at $n=50$ observations. 
  Results for the suggested  test  $T^{P}_{\widehat{\mathcal{S}}}$ and
  the test $T^{P,\mathrm{Fisher}}_{\widehat{\mathcal{S}}}$
, combined with $\cS_1$  or $\SLasso$ test collections and a calibration by permutation. Blue squares represent the suggested test
  $T^{P}_{\widehat{\mathcal{S}}}$, red triangles
  stand for the Fisher test  $T^{P,\mathrm{Fisher}}_{\widehat{\mathcal{S}}}$. The deterministic collection
  $\cS_1$ is drawn in empty points, while the data-driven collection
  $\SLasso$ is in plain points. Results for the DiffRegr procedure are
  represented by green circles, 
 respectively plain and empty for  single-splitting and 
multi-splitting approaches.}
\end{figure}

\section{Application to GGM}\label{sec:GGM}

The following section explicits the extension of the two-sample linear regression testing framework to the two-sample Gaussian graphical model testing framework. 
We describe the tools and guidelines for a correct interpretation of the results and illustrate the approach on a typical two-sample transcriptomic data-set.

\subsection{How to Apply this Strategy to GGM Testing}

\paragraph{Neighborhood Selection Approach}
The procedure developed in Section \ref{sec:test} can
be adapted to the case of Gaussian graphical models as in
\cite{2009_CSDA_Verzelen}. We quickly recall why
estimation of the Gaussian graphical model amounts to the estimation of
$p$ independent linear regressions when adopting a neighborhood selection approach \cite{2006_AS_Meinshausen}.

Consider two Gaussian
random vectors $ {Z}^{(1)} \sim \mathcal{N}(0,\left[\Omega^{(1)}\right]^{-1})$ and $
{Z}^{(2)} \sim \mathcal{N}(0,\left[\Omega^{(2)}\right]^{-1})$. 
Their respective
conditional independence structures are represented by the graphs
$\mathcal{G}^{(1)}$ and $\mathcal{G}^{(2)}$, which consist of a common set of nodes
$\Gamma=\{1, \dots, p\}$ and their respective sets of edges
$\mathcal{E}^{(1)}$ and $\mathcal{E}^{(2)}$. When speaking of gene regulation networks, each node represents a gene, and edges between genes are indicative of potential regulations. 
In contrast with gene co-expression networks, edges in Gaussian graphical models do not reflect correlations but partial correlations between gene expression profiles. 

Formally, an edge $(i,j)$ belongs
to the edge set $\mathcal{E}^{(1)}$ (resp. $\mathcal{E}^{(2)}$) if $Z^{(1)}_i$
  (resp. $Z^{(2)}_i$) is independent from  $Z^{(1)}_j$
  (resp. $Z^{(2)}_j$) conditional on all other variables
  $Z^{(1)}_{\backslash i,j}$ (resp. $Z^{(2)}_{\backslash i,j}$). When
  the precision matrix $\Omega^{(k)}$ is nonsingular, the edges are characterized by its non zero
  entries. 

The  idea of  neighborhood selection  is to  circumvent  the intricate
issue of estimating the precision matrix by recovering the sets of edges
neighborhood by neighborhood, through the conditional distribution of
$Z^{(k)}_i$ given all remaining variables $Z^{(k)}_{\backslash
  i}$. Indeed, this distribution is again a Gaussian distribution, whose mean
is a linear combination of $Z^{(k)}_{\backslash i}$ while its variance
is independent from  $Z^{(k)}_{\backslash i}$. Hence, $Z^{(k)}_i$ can
be decomposed into the following linear regression:
\begin{equation}\label{eq:NS}
Z^{(k)}_i = \sum_{j \neq i} Z_j^{(k)} \beta_{ij}^{(k)} + \varepsilon_{i}^{(k)} = Z_{\backslash i}^{(k)} \beta_{i}^{(k)} + \varepsilon_{i}^{(k)}\ ,
\end{equation} 
where $\beta_{ij}^{(k)} =-\Omega_{ij}^{(k)}/\Omega_{ii}^{(k)}$ and  $\var[\varepsilon_{i}^{(k)}]=(\Omega_{ii}^{(k)})^{-1}$.
\medskip

Given an  $n_1$-sample of $Z^{(1)}$ and an  $n_2$-sample of $Z^{(2)}$,
we recall that our objective is to test as formalized in \eqref{eq:hypothese_GGM}
\[\hyp^{G}_0:\  \Omega^{(1)}=\Omega^{(2)}\quad \text{ versus }\quad 
\hyp^{G}_1:\  \Omega^{(1)}\neq\Omega^{(2)}\ .\]
As a result of Equation \eqref{eq:NS}, testing for the equality of the matrix rows $\Omega^{(1)}_{i.}=\Omega^{(2)}_{i.}$
 is equivalent to testing for   $\beta_{i}^{(1)}=  \beta_{i}^{(2)}$ and $\var[\varepsilon_{i}^{(1)}]=\var[\varepsilon_{i}^{(2)}]$.
Denote   by   $\Sigma^{(k)}_{\backslash    i}$   the   covariance   of
$Z^{(k)}_{\backslash i}$.  Under the  null $\hyp^{G}_0$, we  have that
for  any  $i$,  $\Sigma^{(1)}_{\backslash  i}=\Sigma^{(2)}_{\backslash
  i}$.  Consequently, we  can translate  the GGM  hypotheses  given in
Equation \eqref{eq:hypothese_GGM} 
into a conjunction of two-sample linear regression tests:
\begin{eqnarray}\label{eq:hypothese_NS}
\hyp^{G}_0:&&\  \bigcap_{i} \left[\beta_{i}^{(1)}= \beta_{i}^{(2)},\  \Omega_{ii}^{(1)}= \Omega_{ii}^{(2)},\ \Sigma^{(1)}_{\backslash i}=\Sigma^{(2)}_{\backslash i} \right]\\
\hyp^{G}_1:&&\  \bigcup_{i} \left[\beta_{i}^{(1)} \neq \beta_{i}^{(2)}\right]\cup\left[\Omega_{ii}^{(1)} \neq \Omega_{ii}^{(2)}\right] \ .\nonumber
\end{eqnarray}

Concretely, we apply the previous two-sample
linear regression model with $\mathbf{X}^{(1)} = \mathbf{Z}^{(1)}_{,\backslash i}$,$\mathbf{X}^{(2)}= \mathbf{Z}^{(2)}_{,\backslash i}$,
$\mathbf{Y}^{(1)}  = \mathbf{Z}^{(1)}_{,i}$,  and  $\mathbf{Y}^{(2)} =
\mathbf{Z}^{(2)}_{,i}$  for  every   gene  $i$  and  combine  multiple
neighborhood  tests using  a  Bonferroni calibration  as presented  in
Algorithm \ref{algo:GGMtesting}.  The equality of $\sigma^{(k)}$'s in $\hyp_0$
models the  equality of $\Omega_{ii}^{(k)}$'s in  $\hyp_0^G$ while the
equality   of   $\Sigma^{(k)}$'s   accounts   for  the   equality   of
$\Sigma_{\backslash i}^{(k)}$'s.

\begin{algorithm}
\caption{Gaussian Graphical Model Testing Strategy}\label{algo:GGMtesting}
\begin{algorithmic}
\Require{Data $\mathbf{Z}^{(1)}$,$\mathbf{Z}^{(2)}$, maximum model dimension $D_{max}$ and desired level $\alpha$}
\For{each gene $i=1,\dots,p$}
\Procedure{Neighborhood Test}{}
\State Define $\mathbf{X}^{(1)} = \mathbf{Z}^{(1)}_{,\backslash i},\quad \mathbf{X}^{(2)}= \mathbf{Z}^{(2)}_{,\backslash i}$
\State Define $\mathbf{Y}^{(1)}  = \mathbf{Z}^{(1)}_{,i}, \quad \mathbf{Y}^{(2)}  = \mathbf{Z}^{(2)}_{,i} $
\State Apply the Adaptive Testing Strategy of Algorithm
\ref{algo:overall} to $\mathbf{X}^{(1)}$,$\mathbf{X}^{(2)}$,$\mathbf{Y}^{(1)}$,$\mathbf{Y}^{(2)}$
\EndProcedure
\EndFor
\State Reject the global null hypothesis if at least one Neighborhood Test is rejected  at level $\alpha/p$
\end{algorithmic}
\end{algorithm}

\paragraph{Interpretation.}
Because    we   need   $\Omega_{ii}^{(1)}=    \Omega_{ii}^{(2)}$   and
$\Sigma_{\backslash  i}^{(1)}= \Sigma_{\backslash i}^{(2)}$  for every
neighborhood in the two-sample GGM null hypothesis $\hyp^N_0$ \eqref{eq:hypothese_NS}, 
the      assumptions     that      $\sigma^{(1)}=\sigma^{(2)}$     and
$\Sigma^{(1)}=\Sigma^{(2)}$ in  the two-sample linear  regression null
hypothesis $\hyp_0$ \eqref{eq:hypothese} are crucial for each neighorbood test to be interpreted correctly. 
As a result, only the global test can be strictly speaking interpreted in a statistically correct sense. 

However in practice, when the global null hypothesis is rejected, our construction of neighborhood tests provides helpful clues on the location of disruptive regulations. In particular, for each rejected neighorhood test $i$, one can keep track of
the rejected model $S_R^i$, retaining sensible information on
which particular regulations are most likely altered between samples.



\subsection{Illustration on Real Transcriptomic Breast Cancer Data}

We apply this strategy to the full (training and validation) breast cancer dataset studied by
\cite{2006_JCO_Hess} and \cite{2008_BMC_Natowicz}, whose training
subset was originally published in \cite{2005_CCR_Rouzier}. The full dataset
consists of microarray gene expression profiles from 133 patients with stage I-III breast cancer undergoing
preoperative chemotherapy. A majority of patients
(n=99) presented residual disease (RD), while 34 patients demonstrated a pathologic
complete response (pCR). The common objective of \cite{2006_JCO_Hess} and
\cite{2008_BMC_Natowicz} was to develop a predictor of complete
response to treatment from gene expression profiling.
In particular, \cite{2006_JCO_Hess} identified an optimal predictive subset of 30
probes, mapping to 26 distinct genes. 

\cite{2009_EJS_Ambroise} inferred
Gaussian graphical models among those 26 genes on each patient class
using weighted neighborhood selection. The corresponding graphs of
conditional dependencies for medium regularization are presented in
Figure \ref{fig:graphEJS2009}. Those two graphs happen to differ
dramatically from one another. 
The question we tackle is whether those differences
remain when taking into account estimation uncertainties.

\begin{figure}[!ht]
\centering
\begin{tabular}{cc}
Pathologic Complete Response (pCR) & Residual Disease (RD) \\[2ex]
\includegraphics[width=.4\textwidth]{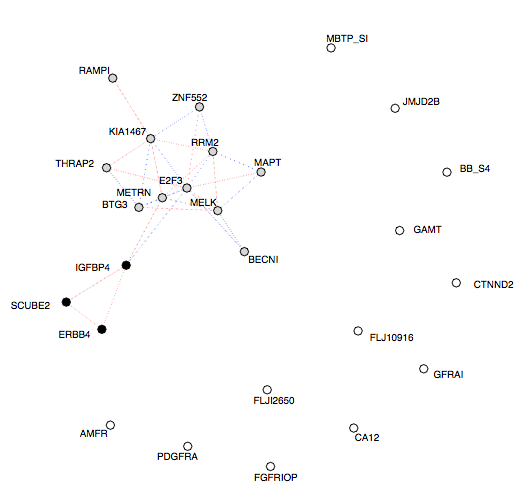}
& \includegraphics[width=.4\textwidth]{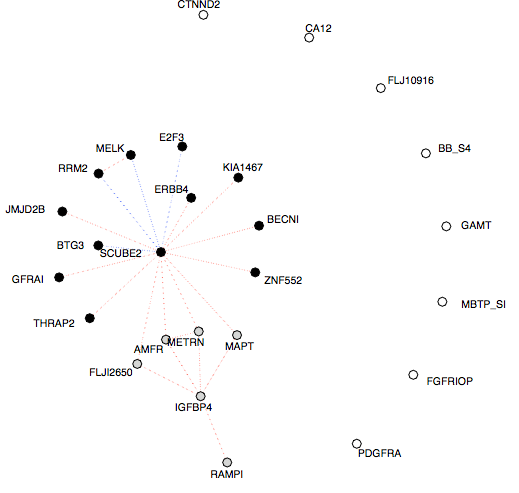}\\
\end{tabular}
\caption{\label{fig:graphEJS2009}Graphs of
conditional dependencies among the 26 genes selected by
\cite{2006_JCO_Hess} on patients with pathologic complete response or
residual disease with medium regularization as presented in
Figure 3 of \cite{2009_EJS_Ambroise}.}
\end{figure}

We run for each of the $p=26$ genes a neighborhood test $T_{\hat{S}_{Lasso}}^P$ at level
$0.05/26$. We associate to each neighborhood test the empirical p-value
defined in \ref{eq:neighborhood_pvalue} that has to be be compared to $\alpha/p$.

Most of the graph estimation methods proposed in the literature, such
as the procedure of \cite{2009_EJS_Ambroise} leading to
Figure  \ref{fig:graphEJS2009}, rely on the assumption that
observations are i.i.d. Yet the training and validation datasets have
been collected and analysed separately by two different clinical centers. We
therefore start by checking whether the pooled sample can be
considered as homogeneous. Within each group of patients (RD and pCR), we lead a test for the
homogeneity of Gaussian graphical models between the training and
validation subsets.

Within pCR patients (\ref{tab:Homo test pCR}), two
neighborhood tests corresponding to 
CA12 and PDGFRA are rejected at level
$0.05/26$. Within RD patients (\ref{tab:Homo test RD}), half of the
neighborhoods happen to differ significantly between the training and
validation datasets. Genes CA12 and JMJD2B are responsible for the
rejection of respectively seven and six neighborhoods.

\begin{table}[!ht]
\begin{center}
\begin{tabular}{l*{8}{c}} 
&             AMFR     &   BB\_S4    &   BECNI   &    BTG3  &      CA12
&CTNND2   &   E2F3 \\
\hline 
decision   &  0     &  0 &    0  &   0    & 1  &   0   &  0  \\   
$p^{empirical}$ & 0.0492 & 0.0072& 0.1972& 1& 0.0018& 0.0100& 0.1080 \\[2ex]

&ERBB4 &      FGFRIOP  &   FLJ10916   &FLJI2650 &   GAMT  &      GFRAI
&     IGFBP4  \\
\hline 
decision     0 &   0&     0   & 0  &   0   &  0   &  0    \\    
$p^{empirical}$ &0.5610 &0.0242& 0.2542& 0.0312& 0.1158& 0.5318& 0.0458 \\[2ex]
   
&JMJD2B    &   KIA1467 &    MAPT&        MBTP\_SI&     MELK    &
METRN     &PDGFRA\\ 
\hline 
decision   &0&  0&     0  &   0     &0   &  0  &  1   \\   
$p^{empirical}$ &  0.0128 &0.0272 &0.0178 &0.0062 &0.5602& 1& 0.0012  \\[2ex]

              &   RAMPI     &  RRM2  &      SCUBE2   &   THRAP2
              &ZNF552\\
\hline 
decision  &    0&  0 &    0&     0 &   0\\
$p^{empirical}$ & 0.0444& 0.0022& 0.2372& 0.0228 &0.0028\\[2ex]
\end{tabular}
\end{center}
\caption{\label{tab:Homo test pCR}Homogeneity test between training and test samples among
  pCR patients. Summary of test decisions after Bonferroni multiple testing correction
  and empirical p-values for each neighborhood test as defined in Section
  \ref{interpretation}.}
\end{table}

\begin{table}[!ht]
\begin{center}
\begin{tabular}{l*{8}{c}} 

  &AMFR     &   BB\_S4   &    BECNI &      BTG3     &   CA12
  &    CTNND2   &   E2F3   \\  
\hline
  decision    & 0 & 1  & 1 & 0 & 1 &  0 & 0 \\
  $p^{empirical}$  & 0.0046 & $<$0.0001  &  $<$0.0001& 0.0202 &   $<$0.0001 &
  0.0684 &    0.0428 \\ [2ex] 
  
  & ERBB4 &   FGFRIOP &    FLJ10916   & FLJI2650 &   GAMT      &
  GFRAI  &     IGFBP4 \\
\hline
  decision    & 0  &   1  &    1  &   1  &    1   &   0  &   0 \\
  $p^{empirical}$   &   0.26 & $<$0.0001   &  $<$0.0001 & 0.002   &  $<$0.0001   &  0.3606& 0.389 \\[2ex]

  & JMJD2B & KIA1467   &  MAPT&  MBTP\_SI   &  MELK  &
  METRN & PDGFRA \\
\hline
  decision   &      1 &    1 &     0 &    1 &     0 &    0 &    1\\
  $p^{empirical}$  & $<$0.0001 & 2e-04 & 0.006  & 6e-04 & 0.1556 & 0.1054  & $<$0.0001 \\[2ex]
  
  &   RAMPI  &     RRM2   &     SCUBE2    &  THRAP2   &    ZNF552 \\     
\hline
  decision   &   0  &   0&     0  &   1 &    1 \\
  $p^{empirical}$    &  0.2288  &  0.2988 &   0.3552   & $<$0.0001  &   $<$0.0001   \\
\end{tabular}
\end{center}
\caption{\label{tab:Homo test RD}Homogeneity test between training and test samples among
  RD patients. Summary of test decisions after Bonferroni multiple testing correction
  and empirical p-values for each neighborhood test as defined in Section
  \ref{interpretation} 
}
\end{table}

Because of these surprisingly significant divergences between training
and validation subsets, we restrict the subsequent analysis to the
training set (n=82 patients, among which 61 RD and 21 pCR patients).

To roughly check that we got rid of the underlying heterogeneity, we create an artificial dataset under $H_0$ by permutation of the
 patients, regardless of their class. No neighborhood test is rejected
 at a level corrected for multiple testing. We also cut the
group of patients with residual disease artificially in half. When testing for the difference between the two
halves, no significant heterogeneity remains, whatever the neighborhood.

Within the training set, the comparison of Gaussial graphical
structures between pCR and RD patients
leads to the rejection of all neighborhood tests after Bonferroni
correction for multiple testing of the 26 neighborhoods, as summarized
in Table \ref{tab:significant}. RRM2, MAPT
and MELK genes appear as responsible for the rejection of respectively
nine, nine and four of these neighborhood tests. Quite interestingly, these
three genes have all been described in clinical literature as new promising drug targets.
\cite{2007_CCR_Heidel} exhibited inhibitors of RRM2 expression,
which reduced \emph{in vitro} and \emph{in vivo} cell
proliferation. \cite{2005_PNAS_Rouzier} led functional biology experiments
validating the relationship between MAPT expression levels and
response to therapy, suggesting to inhibit its expression to increase
sensivity to treatment. More recently, \cite{2012_O_Chung} developed a therapeutic candidate
inhibiting MELK expression that was proved to suppress the growth of
tumour-initiating cells in mice with various cancer types, including breast cancer.

\begin{table}[!ht]
\begin{center}
  \begin{tabular}{l*{8}{c}} 
    &    AMFR & BB\_S4   &BECNI   &BTG3   &     CA12   & CTNND2&
    E2F3   \\
    \hline
    decision &    1 & 1 &    1 &    1 &        1    & 1 &   1 \\   
    $p^{empirical}$  &   $<$ 0.0001 &  $<$0.0001 &    4e-04 & $<$0.0001 &2e-04 &$<$0.0001 &   $<$0.0001    \\
    rejected model & RRM2 & RRM2  &  MAPT &   MAPT   &     MAPT  &  RRM2 &  MAPT   \\[2ex]
    
    &   ERBB4   &    FGFRIOP& FLJ10916&    FLJI2650& GAMT&
    GFRAI     & IGFBP4 \\
    \hline
    decision   &  1      &   1   &  1 &        1   &   1 &        1& 1\\       
    $p^{empirical}$   &  $<$0.0001    &    4e-04 & 4e-04 & $<$0.0001 & $<$0.0001
    & $<$0.0001  &$<$0.0001     \\
    rejected model & MELK  &      MAPT  &  RRM2     &   MAPT   &  RRM2  &
    BTG3 &MELK \\ [2ex]        
    
    & JMJD2B  & KIA1467& MAPT &MBTP\_SI &    MELK &
    METRN   &    PDGFRA    \\
    \hline
    decision   &  1  &  1&     1 & 1     &    1  &      1  &       1  \\
    $p^{empirical}$  &  $<$0.0001  &$<$0.0001 &$<$0.0001& $<$0.0001 & $<$0.0001 &$<$0.0001 & $<$0.0001 \\
    rejected model &        MAPT   &     MELK&    RRM2 &E2F3  &      MAPT&
MELK    &    RRM2    \\[2ex]

&     RAMPI &  RRM2 &SCUBE2 &     THRAP2 &ZNF552 \\
\hline
decision  &     1 &    1&  1 &    1&    1 \\    
$p^{empirical}$   &   $<$0.0001 &    $<$0.0001&  $<$0.0001     &
$<$0.0001  &2e-04 \\
rejected model&  RRM2  &  MAPT &BTG3&  E2F3 & RRM2  \\  
\end{tabular}
\end{center}
\caption{\label{tab:significant}Summary of neighborhood tests between RD and pCR patients within the
  training set (n=82). Decision is made at level $0.05/26$ to correct
  for multiple testing. The empirical $p$-value and the rejected model
  are defined in Section  \ref{interpretation}.}
\end{table}

For comprehensiveness, we add that similar analysis of the validation
set (n=51 patients, among which 38 RD
and 13 pCR patients) leads to the identification of only 9 significantly
altered neighborhoods between pCR and RD patients \ref{tab:significantTestSet}. This difference
in the number of significantly altered neighborhoods can be explained
by the reduced size of the sample. Yet, genes responsible for the
rejection of the tests differ from those identified on the training
set. In particular, five of the significant tests are rejected because of SCUBE2, which has been
recently recognised as a novel tumor suppressor gene \cite{2011_JBC_Lin}.

\begin{table}[!ht]
\begin{center}
  \begin{tabular}{l*{8}{c}} 
       &     AMFR     &   BB\_S4    &   BECNI  &     BTG3    &    CA12
       &CTNND2  &    E2F3  \\    
\hline
decision    & 0&     0&     0&     1&      0&    0&     1\\
 $p^{empirical}$   & 0.0024  &  0.0028 &   0.0048 &  0.0018 & 0.0028 &   0.0082  &  $<$0.0001 \\  
rejected model & - & - & - & SCUBE2   &     - &  - & METRN\\[2ex]

          &  ERBB4 & FGFRIOP & FLJ10916 & FLJI2650 & GAMT & GFRAI& IGFBP4 \\
\hline
decision  &        1&      0&     1&     0&     0&     1&      1\\    
 $p^{empirical}$  & 0.0014 & 0.0072 &   8e-04 &
0.0142 & 0.0046 &   8e-04 &2e-04\\
rejected model &SCUBE2 & - & SCUBE2 & - & - & E2F3 &SCUBE2 \\[2ex]
           &JMJD2B& KIA1467 &    MAPT  &      MBTP\_SI  &   MELK   &
           METRN   &   PDGFRA\\
\hline 
decision    &0 & 1&      0&     0&     0&     1&     0\\
 $p^{empirical}$   &0.0054 & 0.0018&    0.0032  &0.0078 &0.0036 &4e-04& 0.0104\\
rejected model &-  &SCUBE2& - & - & - & E2F3 & - \\        [2ex]
           &  RAMPI &      RRM2   &     SCUBE2&  THRAP2 &  ZNF552 \\
\hline       
decision   &    0&     0& 1  &   0& 0 \\    
 $p^{empirical}$   &  0.0056&  0.0034 &   2e-04 & 0.0024& 0.006 \\   
rejected model & - & - & FLJ10916    &- & -\\
\end{tabular}
\end{center}
\caption{\label{tab:significantTestSet}Summary of neighborhood tests between RD and pCR patients within the
  validation set (n=51). Decision is made at level $0.05/26$ to correct
  for multiple testing. The empirical $p$-value and the rejected model
  are defined in Section  \ref{interpretation}. }
\end{table}





\section{Discussion}\label{sec:discussion}

\paragraph{Design hypotheses.} In this work, we have made two main assumptions on the design matrices:
\begin{enumerate}
\item[(i)] The design matrices ${\bf X}^{(1)}$ and ${\bf X}^{(2)}$ are random. 
\item[(ii)] Under the null hypothesis \eqref{eq:hypothese}, we further suppose that the population covariances $\Sigma^{(1)}$ and $\Sigma^{(2)}$ are equal.
\end{enumerate}
Although this setting is particularly suited to consider the two-sample GGM testing (Section \ref{sec:GGM}), one may wonder whether one can circumvent these two restrictions. We doubt that this is possible without making the testing problem much more difficult.

First, the formulation \eqref{eq:hypothese} allows the null hypothesis to be interpreted as a relevant
intermediary case between two extreme fixed design settings:
design equality (${\bf X}^{(1)}={\bf X}^{(2)}$) and arbitrary
different design (${\bf X}^{(1)}\neq {\bf X}^{(2)}$).  In the first
case, the two-sample problem amounts to a one-sample problem by
considering $\tilde{\bf Y}={\bf Y}^{(1)}-{\bf Y}^{(2)}$ and it has
therefore already been thoroughly studied. The second case is on the
contrary extremely difficult as illustrated by the proposition below.
\begin{proposition}\label{prte:two_sample_fixed}
Consider the design matrices ${\bf X}^{(1)}$ and ${\bf X}^{(2)}$ as fixed and assume that $\sigma^{(1)}=\sigma^{(2)}=1$. If the  $(n_1+n_2)\times p$ matrix formed by ${\bf X}^{(1)}$ and ${\bf X}^{(2)}$ has rank $n_1+n_2$, then any test $T$ of $\beta^{(1)}=\beta^{(2)}$ vs $\beta^{(1)}\neq \beta^{(2)}$ based on the data $({\bf Y},{\bf X})$ satisfies:
\[\sup_{\beta\in\mathbb{R}^p}\mathbb{P}_{\beta,\beta}[T=1]+ \inf_{\beta^{(1)}\neq \beta^{(2)}\in\mathbb{R}^p}\mathbb{P}_{\beta^{(1)},\beta^{(2)}}[T=0]\geq 1\ ,\]
where $\mathbb{P}_{\beta^{(1)},\beta^{(2)}}(.)$ denotes the distribution of $({\bf Y}^{(1)},{\bf Y}^{(2)})$.
In other words, any level-$\alpha$ test $T$ has a type II error larger than $1-\alpha$, and this uniformly over $\beta^{(1)}$ and $\beta^{(2)}$. Consequently, any test in this setting cannot perform better than complete random guess. 
\end{proposition}

Furthermore, if $\Sigma^{(1)}\neq \Sigma^{(2)}$ is allowed  in the null \eqref{eq:hypothese}, then the two-sample testing problem becomes much more difficult  in the sense that it is impossible to reformulate the null hypothesis into a conjonction of low-dimensional hypotheses as done in 
 Lemma \ref{lemma_equivalence_hypothese}. Indeed, consider the following toy example: $\sigma^{(1)}=\sigma^{(2)}=1$, $\beta^{(1)}=\beta^{(2)}=(a,0,0,\ldots)^T$ for some $a>0$,  $\Sigma^{(1)}=I_p$ and $\Sigma^{(2)}=(\rho+ \mathbf{1}_{i=j})_{1\leq i,j\leq p}$ for some $\rho>0$. Then, for any subset $S$ that does not contain the first component, the parameters $\beta^{(1)}_S$ and $\beta^{(2)}_S$ are different. Consequently,  $\beta^{(1)}_S\neq \beta^{(2)}_S$  does not imply that $\beta^{(1)}\neq\beta^{(2)}$ and one should not rule out the parameter equality hypothesis relying on some low-dimensional regressions.

\paragraph{Comparison with related
  work~\cite{2013_ArXiv_Stadler,2014_ArXiv_Stadler}} 
St\"adler and Mukherjee propose a very general approach to
high-dimensional two-sample testing, being applicable to a wide range
of  models.  In  particular   this  approach  allows  for  the  direct
comparison of two-sample Gaussian graphical models without adopting a neighborhood selection
approach.  This  avoids the  burden  of  multiple neighborhood  linear
regression and the multiple testing correction
which follows. 

Because they estimate the supports of
sample-specific estimators and joint estimator separately in the
screening step, they resort
to an elegant estimation of the $p$-values for the non-nested likelihood
ratio test in the cleaning step. Yet, they do not provide any
theoretical controls on type I 
error rate or power for their overall testing strategy.

Finally, as it appears in the numerical experiments, their approach is
based on half-sampling and can  thus suffer from an acute reduction of
power on small samples. On  the bright side, the multi-split procedure
shows stable results and performs well even in difficult scenarios as soon as $n$ is sufficiently large.


\paragraph{Non asymptotic bounds and constants.}
In the spirit of \cite{2003_AS_Baraud}, our type II error analysis is completely non-asymptotic. However, the numerical constants involved in the bounds are clearly not optimal. Another line of work initiated by \cite{2004_AS_Donoho} considers an asymptotic but high-dimensional framework and aims to provide detection rates with optimal constants. For instance  \cite{2011_AS_Arias-Castro,2010_EJS_Ingster} have derived such results in the one-sample high-dimensional linear regression testing problem under strong assumptions on the design matrices. In our opinion, both analyses are complementary. While deriving sharp detection rates (under perhaps stronger assumptions on the covariance) is a stimulating open problem, it is beyond the scope of our paper.

\paragraph{Loss functions and interpretation.} The Kullback discrepancies considered in the power analysis of the
test depend on $\beta^{(1)}$ and $\beta^{(2)}$ through the prediction
distances $\|\beta^{(1)}~-~\beta^{(2)}\|_{\Sigma^{i}}$, $i=1,2$ rather
than the $l_2$ distance $\|\beta^{(1)}~-~\beta^{(2)}\|$. On the one hand, such a dependency on the prediction abilities is natural, as our testing procedures relies on the likelihood ratio.
 On the other hand, it is possible to characterize the power of our
 testing procedures as in Theorems \ref{cor:section_complete} and
 \ref{thrm:lasso} in terms of the distance
 $\|\beta^{(1)}-\beta^{(2)}\|$ by inverting $\Sigma^{(1)}$ and
 $\Sigma^{(2)}$ at $\beta^{(1)}-\beta^{(2)}$. However, the inversion
 would lead to an additional factor of the form $\Phi^{-1}_{|\beta^{(1)}-\beta^{(2)}|_0,-}(\sqrt{\Sigma^{(i)}})$ in the testing rates.

In terms of interpretation, even though our procedure adopts a global testing approach
through prediction distances, our real dataset example illustrates
that identifying which subset in the collection is responsible for
rejecting the null hypothesis provides clues into which specific
coefficients are most likely to differ between samples.

\paragraph{Gene network inference.}
Thinking of gene network inference by Gaussian graphical modeling, the high levels of correlations encountered within transcriptomic
datasets and the potential number of missing variables result in highly
unstable graphical estimations. Our global testing approach provides a way to
validate whether sample-specific graphs eventually share comparable
predictive abilities or disclose genuine structural changes. Such a statistical validation is obviously crucial before
translating any graphical analysis into further biological
experiments. Interestingly, the three main genes pointed out by our testing strategy
have been validated as promising therapeutic targets by functional
biology experiments. 

Finally, this test should also facilitate the validation of the fundamental
i.i.d. assumption across multiple samples, paving the way to pooled
analyses when possible. In that respect, we draw attention to the
significant heterogeneity detected between the
training and validation subsets of the
well-known Hess \emph{et al} dataset, suggesting that these samples
should be used separately as originally intended. Methods which
require i.i.d. observations should only be applied with caution to
this dataset if considered as a single large and homogeneous sample.

\section{Proofs}\label{sec:proofs}
\subsection{Two-sample testing for fixed and differents designs}

\begin{proof}[Proofs of Proposition \ref{prte:two_sample_fixed}]
Using the rank condition, we derive that for any vector $(a,b)$ in $\mathbb{R}^{n_1}\times \mathbb{R}^{n_2}$, there exists $\beta\in\mathbb{R}^p$ such that ${\bf X}^{(1)}\beta=a$ and ${\bf X}^{(2)}\beta=b$. Consequently, under the null hypothesis, $({\bf Y}^{(1)},{\bf Y}^{(2)})$ follows any distributions $\mathcal{N}(a,I_{n_1})\otimes \mathcal{N}(b,I_{n_2})$ with $(a,b)$ arbritary in $\mathbb{R}^{n_1}\times \mathbb{R}^{n_2}$. Hence, for any $\beta^{(1)}\neq \beta^{(2)}\in\mathbb{R}^p$, the distribution $\mathbb{P}_{\beta^{(1)},\beta^{(2)}}$ of $({\bf Y}^{(1)},{\bf Y}^{(2)})$ is not distinguishable from the null hypothesis. The result follows.
 
\end{proof}

\subsection{Upper bounds of the quantiles}

\begin{proof}[Proof of Proposition \ref{prte:majoration_Q}]
For the sake of simplicity, we note $N=n_1-|S|$, $(Z_1,\ldots,
Z_{|S|})$ a standard Gaussian random vector and $W_N$ a $\chi^2$
random variable with $N$ degrees of freedom. We apply Laplace method
to upper bound $ \mathbb{P}[F_{S,1}\geq u]$:
\begin{eqnarray*}
 \mathbb{P}[F_{S,1}\geq u]&= &\mathbb{P}\left[\sum_{i=1}^{|S|} a_iZ_i^2\geq u W_N/N \right]
\leq \inf_{\lambda>0}\mathbb{E}\exp\left[\lambda\sum_{i=1}^{|S|} a_iZ_i^2- \lambda uW_N/N\right]\\
&\leq & \inf_{0<\lambda<|a|_{\infty}/2}\exp\left[\psi_{u}(\lambda)\right]\ ,
\end{eqnarray*}
where
\[\psi_{u}(\lambda)= -\frac{1}{2}\sum_{i=1}^{|S|}\log(1-2\lambda a_i)-\frac{N}{2}\log\left(1+\frac{2\lambda u}{N}\right)\ .\]
The sharpest upper-bound is given by the value $\lambda^*$  which
minimizes $\psi_{u}(\lambda)$. We obtain an approximation of
$\lambda^*$ by cancelling the second-order approximation of its
derivative. Deriving $\psi_{u}$ gives
\[\psi'_{u}(\lambda)=\sum_{i=1}^{|S|}\frac{a_i}{1-2\lambda a_i} -\frac{u}{1+\frac{2\lambda u}{N}}\ ,\]
which admits the following second order approximation :
\begin{equation}\label{eq:derive_def_lambda}
|a|_1 + \frac{2\lambda \|a\|^2}{1-2|a|_{\infty}\lambda}-\frac{u}{1+\frac{2\lambda u}{N}}\ .
\end{equation} 

Cancelling this quantity amounts to solving a polynomial equation of the second degree.  The smallest solution of this equation leads to the desired $\lambda^*$.
\end{proof}

{\bf Additional Notations}. Given a subset $S$, $\Pi^{(1)}_{S}$ (resp. $\Pi^{(2)}_S$) stands for the orthogonal projection onto the space spanned by the rows of $\X^{(1)}_S$ (resp. $\X^{(2)}_S$). Moreover,  $\Pi^{(1)}_{S^{\perp}}$ denotes the orthogonal projection along the space spanned by the rows of $\X^{(1)}_{S}$.

\subsection{Distributions of $F_{S,V}, F_{S,1}$ and $F_{S,2}$ (Proposition \ref{prte:distribution_parametrique})}

Let us consider the regression of $Y^{(1)}$ (resp. $Y^{(2)}$) with respect to $X_{S}^{(1)}$ (resp. $X_{S}^{(2)}$):
\[
 { Y}^{(1)} =  {X}_S^{(1)}\beta_S^{(1)}+ \epsilon_S^{(1)}\ ,\quad \quad 
 {Y}^{(2)} =  {X}_S^{(2)}\beta_S^{(2)}+\epsilon_S^{(2)}\ .\]
Under the null hypothesis $\hyp_{0,S}$, we have $\beta_S^{(1)}=\beta_S^{(2)}$ and $\sigma^{(1)}_{S}=\sigma^{(2)}_{S}$. For the sake of simplicity, we write
$\beta_S$ and  $\sigma_{S}$ for these two quantities. Define the random variable $T_1$ and $T_2$ as 
\begin{equation}
T_1  =  \frac{\|\Pi^{(1)}_{S^{\perp}}
\eps^{(1)}_S\|^2}{(n_1-|S|)\sigma_{S}^{2}}\ ,\hspace{2cm}
T_2  =  \frac{\|\Pi^{(2)}_{S^{\perp}}
\eps^{(2)}_S\|^2}{(n_2-|S|)\sigma^{2}_{S}}\ .\label{eq:defi_T1}
\end{equation}
Given $\X$, $T_1/T_2$ follows a Fisher distribution with $(n_1~-~|S|,n_2~-~|S|)$ degrees of freedom. Observing that under the null hypothesis
\begin{eqnarray*}
 F_{S,V} = -2 +\frac{T_1}{T_2}\frac{n_2(n_1-|S|)}{n_1(n_2-|S|)}+ \frac{T_2}{T_1}\frac{n_1(n_2-|S|)}{n_2(n_1-|S|)}
\end{eqnarray*}
allows us to prove the first assertion of Proposition
\ref{prte:distribution_parametrique}. 
Let us turn to the second statistic:
\begin{eqnarray*}
 F_{S,1} = \frac{n_1}{n_2(n_1-|S|)}\frac{U}{T_1}\ ,
\end{eqnarray*}
where 
\begin{equation*}
U  =  \frac{\|\X_S^{(2)}(\X_S^{(2)\intercal}\X_S^{(2)})^{-1}\X_S^{(2)\intercal}\eps^{(2)}_S-
\X_S^{(2)}(\X_S^{(1)\intercal}\X_S^{(1)})^{-1}\X_S^{(1)\intercal}\eps^{(1)}_S\|^2}{\sigma_S^2}.
\end{equation*}
Given $\X$, $U$ is independent from $T_1$ since $T_{1}$ is a function of $\Pi^{(1)}_{S^{\perp}}
\eps^{(1)}_S$ while $U$ is a function of $(\eps_S^{(2)}, \Pi^{(1)}_{S} \eps^{(1)}_S)$. Furthermore, $U$ is the squared norm of a centered Gaussian vector with covariance 
\[ \X_S^{(2)}\left[(\X_S^{(1)\intercal}\X_S^{(1)})^{-1}+(\X_S^{(2)\intercal}\X_S^{(2)})^{-1}\right]\X_S^{(2)\intercal}\ .\]

\subsection{Calibrations}
\begin{proof}[Proof of Proposition \ref{prop:sizeB}]
By definition of the $p$-values $\widetilde{Q}_{i,|S|}$, we have under $\hyp_0$ for each $S\in\cS$ and each $i\in\{V,1,2\}$
\begin{eqnarray*}
\mathbb{P}_{\hyp_0}\left[\widetilde{Q}_{i,|S|}\left(F_{S,i}|\X_{S}\}\right)\leq \alpha_{i,S}|\X_S\right]\leq \alpha_{i,S}.
\end{eqnarray*}
Applying a union bound and integrating with respect to $\X$ allows us to control the type~I error:
\begin{eqnarray*}
 \mathbb{P}_{\hyp_0}[T^{B}_{\widehat{\cS}}= 1]
& = & \mathbb{E} \left[\sum_{S\in\widehat{\cS}}\sum_{i=V,1,2}\mathbb{P}\left[\widetilde{Q}_{i,|S|}\left(F_{S,i}|\X_{S}\}\right) <\alpha_{i,S}\right]\right]\\
&\leq& \sum_{S\in\cS}\sum_{i=V,1,2}\mathbb{P}\left[\widetilde{Q}_{i,|S|}\left(F_{S,i}|\X_{S}\}\right) <\alpha_{i,S}\right] \\
&\leq& \sum_{S\in\cS}\sum_{i=V,1,2}\mathbb{E}_{\X_S}\left[\mathbb{P}\left[ \widetilde{Q}_{i,|S|}\left(F_{S,i}|\X_{S}\}\right)<\alpha_{i,S}\right]\right]\leq \sum_{S\in\cS} \alpha_{i,S}\leq \alpha\ ,
\end{eqnarray*} 
where we have upper bounded the sum over the random collection $\mathcal{\cS}$ by the sum over $\cS$.
\end{proof}

\begin{proof}[Proof of Proposition \ref{prop:sizeP}]

Consider $i\in\{V,1,2\}$. Under $\hyp_0$, the distributions of 
\begin{eqnarray*}
\min_{S\in\widehat{\cS}_{\pi}}\left\{\widetilde{Q}_{V,|S|}\left(F_{S,V}(\pi)|\X_{S}^{\pi}\right)\binom{p}{|S|}\right\}\ , \\
 \min_{S\in\widehat{\cS}_{\pi}}\left\{\left(\widetilde{Q}_{1,|S|}\left(F_{S,1}(\pi)|\X_{S}^{\pi}\right)\bigwedge \widetilde{Q}_{2,|S|}\left(F_{S,2}(\pi)|\X_{S}^{\pi}\right)\right)\binom{p}{|S|}\right\}\ 
\end{eqnarray*}
are invariant with respect to the permutation $\pi$. Hence, we derive 
\begin{eqnarray*}
\mathbb{P}_{\hyp_0}\left[\left. \min_{S\in\widehat{\cS}}\overline{Q}_{V,|S|}\left(F_{S,V}|\X_{S}\right)\binom{p}{|S|} \leq \widehat{C}_{1}\right|\X_{S} \right] =\alpha/2\ ,\\
\mathbb{P}_{\hyp_0}\left[\left. \min_{S\in\widehat{\cS}_{\pi}}\left\{\left(\widetilde{Q}_{1,|S|}\left(F_{S,1}(\pi)|\X_{S}^{\pi}\right)\bigwedge \widetilde{Q}_{2,|S|}\left(F_{S,2}(\pi)|\X_{S}^{\pi}\right)\right)\binom{p}{|S|}\right\}\leq \widehat{C}_{2}\right|\X_{S} \right] =\alpha/2\ .
\end{eqnarray*}
Applying a union bound and integrating with respect to $\X$ allows us to conclude.
\end{proof}

\subsection{Proof of Theorem \ref{thrm_puissance}}

The objective is to exhibit a subset for which the power of
$T_\cS^B$ is larger than $1-\delta$. This subset is such that the
distance between the two sample-specific distributions is large enough
that we can actually reject the null hypothesis with large
probability. As exposed in Theorem \ref{thrm_puissance}, we rely on the semi-distances $\mathcal{K}_1(S)+ \mathcal{K}_2(S)$ for $S\in\cS$:
\begin{align}
\label{distance_decomp} 2(\mathcal{K}_1(S)+ \mathcal{K}_2(S))  = &
 \left(\frac{\sigma^{(1)}_{S}}{\sigma^{(2)}_{S}}\right)^2+
 \left(\frac{\sigma^{(2)}_{S}}{\sigma^{(1)}_{S}}\right)^2- 2   +
 \frac{\|\beta_S^{(2)}-\beta_S^{(1)}\|_{\Sigma^{(2)}}^2}{(\sigma_{S}^{(2)})^2}
  + \frac{\|\beta_S^{(2)}-\beta_S^{(1)}\|_{\Sigma^{(1)}}^2}{(\sigma_{S}^{(1)})^2}\ .
\end{align}

The proof is split into five main lemmas. First, we upper bound
$\widetilde{Q}_{V,|S|}^{-1}(x|{\bf X}_S)$,
$\widetilde{Q}_{1,|S|}^{-1}(x|{\bf X}_S)$, and
$\widetilde{Q}_{2,|S|}^{-1}(x|{\bf X}_S)$ in Lemmas
\ref{lemma_fisher}, \ref{lemma:upper_bound_q2}
and \ref{lemma:upper_bound2_q2}. Then, we control the deviations of
$F_{S,V}$, $F_{S,1}$, and $F_{S,2}$ under $\hyp_{1,S}$ in Lemmas
\ref{lemma_fisher2} and \ref{lemma_puissance_deuxieme}.
In the sequel, we call $\S'$ the subcollection of $\S$ made of subsets $S$ satisfying $|S|\leq (n_1\wedge n_2)/2$ and 
\begin{equation}\label{eq:thrm_puissance_bonne_hypothese}
 \log(12/\delta)<L_1^\bullet (n_1\wedge n_2),\quad   \log(1/\alpha_S)\leq L_2^{\bullet}(n_1\wedge n_2)\ ,\quad  |S|\leq L_3^{\bullet}
\end{equation}
where the numerical constants $L_1^\bullet$, $L_2^\bullet$, and $L_3^\bullet$ only depend on $L^*_2$ in \eqref{eq:lower_bound} and on the constants introduced in Lemmas  \ref{lemma_fisher}--\ref{lemma_puissance_deuxieme}.
These conditions allow us to fix the constants  in the statement (\ref{eq:H1}) of Theorem \ref{thrm_puissance}.

\begin{lemma}[Upper-bound of $\widetilde{Q}_{V,|S|}^{-1}(x|{\bf X}_S)$ ]\label{lemma_fisher}
There exists a positive universal constant $L$ such that the following holds.
Consider some $0<x<1$ such that $ 16\log(2/x)\leq n_1 \wedge n_2 \
.$ For any subset $S$ of size smaller than $(n_1\wedge n_2)/2$, we have
\begin{eqnarray}\label{Majoration_premier_terme}
\widetilde{Q}_{V,|S|}^{-1}(x|{\bf X}_S)\leq  L
\left\{\left(\frac{|S|(n_1-n_2)}{n_1n_2}\right)^2 +
\log(2/x)\left(\frac{1}{n_1}+
\frac{1}{n_2}\right)\right\} \ .
\end{eqnarray}
\end{lemma}

We recall that  $a=(a_1,\ldots , a_{|S|})$ denotes the positive eigenvalues of \[\frac{n_1}{n_2(n_1-|S|)}\X_S^{(2)}\left[(\X_S^{(1)\intercal}\X_S^{(1)})^{-1}+(\X_S^{(2)\intercal}\X_S^{(2)})^{-1}\right]\X_S^{(2)\intercal}\ .\]

\begin{lemma}[Upper-bound of $\widetilde{Q}_{1,|S|}^{-1}(x|{\bf X}_S)$ ]\label{lemma:upper_bound_q2}
There exist two positive universal constants $L_1$ and $L_2$ such that the following holds.
If $|a|_1 < u\leq (n_1-|S|)|a|_{\infty}$ and if $|S|\leq L_1n_1$,
\begin{eqnarray*}
 \log\left[\widetilde{Q}_{1,|S|}(u|\X_S)\right]\leq -\frac{(u-|a|_1)^2}{4\left[|a|_{\infty}(u-|a|_1)+\|a\|^2\right]}+  \frac{(u-|a|_1)u^3}{2(n_1-|S|)\left[|a|_{\infty}(u-|a|_1)+\|a\|^2\right]^2}\ .
\end{eqnarray*}
For any $0<x<1$, satisfying 
\begin{equation}\label{eq:condition_majoration_q2-1}
L_2\log(1/x)\leq n_1-|S|\ , 
\end{equation}
we have the following upper bound 
\begin{equation}\label{eq:upper_bound_q2-1}
\widetilde{Q}_{1,|S|}^{-1}(x|{\bf X}_S)\leq |a|_{\infty}\left[2|S|+2\sqrt{2|S|\log(1/x)}+8\log(1/x)\right]\ . 
\end{equation}
\end{lemma}

\begin{lemma}[Upper-bound of $|a|_{\infty}$]\label{lemma:upper_bound2_q2}
There exist two positive universal constants $L_1$ and $L_2$ such that the following holds.
Consider $\delta$ a positive number sastifying $L_1\log(4/\delta) <n_1\wedge n_2$.
With probability larger than $1-\delta/2$, we have 
\begin{eqnarray*}
 |a|_{\infty}&\leq& L_2\left[\frac{1}{n_2}+ \frac{\varphi_{\max}\left\{\sqrt{\Sigma_S^{(2)}}(\Sigma_S^{(1)})^{-1}\sqrt{\Sigma_S^{(2)}}\right\}}{n_1}\right]\ .
\end{eqnarray*}
\end{lemma}


\begin{lemma}[Deviations of $F_{S,V}$]\label{lemma_fisher2}
There exist three positive universal constants $L_1$, $L_2$ and $L_3$ such that the following holds.
Assume that $L_1\log(1/\delta)\leq n_1\wedge n_2$. With probability larger than $1-\delta$, we have
\begin{eqnarray}\label{Sajoration_puissance_fisher}
F_{S,V} \geq
L_2\left(\frac{(\sigma_{S}^{(1)})^2-(\sigma^{(2)}_{S})^{2}}{\sigma_{S}^{(1)}\sigma^{(2)}_{S}}\right)^2
-L_3\left[|S|^2\left(\frac{1}{n_1^2}+\frac{1}{n_2^2} \right)+ \log\left(\frac{1}{\delta}\right)\left(\frac{1}{n_1}
+\frac{1}{n_2}\right)\right]\ .
\end{eqnarray}

\end{lemma}

\begin{lemma}[Deviations of $F_{S,1}$]\label{lemma_puissance_deuxieme}
There exist two positive universal constants $L_1$ and  $L_2$ such that the following holds.
Assume that 
\begin{equation}\label{eq:condition_lemma_puissance_deuxieme}
L_1\log(12/\delta)<n_1\wedge n_2\ . 
\end{equation}
With probability larger than $1-\delta/2$, we have
\begin{eqnarray}\label{Sajoration_puissance_deuxieme_terme}
F_{S,1}\geq
\frac{\|\beta_S^{(2)}-\beta_S^{(1)}\|_{\Sigma^{(2)}}^2}{8(\sigma_{S}^{(1)})^2} - \log\left(6/\delta\right)L_2\left[\frac{1}{n_2}\frac{(\sigma^{(2)}_{S})^2}{(\sigma^{(1)}_{S})^2}+\frac{\varphi_{S}}{n_1}\right]\ ,
\end{eqnarray}
where $\varphi_S$ is defined in (\ref{eq:definition_varphiS}).
\end{lemma}

Consider some $S\in\S'$. Combining Lemmas \ref{lemma_fisher} and
\ref{lemma_fisher2},  we derive that $\widetilde{Q}_{V,|S|}(F_{S,V}|{\bf X}_S)\leq \alpha_S$ holds with probability larger than $1-\delta$ if 
\begin{eqnarray*}
\frac{\left[(\sigma_{S}^{(1)})^2-(\sigma^{(2)}_{S})^{2}\right]^2}{(\sigma_{S}^{(1)})^2(\sigma^{(2)}_{S})^{2}}\geq  L \left[|S|^2\left(\frac{1}{n_1^2}+\frac{1}{n_2^2} \right)+ \log[1/(\alpha_S\delta)]\left(\frac{1}{n_1}
+\frac{1}{n_2}\right)\right]\ . 
\end{eqnarray*}
Similarly, combining Lemmas \ref{lemma:upper_bound_q2},
\ref{lemma:upper_bound2_q2}, and \ref{lemma_puissance_deuxieme}, we derive that $\widetilde{Q}_{1,|S|}(F_{S,1}|{\bf X}_S)\leq \alpha_S$ with probability larger than $1-\delta$ if 
\begin{align*}
\frac{\|\beta_S^{(2)}-\beta_S^{(1)}\|_{\Sigma^{(2)}}^2}{(\sigma_{S}^{(1)})^2}\geq &
L'_1\left(\varphi_S+1\right)\left(\frac{1}{n_1}+\frac{1}{n_2}\right)\left[|S|+
  \log\left(\frac{6}{\delta\alpha_S}\right)\right]  + \frac{L'_2}{n_2} \left(\frac{\sigma^{(2)}_{S}}{\sigma^{(1)}_{S}}\right)^2\log\left(\frac{6}{\delta}\right)\ .
\end{align*}
A symmetric result holds for  $\widetilde{Q}_{2,|S|}(F_{S,2}|{\bf X}_S)$.

Consequently, $\widetilde{Q}_{V,|S|}(F_{S,V}|{\bf X}_S)\wedge \widetilde{Q}_{1,|S|}(F_{S,1}|{\bf X}_S)\wedge \widetilde{Q}_{2,|S|}(F_{S,2}|{\bf X}_S)\leq \alpha_S$ with probability larger than $1-\delta$ if 
\begin{eqnarray}\nonumber 
\mathcal{K}_1(S)+ \mathcal{K}_2(S)&\geq& L^*_1\varphi_{S} \left(\frac{1}{n_1}+ \frac{1}{n_2}\right)\left[|S|+\log\left(\frac{6}{\alpha_S\delta}\right)\right]
\\ &&+
 L_2^*\log(6/\delta)\left(\frac{1}{n_1}+\frac{1}{n_2}\right)\left[\left(\frac{\sigma^{(2)}_{S}}{\sigma^{(1)}_{S}}\right)^2+  \left(\frac{\sigma^{(1)}_{S}}{\sigma^{(2)}_{S}}\right)^2\right]\ .\label{eq:lower_bound}
\end{eqnarray}
Since we assume that $4L^*_2\log(6/\delta)\leq n_1\wedge n_2$ in \eqref{eq:thrm_puissance_bonne_hypothese}, the last condition is fulfilled if 
\[ \mathcal{K}_1(S)+ \mathcal{K}_2(S)\geq L^*\varphi_{S} \left(\frac{1}{n_1}+ \frac{1}{n_2}\right)\left[|S|+\log\{6/(\alpha_S\delta)\}\right]\ .\]

We now proceed to the proof of the five previous lemmas.

\begin{proof}[Proof of Lemma \ref{lemma_fisher}]
Let $u\in (0,1)$ and $\bar{F}^{-1}_{D,N}(u)$ be the $1-u$ quantile of a Fisher
random variable with $D$ and $N$ degrees of freedom.
 According to \cite{2003_AS_Baraud}, we have
\begin{eqnarray*}
\bar{F}^{-1}_{D,N}(u)&\leq& 1 + 2\sqrt{\left(\frac{1}{D}+
\frac{1}{N}\right)\log\left(\frac{1}{u}\right)}+
\left(\frac{N}{2D}+1\right)\left[\exp\left(\frac{4}{N}
\log\left(\frac{1}{u}\right)\right)-1\right]\ .
\end{eqnarray*}
Let us assume that $8/N\log(1/u)\leq 1$. By convexity of the exponential
function it holds that 
\begin{eqnarray*}
\bar{F}^{-1}_{D,N}(u)&\leq& 1 + 2\sqrt{\left(\frac{1}{D}+
\frac{1}{N}\right)\log\left(\frac{1}{u}\right)} +
\left(\frac{4}{D}+\frac{8}{N}\right)
\log\left(\frac{1}{u}\right)\ .
\end{eqnarray*}

Recall $T_1$ and $T_2$ defined in (\ref{eq:defi_T1}).
Under hypothesis $\hyp_0$, $$\frac{T_1}{T_2}\sim
\mathrm{Fisher}(n_1-|S|,n_2-|S|)\ .$$ 
Consider some  $x>0$ such that  $[8/(n_1-|S|)\vee 8/(n_2-|S|)]\log(2/x)\leq 1$.
Then, with probability larger than $1-x/2$ we have,
\begin{eqnarray}
\frac{T_1}{T_2} \frac{n_2(n_1-|S|)}{n_1(n_2-|S|)}&\leq&
\left(1+\frac{|S|(n_1-n_2)}{n_1(n_2-|S|)}\right)\left(1+
8\sqrt{\frac{\log(2/x)}{n_1-|S|}} +8 \sqrt{\frac{\log(2/x)}{n_2-|S|}}\right)
\nonumber\\
&\leq & \left(1+\frac{|S|(n_1-n_2)}{n_1(n_2-|S|)}\right)\left(1+
12\sqrt{\frac{\log(2/x)}{n_1}}
+12\sqrt{\frac{\log(2/x)}{n_2}}\right)\leq L \nonumber\ , 
\end{eqnarray}
since $|S|\leq (n_1\wedge n_2)/2$. Similarly, with probability at least $1-x/2$, we have
\begin{eqnarray}
\frac{T_2}{T_1}\frac{n_1(n_2-|S|)}{n_2(n_1-|S|)}\leq 
\left[\left(1+\frac{|S|(n_2-n_1)}{n_2(n_1-|S|)}\right)\left(1+
12\sqrt{\frac{\log(2/x)}{n_1}}
+12\sqrt{\frac{\log(2/x)}{n_2}}\right)\right]\wedge L \
.\label{controle_deviation_fisher}
\end{eqnarray}
Depending on the sign of $\frac{T_1}{T_2}\frac{n_2(n_1-|S|)}{n_1(n_2-|S|)}-1$, we apply one the two following
identities:
\begin{eqnarray*}
\frac{T_1}{T_2}\frac{n_2(n_1-|S|)}{n_1(n_2-|S|)}+ \frac{T_2}{T_1}\frac{n_1(n_2-|S|)}{n_2(n_1-|S|)}-2  &=&
\left(\frac{T_1}{T_2}\frac{n_2(n_1-|S|)}{n_1(n_2-|S|)}-1\right)^2\frac{T_2}{T_1}\frac{n_1(n_2-|S|)}{n_2(n_1-|S|)}\ ,\\
\frac{T_1}{T_2}\frac{n_2(n_1-|S|)}{n_1(n_2-|S|)}+ \frac{T_2}{T_1}\frac{n_1(n_2-|S|)}{n_2(n_1-|S|)}-2 &= &\left(\frac{T_2}{T_1}\frac{n_1(n_2-|S|)}{n_2(n_1-|S|)}-1\right)^2
\frac{T_1}{T_2}\frac{n_2(n_1-|S|)}{n_1(n_2-|S|)}\ .
\end{eqnarray*}
Combining the different bounds, we conclude that with probability larger than
$1-x$,
\begin{eqnarray*}
F_{S,V}&:=&\frac{T_1}{T_2}\frac{n_2(n_1-|S|)}{n_1(n_2-|S|)}+ \frac{T_2}{T_1}\frac{n_1(n_2-|S|)}{n_2(n_1-|S|)}-2\\
&\leq& L
\left[\left(\frac{|S|(n_1-n_2)}{n_1n_2}\right)^2 +
\log(2/x)\frac{n_1+n_2}{n_1n_2}\right]\ .
\end{eqnarray*}
\end{proof}

\begin{proof}[Proof of Lemma \ref{lemma:upper_bound_q2}]
As in the proof of Proposition \ref{prte:majoration_Q}, we note
$N=n_1-|S|$.
Recall that $\widetilde{Q}_{1,|S|}(u|{\bf X}_S)$ is defined  as
$\exp\psi_u(\lambda^\star)$ (see Definition \ref{defi:deviation}). We start by
upper-bounding $\psi_u(\lambda^\star)$, which proves the first
upper-bound of the logarithm of the tail probability $\log
\widetilde{Q}_{1,|S|}(u|{\bf X}_S)$. We then exhibit a value $u_x$
such that $\psi_{u_x}(\lambda^\star) \leq \log x$.

\paragraph{Upper-bound of the tail probability.}
Since Equation (\ref{eq:derive_def_lambda}) is increasing with respect to $\lambda$ and with respect to $N$, $\lambda^{*}$ decreases with $N$. Consequently,
\begin{eqnarray*}
  \lambda^*\leq \lambda_{+} := \frac{u-|a|_1}{2\left[|a|_{\infty}(u-|a|_1)+\|a\|^2\right]}\ .
\end{eqnarray*}
By convexity, $1-\sqrt{1-x}\geq x/2$ for any $0\leq x\leq 1$. Applying this inequality, we upper bound $\sqrt{\Delta}$ and derive that
\begin{eqnarray*}
 \lambda^*\geq \lambda_{-} := \frac{u-|a|_1}{2\left[|a|_{\infty}(u-|a|_1)+\|a\|^2+\frac{|a|_1u}{N}\right]}\ .
\end{eqnarray*}

Since  $u\leq N|a|_{\infty}$, $2\lambda^{*}u\leq N$. Observing that $-\log(1-2x)/2\leq x+x^2/(1-2x)$ for any $0<x<1/2$ and that $\log(1+x)\geq x-x^2$ for any $x>0$, we derive
\begin{eqnarray}\nonumber
\psi_u(\lambda^*)&\leq& |a|_1\lambda_{+}+\frac{\lambda^2_{+} \|a\|^2}{1-2|a|_{\infty}\lambda_{+}}-\lambda^*u+2\frac{(\lambda^*)^2u^2}{N}\\ \nonumber
&\leq & -\frac{(u-|a|_1)^2}{4\left[|a|_{\infty}(u-|a|_1)+\|a\|^2\right]}+ \frac{2\lambda_{+}^2u^2}{N}+ (\lambda_{+}-\lambda_{-})u\\ \label{eq:upper_tail}
&\leq & -\frac{(u-|a|_1)^2}{4\left[|a|_{\infty}(u-|a|_1)+\|a\|^2\right]} +  \frac{(u-|a|_1)u^3}{2N\left[|a|_{\infty}(u-|a|_1)+\|a\|^2\right]^2}\ .
 \end{eqnarray}

\paragraph{Upper-bound of the quantile.} Let us turn to the upper bound of $\widetilde{Q}_{1,|S|}^{-1}(x|{\bf X}_S)$. Consider $u_x$ the solution larger than $|a|_1$ of the equation 
\[\frac{(u-|a|_1)^2}{4\left[|a|_{\infty}(u-|a|_1)+\|a\|^2\right]}=2\log(1/x)\ , \]
and observe that 
\begin{eqnarray*}
2\|a\|\sqrt{\log(1/x)}\leq  u_x-|a|_1\leq 2\sqrt{2}\|a\|\sqrt{\log(1/x)}+ 8|a|_{\infty}\log(1/x) \ . 
\end{eqnarray*}
Choosing $L_1$ and $L_2$ large enough  in the condition $|S|\leq L_1n_1$ and  in condition \eqref{eq:condition_majoration_q2-1} leads us to $u_x\leq
N|a|_{\infty}$. We now prove that $\psi_{u_x\vee 2|a|_1}(\lambda^*)
\leq \log{x}$. If $u_x\geq 2|a|_1$, then $u_x^3\leq 8(u_x-|a|_1)^3$ and it follows from \eqref{eq:upper_tail} that 
\begin{eqnarray*}
\psi_{u_x}(\lambda^*)&\leq& \log(1/x)\left[-2 + \frac{2^8\log(1/x)}{N}\right]\leq - \log(1/x)
\end{eqnarray*}
if we take $L_2$ large enough in Condition (\ref{eq:condition_majoration_q2-1}). 
If $u_x\leq 2|a|_1$, then $|a|_1^2/(|a|_\infty|a|_1+ \|a\|^2)\geq 8\log(1/x)$ and 
\begin{eqnarray*}
\psi_{u_x\vee 2|a|_1}(\lambda^*)&\leq& -\frac{|a|_1^2}{4\left[|a|_{\infty}|a|_1+\|a\|^2\right]}\left[1 - \frac{2^4|a|_1^2}{N\left[|a|_{\infty}|a|_1+\|a\|^2\right]}\right]\leq -\log(1/x)\ ,
\end{eqnarray*}
if we take $L_1$ and $L_2$ large enough in the two aforementionned condition. 
since $|S|\leq 2^{-6}n_1$. 
Thus, we conclude that 
\[
\widetilde{Q}_{1,|S|}^{-1}(x|{\bf X}_S) \leq u_x\vee 2|a|_1\leq |a|_1+
\left[2\sqrt{2}\|a\|\sqrt{\log(1/x)}+
  8|a|_{\infty}\log(1/x)\right]\vee |a|_1\ . 
\]
\end{proof}

\begin{proof}[Proof of Lemma \ref{lemma:upper_bound2_q2}]
Upon defining ${\bf Z}_S^{(1)}={\bf X}_S^{(1)}\left(\Sigma_S^{(1)}\right)^{-1/2}$ and ${\bf Z}_S^{(2)}={\bf X}_S^{(2)}\left(\Sigma_S^{(2)}\right)^{-1/2}$, it follows that ${\bf Z}_{S}^{(1)}$ and ${\bf Z}_S^{(2)}$ follow standard Gaussian distributions.

\begin{eqnarray*}
 |a|_{\infty}&\leq & \frac{n_1}{n_2(n_1-|S|)}\left[1+ \varphi_{\max}\left\{{\bf Z}_{S}^{(2)}\sqrt{\Sigma_S^{(2)}(\Sigma_S^{(1)})^{-1}}\left({\bf Z}_{S}^{(1)\intercal}{\bf Z}_{S}^{(1)}\right)^{-1}\sqrt{(\Sigma_S^{(1)})^{-1}\Sigma_S^{(2)}}{\bf Z}_{S}^{(2)\intercal}\right\}\right]\\
&\leq& \frac{2}{n_2}+ 2\frac{\varphi_{\max}[{\bf Z}_{S}^{(2)\intercal}{\bf Z}_{S}^{(2)}]}{n_2\varphi_{\max}[{\bf Z}_{S}^{(1)\intercal}{\bf Z}_{S}^{(1)}]}\varphi_{\max}\left[\sqrt{\Sigma_S^{(2)}}(\Sigma_S^{(1)})^{-1}\sqrt{\Sigma_S^{(2)}}\right]\ .
\end{eqnarray*}
In order to conclude, we control the largest and the smallest eigenvalues of Standard Wishart matrices applying Lemma \ref{lemma:concentration_vp_wishart}.
\end{proof}

\begin{proof}[Proof of Lemma \ref{lemma_fisher2}]
By symmetry, we can assume that  $\sigma_{S}^{(1)}/\sigma^{(2)}_{S}\geq 1$. Recall the definition of $T_1$ and $T_2$ in the proof of Proposition \ref{prte:distribution_parametrique}\\

\noindent {\bf CASE 1}. Suppose that $T_1/T_2\geq 1$.\\
\begin{eqnarray}
-2 +\frac{(\sigma_{S}^{(1)})^2}{(\sigma^{(2)}_{S})^{2}}\frac{T_1}{T_2}+
\frac{(\sigma^{(2)}_{S})^{2}}{(\sigma_{S}^{(1)})^2}\frac{T_2}{T_1}& \geq &
\frac{[(\sigma_{S}^{(1)})^2-(\sigma^{(2)}_{S})^{2}]^2}{(\sigma_{S}^{(1)})^2(\sigma^{(2)}_{S})^{2}}+
\frac{(\sigma_{S}^{(1)})^2}{(\sigma^{(2)}_{S})^{2}}\left(\frac{T_1}{T_2}-1\right)
+\frac{(\sigma^{(2)}_{S})^{2}}{(\sigma_{S}^{(1)})^2}\left(\frac{T_2}{T_1}-1\right)\nonumber\\ 
  &\geq &
\frac{[(\sigma_{S}^{(1)})^2-(\sigma^{(2)}_{S})^{2}]^2}{(\sigma_{S}^{(1)})^2(\sigma^{(2)}_{S})^{2}}\
.\label{Sinoration_fisher_puissance1}
\end{eqnarray}

\noindent {\bf CASE 2}. Suppose that $T_1/T_2\leq  1$.
\begin{eqnarray*}
-2 +\frac{(\sigma_{S}^{(1)})^2}{(\sigma^{(2)}_{S})^{2}}\frac{T_1}{T_2}+
\frac{(\sigma^{(2)}_{S})^{2}}{(\sigma_{S}^{(1)})^2}\frac{T_2}{T_1} &=&
\left(\frac{(\sigma_{S}^{(1)})^2}{(\sigma^{(2)}_{S})^{2}}-\frac{T_2}{T_1}\right)^2
\frac{ (\sigma^{(2)}_{S})^{2} } {(\sigma_{S}^{(1)})^2 }\frac { T_1 } { T_2 } \nonumber\\
&\geq  &\frac{T_1}{T_2}
\frac{[(\sigma_{S}^{(1)})^2-(\sigma^{(2)}_{S})^{2}]^2}{4(\sigma_{S}^{(1)})^2(\sigma^{(2)}_{S})^{2}}\mathbf{1
}_{ \frac{
(\sigma_{S}^{(1)})^2}{(\sigma^{(2)}_{S})^{2}}-1\geq 2\left(\frac{T_2}{T_1}-1\right)}\ .
\end{eqnarray*}
We need to control the deviations of $T_2/T_1$. Using bound 
\eqref{controle_deviation_fisher}, we get
\begin{eqnarray*}
\frac{T_2}{T_1}\leq  \left(1+\frac{|S|(n_2-n_1)}{n_2(n_1-|S|)}\right)\left(1+
12\sqrt{\frac{\log(1/\delta)}{n_1}}
+12\sqrt{\frac{\log(1/\delta)}{n_2}}\right)\ ,
\end{eqnarray*}
 with probability larger than $1-\delta$. Since $|S|\leq (n_1\wedge n_2)/2$, we derive
that 
$$\frac{T_2}{T_1}-1\leq  \frac{2|S|}{n_1}+
24\sqrt{\frac{\log(1/\delta)}{n_1}}
+24\sqrt{\frac{\log(1/\delta)}{n_2}}\leq 3 \ ,$$
for $L_1$ large enough in the statement of the lemma.
In conclusion, we have
\begin{eqnarray}
-2 +\frac{(\sigma_{S}^{(1)})^2}{(\sigma^{(2)}_{S})^{2}}\frac{T_1}{T_2}+
\frac{(\sigma^{(2)}_{S})^{2}}{(\sigma_{S}^{(1)})^2}\frac{T_2}{T_1} \geq
\frac{[(\sigma_{S}^{(1)})^2-(\sigma^{(2)}_{S})^{2}]^2}{16(\sigma_{S}^{(1)})^2(\sigma^{(2)}_{S})^{2}}\ ,
\label{Sinoration_fisher_puissance2}
\end{eqnarray}
with probability larger than $1-\delta$, as long as 
\begin{eqnarray}\label{Sinoration_fisher_puissance3}
\frac{[(\sigma_{S}^{(1)})^2-(\sigma^{(2)}_{S})^{2}]^2}{(\sigma_{S}^{(1)})^2(\sigma^{(2)}_{S})^{2}}\geq
L\left[\frac{|S|^2}{n_1^2}+\frac{|S|^2}{n_2^2}+ \log(1/\delta)\left(\frac{1}{n_1}
+\frac{1}{n_2}\right)\right]\ .
\end{eqnarray}
Combining (\ref{Sinoration_fisher_puissance1}),
(\ref{Sinoration_fisher_puissance2}), and (\ref{Sinoration_fisher_puissance3}),
we derive
\begin{eqnarray*}
-2 +\frac{(\sigma_{S}^{(1)})^2}{(\sigma^{(2)}_{S})^{2}}\frac{T_1}{T_2}+
\frac{(\sigma^{(2)}_{S})^{2}}{(\sigma_{S}^{(1)})^2}\frac{T_2}{T_1} \geq
\frac{[(\sigma_{S}^{(1)})^2-(\sigma^{(2)}_{S})^{2}]^2}{16(\sigma_{S}^{(1)})^2(\sigma^{(2)}_{S})^{2}}-
L\left[\frac{|S|^2}{n_1^2}+ \frac{|S|^2}{n_2^2}+ \log(1/\delta)\left(\frac{1}{n_1}
+\frac{1}{n_2}\right)\right]\ ,
\end{eqnarray*}
with probability larger than $1-\delta$.
\end{proof}

\begin{proof}[Proof of Lemma \ref{lemma_puissance_deuxieme}]

We  want to lower bound the random variable $F_{S,1}=  \frac{R n_1}{(\sigma^{(1)}_{S})^2T_1(n_1-|S|)}$ where $R$ is defined by
$$R:= \|\X_S^{(2)}(\beta_S^{(2)}-\beta_S^{(1)})+\Pi_S^{(2)}
\eps_S^{(2)}-\X_S^{(2)}(\X_S^{(1)\intercal
} \X_S^{
(1)})^{(-1)}\X_{S}^{(1)\intercal}\eps^{(1)}_S \|^2/n_2\ .$$
Let us first work conditionally to  ${\bf X}^{(1)}_S$ and $\X_S^{(2)}$. Upon
defining the Gaussian vector $W$ by
 $$W\sim
\mathcal{N}\left[0,(\sigma^{(2)}_{S})^{2}\Pi_S^{(2)}+(\sigma_{S}^{(1)})^2\X_S^{(2)}(\X_S^{(1)\intercal
} \X_S^{
(1)})^{(-1)}\X_{S}^{(2)\intercal}\right]\ ,$$
we get $R= \|\X_S^{(2)}(\beta_S^{(2)}-\beta_S^{(1)})+W\|^2/n_2$. We have the
following lower bound:
\begin{eqnarray*}
R &\geq& \left(\|\X_S^{(2)}(\beta_S^{(2)}-\beta_S^{(1)})\|+\left\langle
W,\frac{\X_S^{(2)}(\beta_S^{(2)}-\beta_S^{(1)})}{\|\X_S^{(2)}
(\beta_S^ { (2) } -\beta_S^{(1)})\|}\right\rangle\right)^2/n_2\\
&\geq&\frac{\|\X_S^{(2)}(\beta_S^{(2)}-\beta_S^{(1)})\|^2}{2n_2}-\frac{1}{n_2}\left\langle
W,\frac{\X_S^{(2)}(\beta_S^{(2)}-\beta_S^{(1)})}{\|\X_S^{(2)}
(\beta_S^ { (2) } -\beta_S^{(1)})\|}\right\rangle^2
\end{eqnarray*}
The random variable
$\|\X_S^{(2)}(\beta_S^{(2)}-\beta_S^{(1)})\|^2/\|\beta_S^{(2)}-\beta_S^{(1)}\|^2_{\Sigma^{(2)}}$
follows a $\chi^2$ distribution with $n_2$ degrees of freedom. Given $(\X^{(1)}_S, \X^{(2)}_S)$, $\left\langle
W,\frac{\X_S^{(2)}(\beta_S^{(2)}-\beta_S^{(1)})}{\|\X_S^{(2)}
(\beta_S^ { (2) } -\beta_S^{(1)})\|}\right\rangle^2$ is proportional to a $\chi^2$ distributed random variable with one degree of freedom and its variance is smaller than $(\sigma_{S}^{(2)})^2+ \varphi_{\max}[\X_S^{(2)}(\X_S^{(1)\intercal
} \X_S^{(1)})^{(-1)}\X_{S}^{(2)\intercal}](\sigma_{S}^{(1)})^2$.
Applying Lemma \ref{lemma:concentration_chi2}, we derive that with probability larger than $1-x/6$, 
\begin{eqnarray*}
 R &\geq & \frac{\|\beta_S^{(2)}-\beta_S^{(1)}\|^2_{\Sigma^{(2)}}}{2}\left[1-2\sqrt{\frac{\log(12/x)}{n_2}}\right]\\ &-& 4 \frac{\log\left(12/x\right)}{n_2}\left[(\sigma_{S}^{(2)})^2+(\sigma_{S}^{(1)})^2\varphi_{\max}\{\X_S^{(2)}(\X_S^{(1)\intercal
} \X_S^{(1)})^{(-1)}\X_{S}^{(2)\intercal}\}\right]\ .
\end{eqnarray*}
Using the upper bound $|S|\leq (n_1\wedge n_2)/2$ and Lemma \ref{lemma:concentration_vp_wishart}, we control  the last term
\begin{eqnarray*}
\varphi_{\max}\left[\X_S^{(2)}(\X_S^{(1)\intercal
} \X_S^{(1)})^{(-1)}\X_{S}^{(2)\intercal}\right]\leq L\varphi_S \frac{n_2}{n_1}\ , 
\end{eqnarray*}
with probability larger than $1-2\exp[-(n_1\wedge n_2)L']$. If we take the constant $L_1$ large enough in condition (\ref{eq:condition_lemma_puissance_deuxieme}), then we get
\begin{eqnarray}\label{Sinoration_numerateur_puissance}
 R &\geq & \frac{\|\beta_S^{(2)}-\beta_S^{(1)}\|^2_{\Sigma^{(2)}}}{4}-  \log\left(12/\delta\right)L\left[\frac{(\sigma_{S}^{(2)})^2}{n_2}+\frac{(\sigma_{S}^{(1)})^2}{n_1}\varphi_{S}\right]\ ,
\end{eqnarray}
with probability larger than $1-\delta/3$.

\medskip

Let us now upper bound the random variable $T_1(n_1-|S|)/n_1$. Since $(n_1-S)T_1$ follows a $\chi^2$ distribution with $n_1-|S|$ degrees of freedom, we derive from Lemma \ref{lemma:concentration_chi2} that 
\begin{eqnarray}\label{Sinoration_denominateur_puissance}
 T_1(n_1-|S|)/n_1& \leq &
1+2\sqrt{\frac{\log(6/\delta)}{n_1}} +
\frac{2}{n_1}\log(6/\delta)\leq 2\ ,
\end{eqnarray}
 with probability  larger than $1-\delta/6$.
Gathering
(\ref{Sinoration_numerateur_puissance}) and
(\ref{Sinoration_denominateur_puissance}), we conclude that
\begin{eqnarray*}
F_{S,1}\geq
\frac{\|\beta_S^{(2)}-\beta_S^{(1)}\|_{\Sigma^{(2)}}^2}{8(\sigma_{S}^{(1)})^2} - \log\left(6/\delta\right)L\left[\left(\frac{\sigma_{S}^{(2)}}{\frac{1}{n_2}\sigma^{(1)}_{S}}\right)^2+\frac{\varphi_{S}}{n_1}\right]\ ,
\end{eqnarray*}
with probability larger than $1-\delta/2$.
 
\end{proof}

\subsection{Proof of Theorem \ref{cor:section_complete}: Power of $T_{\cS_{\leq k}}^B$}
This proposition is a straightforward corollary of Theorem \ref{thrm_puissance}. Consider the subsets $S_{\vee}$ and $S_{\Delta}$ of $\{1,\ldots, p\}$ such that $S_{\vee}$ is the union of the support of $\beta^{(1)}$ and $\beta^{(2)}$ and $S_{\Delta}$ is the supports of $\beta^{(2)}-\beta^{(1)}$. Assume first that $S_{\vee}$ and $S_{\Delta}$ are non empty.
By Definition \eqref{eq:condition_bonf_classique} of the weights, we have
\[\log\left(\frac{1}{\alpha_{i,S_\cup}}\right)\leq \log(4k)+ \log(1/\alpha)+ |S_{\vee}|\log(p)\leq 2|S|_{\cup}\log(p)+ \log(1/\alpha)\ .\]
A similar upper bound holds for $\log(1/\alpha_{i,S_\Delta})$.  If we choose the numerical constants large enough in Conditions ${\bf A.1}$ and ${\bf A.2}$, then the sets $S_{\vee}$ and $S_{\Delta}$
follow the conditions of Theorem  \ref{thrm_puissance}.

Applying Theorem \ref{thrm_puissance}, we derive that $T_{\cS_{\leq k}}^B$ rejects $\hyp_0$ with probability larger than $1-\delta$ when 
\[\mathcal{K}_1(S_{\vee})+\mathcal{K}_2(S_{\vee})\geq \varphi_{S_\cup}\left(\frac{1}{n_1}+ \frac{1}{n_2}\right)\left[|S_{\vee}|+ \log\left(\frac{1}{\alpha_{S_\cup}}\right)\right]\ .\]
Observing that $ \varphi_{S_\cup}\leq \varphi_{\Sigma^{(1)},\Sigma^{(2)}} $,  $\mathcal{K}_1(S_{\vee})=\mathcal{K}_1$, $\mathcal{K}_2(S_{\vee})=\mathcal{K}_2$ and that $|S_{\vee}|\leq |\beta^{(1)}|_0+ |\beta^{(2)}|_0$ allows to prove the first result. 
Let us turn to the second result. According to Theorem \ref{thrm_puissance}, $T_{\cS_{\leq k}}^B$ rejects $\hyp_0$ with probability larger than $1-\delta$ when 
\[\mathcal{K}_1(S_{\Delta})+\mathcal{K}_2(S_{\Delta})\geq \varphi_{S_\Delta}\left(\frac{1}{n_1}+ \frac{1}{n_2}\right)\left[|S_{\Delta}|+ \log\left(\frac{1}{\alpha_{S_\Delta}}\right)\right]\ .\]
Since $\mathcal{K}_1(S_{\Delta})+\mathcal{K}_2(S_{\Delta})\geq \frac{\|\beta^{(1)}-\beta^{(2)}\|_{\Sigma}^2}{2[\var(Y^{(1)})\wedge \var(Y^{(2)})]}$ and since $|S_\Delta|=|\beta^{(1)}-\beta^{(2)}|_0$, the second result follows.

 If $S_{\vee}=\emptyset$, then we can consider any subset of size $1$ to prove the first result.  If $S_{\Delta}=\emptyset$, then $\beta^{(1)}=\beta^{(2)}$ and the second result does not tell us anything.

\subsection{Proof of Proposition \ref{prte:lasso_precis}}\label{section_proof_lasso}
For simplicity, we assume in the sequel that $\beta^{(1)}\neq 0$ or $\beta^{(2)}\neq 0$, the case $\beta^{(1)}=\beta^{(2)}=0$ being handled by any set $S\in\mathcal{S}_1\subset \SLasso$.

This proof  is divided into two main steps. First, we prove that with large probability the collection  $\SLasso$ contains some set $\widehat{S}_{\lambda}$ close to  the union $S_{\vee}$ of the supports of $\beta^{(1)}$ and $\beta^{(2)}$. Then, we show that the statistics $(F_{\widehat{S}_{\lambda},V}, F_{\widehat{S}_{\lambda},1},F_{\widehat{S}_{\lambda},2})$ allow to reject $\hyp_0$ with large probability. \\

Recall that the collection $\SLasso$ is based on the Lasso regularization path of the following heteroscedastic Gaussian linear model,
\begin{eqnarray}\label{eq:definition_modele_regression_lasso}
\left[\begin{array}{c}
       {\bf Y}^{(1)}\\{\bf Y}^{(2)}
      \end{array}
\right] = \left[\begin{array}{cc}
                 \X^{(1)}& \X^{(1)} \\
		\X^{(2)} & -{\bf X}^{(2)}
                \end{array}
\right]\left[
\begin{array}{c}
      \theta_*^{(1)}\\ \theta_*^{(2)}
      \end{array}
\right] + \left[\begin{array}{c}
                 \eps^{(1)} \\ \eps^{(2)}
                \end{array}
\right]
\end{eqnarray}
which we denote for short $\Y =  {\bf W}\theta_* +\eps$.
Given a tuning parameter $\lambda$, $\widehat{\theta}_{\lambda}$ refers to the Lasso estimator of $\theta$:
\[\widehat{\theta}_{\lambda}= \arg\inf_{\theta\in\mathbb{R}^{2p}}\|\Y -{\bf W}\theta\|^2+\lambda|\theta|_1\ .\]
In order to analyze the Lasso solution $\widehat{\theta}_{\lambda}$,
we need to control how ${\bf W}$ acts on sparse vectors.

\begin{lemma}[Control of the design ${\bf W}$]\label{lemma:controle_W}
If we take the constants $L^*$, $L^*_1$, and $L^*_2$  in Proposition \ref{prte:lasso_precis} small enough then the following holds.
The event 
\begin{eqnarray*}
 \mathcal{A}&:=&\left\{\forall \theta\  s.t.\  |\theta|_0\leq k_*,\ 1/2\leq \frac{\|\X^{(1)} \theta\|^2}{n_1\|\theta\|^2_{\Sigma^{(1)}}}\leq 2\text{ and }  1/2\leq \frac{\|\X^{(2)} \theta\|^2}{n_2\|\theta\|^2_{\Sigma^{(2)}}}\leq 2 \right\}\ \\
&&\bigcap \left\{\frac{\kappa\left[6,|\beta^{(1)}|_0+|\beta^{(2)}|_0,\X^{(1)}/\sqrt{n_1}\right]}{\kappa\left[6,|\beta^{(1)}|_0+|\beta^{(2)}|_0,\sqrt{\Sigma^{(1)}}\right]}\bigwedge \frac{\kappa\left[6,|\theta_*|_0,\X^{(2)}/\sqrt{n_1}\right]}{\kappa\left[6,|\theta_*|_0,\sqrt{\Sigma^{(2)}}\right]} \geq 2^{-3} \right\}
\end{eqnarray*}
has probability larger than  $1-\delta/4$. Furthermore, on the event $\mathcal{A}$, 
\begin{eqnarray*}
 \Phi_{k,+}({\bf W})& \leq& 4(n_1+n_2)\left[\Phi_{k,+}(\sqrt{\Sigma^{(1)}})\vee \Phi_{k,+}(\sqrt{\Sigma^{(2)}})\right]\ , \\ \Phi_{k,-}({\bf W}) &\geq& (n_1\wedge n_2)\left[\Phi_{k,-}(\sqrt{\Sigma^{(1)}})\wedge \Phi_{k,-}(\sqrt{\Sigma^{(2)}})\right]\ ,
\end{eqnarray*}
for any $k\leq k_*$. 
\end{lemma}

The following property is a slight variation of Lemma 11.2 in \cite{2009_EJS_Geer} and Lemma~3.2 in \cite{GHV:12-sup}.
\begin{lemma}[Behavior of the Lasso estimator $\widehat{\theta}_{\lambda}$]\label{lem:cardinal-lasso}
If we take $L^*_2$ in Proposition \ref{prte:lasso_precis} small enough then the following holds.
The event \[\mathcal{B}= \ac{|{\bf W}^T\eps|_{\infty}\leq 2(\sigma^{(1)}\vee \sigma^{(2)})\sqrt{2\Phi_{1,+}({\bf W})\log(p)}}\] occurs with probability larger than $1-1/p$.
Assume that  \[\lambda\geq 8(\sigma^{(1)}\vee \sigma^{(2)})\sqrt{2\Phi_{1,+}({\bf W})\log(p)}\ .\]
Then, on the event $\mathcal{A}\cap\mathcal{B}$ we have 
\begin{eqnarray}\label{eq:upper_prediction}
\|{\bf W}(\widehat{\theta}_\lambda-\theta_*)\|^2\leq L_1\frac{\lambda^2/(n_1\wedge n_2)}{\kappa^2[6,|\theta_*|_0,\sqrt{\Sigma^{(1)}}]\wedge \kappa^2[6,|\theta_*|_0,\sqrt{\Sigma^{(2)}}] }|\theta_*|_{0}\ ,\\
|\widehat{\theta}_{\lambda}|_0 \leq L_2\frac{n_1\vee n_2}{n_1\wedge n_2} \frac{\Phi_{k_*,+}(\sqrt{\Sigma^{(1)}})\vee \Phi_{k_*,+}(\sqrt{\Sigma^{(2)}})}{ \kappa^2[6,|\theta_*|_0,\sqrt{\Sigma^{(1)}}]\wedge \kappa^2[6,|\theta_*|_0,\sqrt{\Sigma^{(2)}}]} |\theta_*|_0\ \leq k_*/2\ . \label{eq:upper_l0}
\end{eqnarray}
\end{lemma}

In the sequel, we fix 
\[\lambda= 16(\sigma^{(1)}\vee \sigma^{(2)})\sqrt{2(n_1+n_2)\left[\Phi_{1,+}(\sqrt{\Sigma^{(1)}})\vee\Phi_{1,+}(\sqrt{\Sigma^{(2)}}) \right]\log(p)}\ .\]
and we consider the set $\widehat{S}_{\lambda}$ defined by the union of the support of $\widehat{\theta}^{(1)}_{\lambda}$ and $\widehat{\theta}^{(2)}_{\lambda}$. 
On the event $\mathcal{A}\cap \mathcal{B}$, Lemma \ref{lem:cardinal-lasso} tells us that $|\widehat{S}_{\lambda}|\leq k_*$. Thus,  $\widehat{S}_{\lambda}$ belongs to the collection $\SLasso$.
We shall prove that 
\begin{eqnarray*}
\min_{i\in\{V,1,2\}}\widetilde{Q}_{i,|\widehat{S}_{\lambda}|}\left(\left. F_{\widehat{S}_{\lambda},i}\right|{\bf X}_{\widehat{S}_{\lambda}} \right)<\alpha_{i,\widehat{S}_{\lambda}}
\end{eqnarray*}
with probability larger than $1-\delta/2$. In the following lemma, we compare $\mathcal{K}_1(\widehat{S}_{\lambda})+ \mathcal{K}_2(\widehat{S}_{\lambda})$ to $\mathcal{K}_1+ \mathcal{K}_2$.
Note $R_{\Sigma^{(1)},\Sigma^{(2)}}= \frac{\bigvee_{i=1,2}\Phi_{k_*,+}(\sqrt{\Sigma^{(i)}})}{\bigwedge_{i=1,2}\Phi_{k_*,-}(\sqrt{\Sigma^{(i)}}) }\frac{\bigvee_{i=1,2}\Phi_{1,+}(\sqrt{\Sigma^{(i)}})}{\bigwedge_{i=1,2}\kappa^2[6,|\theta_*|_0,\sqrt{\Sigma^{(i)}}] }$.
\begin{lemma}\label{lemma:controle_biais_kullback}
 On the event $\mathcal {A}\cap \mathcal{B}$, we have
\begin{eqnarray*}
 L \left[\mathcal{K}_1(\widehat{S}_{\lambda}) + \mathcal{K}_2(\widehat{S}_{\lambda})\right]&\geq& 1\wedge \left[\mathcal{K}_1+\mathcal{K}_2 - L' R_{\Sigma^{(1)},\Sigma^{(2)}}\frac{|S_{\vee}|(n_1\vee n_2)}{(n_1\wedge n_2)^2}\log(p)\right]\ .
\end{eqnarray*}

\end{lemma}
Then, we closely follow the arguments of Theorem \ref{thrm_puissance} to state that $T_{\SLasso}^B$ rejects $\H_0$ with large probability as long as $\mathcal{K}_1(\widehat{S}_{\lambda})+ \mathcal{K}_2(\widehat{S}_{\lambda})$ is large enough.
\begin{lemma}\label{lemma:controle_puissance_aleatoire}
If on the event $\mathcal{A}\cap \mathcal{B}$, we have 
\begin{equation*}
\mathcal{K}_1(\widehat{S}_{\lambda})+ \mathcal{K}_2(\widehat{S}_{\lambda})\geq L \varphi_{\widehat{S}_{\lambda}} \left(\frac{1}{n_1}+ \frac{1}{n_2}\right)  \left[|\widehat{S}_{\lambda}|\log(p)+ \log\left(\frac{1}{\alpha\delta}\right) + \log(p) \right]\ ,
\end{equation*}
then,  $\min_{i\in\{V,1,2\}}\widetilde{Q}_{i,|\widehat{S}_{\lambda}|}( F_{\widehat{S}_{\lambda},i}|{\bf X}_{\widehat{S}_{\lambda}})<\alpha_{i,\widehat{S}_{\lambda}}$ with probability larger than $1-\delta/2$.
\end{lemma}
We derive from (\ref{eq:upper_l0}) that on the event $\mathcal{A}\cap\mathcal{B}$, 
\[|\widehat{S}_{\lambda}| \leq L'\frac{n_1\vee n_2}{n_1\wedge n_2} \frac{\bigvee_{i=1,2} \Phi_{k_*,+}(\sqrt{\Sigma^{(i)}})}{ \bigwedge_{i=1,2}\kappa^2[6,|\theta_*|_0,\sqrt{\Sigma^{(i)}}]} |S_{\vee}|\ .\]
Since $|S_{\vee}|\leq |\beta^{(1)}|_0+ |\beta^{(2)}|_0$, it follows from Condition (\ref{eq:hypothese_sparsite}) that $|\widehat{S}_{\lambda}| \leq k_*$.
Gathering Lemmas \ref{lemma:controle_biais_kullback} and \ref{lemma:controle_puissance_aleatoire} allows us to conclude if we take $L^*_3$ in Proposition \ref{prte:lasso_precis} large enough.

\begin{proof}[Proof of Lemma \ref{lemma:controle_W}]
In order to bound $ \mathbb{P}(\mathcal{A})$, we apply Lemma \ref{lemma:concentration_vp_wishart} to simultaneously control $\varphi_{\max}({\bf X}_S^{(1)\intercal}\X_S^{(1)})$, $\varphi_{\max}({\bf X}_S^{(2)\intercal}\X_S^{(2)})$, $\varphi_{\min}({\bf X}_S^{(1)\intercal}\X_S^{(1)})$, and $\varphi_{\min}({\bf X}_S^{(2)\intercal}\X_S^{(2)})$ for all sets $S$ of size  $k_*$. Combining a union bound with Conditions (\ref{eq:definition_k_*}) and (\ref{eq:lasso_hypothese_alpha_delta}) allows us to prove that
\begin{eqnarray*}
\mathbb{P}\left[\left\{\forall \theta\  s.t.\  |\theta|_0\leq k_*,\ 1/2\leq \frac{\|\X^{(1)} \theta\|^2}{n_1\|\theta\|^2_{\Sigma^{(1)}}}\leq 2\text{ and }  1/2\leq \frac{\|\X^{(2)} \theta\|^2}{n_2\|\theta\|^2_{\Sigma^{(2)}}}\leq 2 \right\}\right]\geq 1-\delta/8
\end{eqnarray*}
Applying Corollary 1 in \cite{10:RWG_restricted}, we derive that there exist three positive constant $c_1$, $c_2$ and $c_3$ such that the following holds.
With probability larger than $1-c_1\exp[-c_2(n_1\wedge n_2)]$, we have
\begin{eqnarray*}
\bigwedge_{i=1,2} \frac{\kappa\left[6,|\theta_*|_0,\X^{(i)}/\sqrt{n_i}\right]}{\kappa\left[6,|\theta_*|_0,\sqrt{\Sigma^{(i)}}\right]}\geq 2^{-3}\ ,
\end{eqnarray*}
if $|\theta_*|_0\log(p)< c_3\frac{\vee_{i=1,2}\Phi_{1,+}(\sqrt{\Sigma^{(i)}})}{\wedge_{i=1,2} \kappa^2\left[6,|\theta_*|_0,\sqrt{\Sigma^{(i)}}\right]}(n_1\wedge n_2)$. Hence, we conclude that $\mathbb{P}\left[\mathcal{A}\right]\geq 1-\delta/4$.

Consider an integer $k\leq k_*$ and  a $k$-sparse vector $\theta=\left(\begin{array}{c}
\theta^{(1)}\\ \theta^{(2)}                                                                                                                                                                         \end{array}\right)$ in $\mathbb{R}^{2p}$. Under event $\mathcal{A}$, we have 
\begin{eqnarray*}
 \|{\bf W}\theta\|^2& =& \|\X^{(1)}(\theta^{(1)}+ \theta^{(2)})\|^2+  \|\X^{(2)}(\theta^{(1)}- \theta^{(2)})\|^2\\
&\leq & 2 n_1\|\theta^{(1)}+\theta^{(2)}\|^2_{\Sigma^{(1)}}+ 2n_2\|\theta^{(1)}-\theta^{(2)}\|_{\Sigma^{(2)}}^2\\& \leq& 4(n_1+n_2)\left[\Phi_{k,+}(\sqrt{\Sigma^{(1)}})\vee \Phi_{k,+}(\sqrt{\Sigma^{(2)}})\right]\|\theta\|^2\\
 \|{\bf W}\theta\|^2&\geq & \frac{1}{2}\left[n_1\|(\theta^{(1)}+\theta^{(2)}\|^2_{\Sigma^{(1)}}+ n_2\|\theta^{(1)}-\theta^{(2)}\|_{\Sigma^{(2)}}^2\right]\\
&\geq &(n_1\wedge n_2)\left[\Phi_{k,-}(\sqrt{\Sigma^{(1)}})\wedge \Phi_{k,-}(\sqrt{\Sigma^{(2)}})\right]\|\theta\|^2\ .
\end{eqnarray*}

\end{proof}

\begin{proof}[Proof of Lemma \ref{lem:cardinal-lasso}]
Observe that the variance of $[{\bf W}^{\intercal}\eps]_i$ given ${\bf W}$ is smaller than $\Phi_{1,+}({\bf W})(\sigma^{(1)}\vee \sigma^{(2)})^2$. Using a union bound and the deviations of the Gaussian distribution, it follows that $\mathbb{P}(\mathcal{B})\geq 1-1/p$.

Recall the definition of $\eta[.,.]$ in (\ref{eq:definition_eta}). A slight variation of Lemma 11.2  in \cite{2009_EJS_Geer} ensures that
\begin{equation}\label{KLT}
\|{\bf W}(\widehat{\theta}_{\lambda}-\theta_*)\|^2\leq  L\frac{\lambda^2}{\eta^2[3,|\theta_*|_0,{\bf W}]}|\theta_*|_{0}
\end{equation}
on event $\mathcal{B}$. Fix $k=|\theta_*|_0$ and consider some  $\theta=\left(\begin{array}{c}
\theta^{(1)}\\ \theta^{(2)}                                                                                                                                                                         \end{array}\right)\in \mathcal{C}\left(3,T\right)$ with $|T|=k$. 
Define $T'\subset\{1,\ldots,p\}$ by $i\in T'$ if $i\in T$ or $i+p\in T$. We have
\begin{eqnarray*}
|(\theta^{(1)}+\theta^{(2)})_{T'^c}|_1\vee |(\theta^{(1)}-\theta^{(2)})_{T'^c}|_1&\leq& |\theta^{(1)}_{T'^c}|_1+ |\theta^{(2)}_{T'^c}|_1\leq |\theta_{T^c}|_1\leq 3|\theta_T|_1\\
&\leq & 3\left[|\theta^{(1)}_{T'}|_1+ |\theta^{(2)}_{T'}|_1\right]  \\
&\leq & 6 \left[|(\theta^{(1)}+\theta^{(2)})_{T'}|_1\vee |(\theta^{(1)}-\theta^{(2)})_{T'}|_1 \right]
\end{eqnarray*}
It follows that $\theta^{(1)}+\theta^{(2)}\in\mathcal{C}(6,T')$ or $\theta^{(1)}-\theta^{(2)}\in\mathcal{C}(6,T')$. By symmetry, we assume that $|(\theta^{(1)}+\theta^{(2)})_{T'}|_1\geq |(\theta^{(1)}-\theta^{(2)})_{T'}|_1$. Let us lower bound the $l_1$ norm of $\theta^{(1)}+\theta^{(2)}$ in terms of $\theta$.
\begin{eqnarray*}
 2|\theta^{(1)}+\theta^{(2)}|_1&\geq& \left[ \left|\left(\theta^{(1)}+\theta^{(2)}\right)_{T'}\right|_1+\left|\left(\theta^{(1)}-\theta^{(2)}\right)_{T'} \right|_1\right]
\geq  |\theta_T|_1\geq  \frac{|\theta|_1}{4}\ ,
\end{eqnarray*}
since $\theta$ belongs to $\mathcal{C}(3,T)$. Thus, we derive the lower bound 
\begin{eqnarray*}
\frac{k\|{\bf W}\theta\|^2}{|\theta|_1^2}&\geq &  \frac{k\|\X^{(1)}(\theta^{(2)}+\theta^{(1)})\|^2}{|\theta|_1^2}+   \frac{k\|\X^{(2)}(\theta^{(2)}-\theta^{(1)})\|^2}{|\theta|_1^2}\\
&\geq &\frac{(n_1\wedge n_2)|\theta^{(2)}+\theta^{(1)}|^2_1  }{|\theta|_1^2}\left[\bigwedge_{i=1,2}\eta^2\left(6,k,\X^{(i)}/\sqrt{n_i}\right)\right]\\
&\geq & L (n_1\wedge n_2)\left[\bigwedge_{i=1,2}\kappa^2\left(6,k,\X^{(i)}/\sqrt{n_i}\right)\right]\\
&\geq &L (n_1\wedge n_2) \left[\kappa^2\left(6,k,\sqrt{\Sigma^{(1)}}\right)\wedge \kappa^2\left(6,k,\sqrt{\Sigma^{(2)}}\right)\right]\ ,
\end{eqnarray*}
where the last inequality proceeds from Lemma \ref{lemma:controle_W}. 
We conclude that  
\[L'\kappa^2[3,|\theta_*|_0,{\bf W}]\geq (n_1\wedge n_2)\left[\kappa^2\left(6,k,\sqrt{\Sigma^{(1)}}\right)\wedge \kappa^2\left(6,k,\sqrt{\Sigma^{(2)}}\right)\right]\ .\]
Gathering this bound with (\ref{KLT}), it follows that 
\begin{equation*}
\|{\bf W}(\widehat{\theta}_{\lambda}-\theta_*)\|^2\leq  \frac{L'\lambda^2/(n_1\wedge n_2)}{\kappa^2[6,|\theta_*|_0,\sqrt{\Sigma^{(1)}}]\wedge \kappa^2[6,|\theta_*|_0,\sqrt{\Sigma^{(2)}]} }|\theta_*|_{0}\ ,
\end{equation*}
which allows us to prove (\ref{eq:upper_prediction}). Lemma~3.1 in
\cite{GHV:12-sup} tells us that on event
$\mathcal{B}$, \[\lambda^2|\widehat{\theta}_{\lambda}|_0\leq
16\Phi_{|\widehat{\theta}_{\lambda}|_0,+}({\bf W})\|{\bf
  W}(\widehat{\theta}_{\lambda}-\theta_*)\|^2\ .\] Gathering the last two bounds and Lemma \ref{lemma:controle_W}, we obtain 
\begin{equation}
|\widehat{\theta}_{\lambda}|_0\leq L \,\frac{\Phi_{|\widehat{\theta}_{\lambda}|_0,+}({\bf W})}{ \kappa^2[6,|\theta_*|_0,\sqrt{\Sigma^{(1)}}]\wedge \kappa^2[6,|\theta_*|_0,\sqrt{\Sigma^{(2)}]}}|\theta_*|_{0} .\label{eq:upperbound_norme_l0}
\end{equation}
Recall that $|\theta_*|_0\leq |\beta^{(1)}|_0+ |\beta^{(2)}|_0$.
The upper-bound
$\Phi_{|\widehat{\theta}_{\lambda}|_0,+}({\bf W})\leq (1+ |\widehat{\theta}_{\lambda}|_0/k_*)\Phi_{k_*,+}({\bf W})$ and Lemma \ref{lemma:controle_W}
enforce
\begin{eqnarray}\nonumber
|\widehat{\theta}_{\lambda}|_0&\leq& L\frac{n_1\vee n_2}{n_1\wedge n_2} \frac{\Phi_{k_*,+}(\sqrt{\Sigma^{(1)}})\vee \Phi_{k_*,+}(\sqrt{\Sigma^{(2)}})}{ \kappa^2[6,|\theta_*|_0,\sqrt{\Sigma^{(1)}}]\wedge \kappa^2[6,|\theta_*|_0,\sqrt{\Sigma^{(2)}}]} |\theta_*|_0 \left[1+\frac{|\widehat{\theta}_{\lambda}|_0}{k_*}\right]\\ &\leq&\pa{k_*+|\widehat{\theta}_{\lambda}|_0}/2\ , \nonumber
\end{eqnarray}
where the last inequality holds if we take $L^*_2$ in (\ref{eq:hypothese_sparsite}) small enough. Hence, $|\widehat{\theta}_{\lambda}|_0\leq k_*$. Coming back to (\ref{eq:upperbound_norme_l0}), we prove (\ref{eq:upper_l0}).
\end{proof}

\begin{proof}[Proof of Lemma \ref{lemma:controle_biais_kullback}]
Given the Lasso estimator $\widehat{\theta}_{\lambda}$ of $\theta_*$ in model (\ref{eq:definition_modele_regression_lasso}), we define $\widehat{\beta}^{(1)}_{\lambda}$ and  $\widehat{\beta}^{(2)}_{\lambda}$ by
\begin{eqnarray*}
 \widehat{\beta}^{(1)}_{\lambda}= \widehat{\theta}^{(1)}_{\lambda}+ \widehat{\theta}^{(2)}_{\lambda}\ ,\quad \quad \widehat{\beta}^{(2)}_{\lambda}= \widehat{\theta}^{(1)}_{\lambda}- \widehat{\theta}^{(2)}_{\lambda}\ .
\end{eqnarray*}
On event $\mathcal{A}\cap\mathcal{B}$, we upper bound the difference between $(\beta^{(1)},\beta^{(2)})$ and $(\widehat{\beta}^{(1)}_{\lambda},\widehat{\beta}^{(2)}_{\lambda})$.
\begin{eqnarray*}
 \lefteqn{\|\beta^{(1)}-\widehat{\beta}^{(1)}_{\lambda}\|^2_{\Sigma^{(1)}}+ \|\beta^{(2)}-\widehat{\beta}^{(2)}_{\lambda}\|^2_{\Sigma^{(2)}}}&&\\&\leq& 2\left[\|\frac{\X^{(1)}}{\sqrt{n_1}}(\beta^{(1)}-\widehat{\beta}_{\lambda}^{(1)})\|^2+ \|\frac{\X^{(2)}}{\sqrt{n_2}}(\beta^{(2)}-\widehat{\beta}_{\lambda}^{(2)})\|^2\right]\\
&\leq & \frac{2}{n_1\wedge n_2}\|{\bf W}(\theta_*-\widehat{\theta}_{\lambda})\|^2\\
& \leq & L \frac{\bigvee_{i=1,2}\Phi_{1,+}(\sqrt{\Sigma^{(i)}})}{\bigwedge_{i=1,2}\kappa^2[6,|\theta_*|_0,\sqrt{\Sigma^{(i)}}] } \frac{|S_{\vee}|(n_1\vee n_2)}{(n_1\wedge n_2)^2}\log(p)(\sigma^{(1)}\vee \sigma^{(2)})^2\ ,
\end{eqnarray*}
where the last inequality follows from Lemma \ref{lem:cardinal-lasso}. Let us now lower bound the Kullback discrepancy $2[\mathcal{K}_1(\widehat{S}_{\lambda})+\mathcal{K}_2(\widehat{S}_{\lambda})]$ which equals
\begin{eqnarray*}
  \left(\frac{\sigma^{(1)}_{\widehat{S}_{\lambda}}}{\sigma^{(2)}_{\widehat{S}_{\lambda}}}\right)^2+ \left(\frac{\sigma^{(1)}_{\widehat{S}_{\lambda}}}{\sigma^{(2)}_{\widehat{S}_{\lambda}}}\right)^2 -2 + \frac{\|\beta^{(2)}_{\widehat{S}_{\lambda}}-\beta^{(1)}_{\widehat{S}_{\lambda}}\|^2_{\Sigma^{(2)}}}{(\sigma^{(1)}_{\widehat{S}_{\lambda}})^2}+ 
\frac{\|\beta^{(2)}_{\widehat{S}_{\lambda}}-\beta^{(1)}_{\widehat{S}_{\lambda}}\|^2_{\Sigma^{(1)}}}{(\sigma^{(2)}_{\widehat{S}_{\lambda}})^2}\ .
\end{eqnarray*}

\medskip

\noindent 
{\bf CASE 1}: $\frac{\sigma^{(1)}\vee \sigma^{(2)}}{\sigma^{(1)}\wedge \sigma^{(2)}}\geq \sqrt{2}$. By symmetry, we can assume that $\sigma^{(1)}>\sigma^{(2)}$.
\begin{eqnarray}
(\sigma^{(1)}_{\widehat{S}_{\lambda}})^2&= &(\sigma^{(1)})^2+  \|\beta^{(1)}-\beta^{(1)}_{\widehat{S}_{\lambda}}\|_{\Sigma^{(1)}}^2 \geq (\sigma^{(1)})^2\nonumber \\
(\sigma^{(2)}_{\widehat{S}_{\lambda}})^2&= & (\sigma^{(2)})^2 + \|\beta^{(2)}-\beta^{(2)}_{\widehat{S}_{\lambda}}\|_{\Sigma^{(2)}}^2\leq  (\sigma^{(2)})^2 + \|\beta^{(2)}-\widehat{\beta}^{(2)}_{\lambda}\|_{\Sigma^{(2)}}^2\nonumber \\
&\leq&  (\sigma^{(2)})^2 + L \frac{\bigvee_{i=1,2}\Phi_{1,+}(\sqrt{\Sigma^{(i)}})}{\bigwedge_{i=1,2}\kappa^2[6,|\theta_*|_0,\sqrt{\Sigma^{(i)}}] } \frac{|S_{\vee}|(n_1\vee n_2)}{(n_1\wedge n_2)^2}\log(p)(\sigma^{(1)})^2\label{eq:L}\\
&\leq & (\sigma^{(2)})^2+\frac{(\sigma^{(1)})^2}{4}\ ,\nonumber 
\end{eqnarray}
where we used conditions (\ref{eq:definition_k_*}) and (\ref{eq:hypothese_sparsite}) in the last inequality assuming that we have taken $L^*$ and $L^*_2$ small enough in these two conditions.
This enforces 
\begin{eqnarray*}
  2\left[\mathcal{K}_1\left(\widehat{S}_{\lambda}\right)+\mathcal{K}_2\left(\widehat{S}_{\lambda}\right)\right]\geq \frac{1}{12}\ .
\end{eqnarray*}
\medskip

\noindent 
{\bf CASE 2}: $\frac{\sigma^{(1)}\vee \sigma^{(2)}}{\sigma^{(1)}\wedge \sigma^{(2)}}\leq \sqrt{2}$. Let us note 
\begin{eqnarray*}
 A = 2L \frac{\bigvee_{i=1,2}\Phi_{1,+}(\sqrt{\Sigma^{(i)}})}{\bigwedge_{i=1,2}\kappa^2[6,|\theta_*|_0,\sqrt{\Sigma^{(i)}}] } \frac{|S_{\vee}|(n_1\vee n_2)}{(n_1\wedge n_2)^2}\log(p)\ ,
\end{eqnarray*}
with $L$ as in \eqref{eq:L}.
Arguing as in Case 1, we derive that  
\begin{eqnarray*}
 (\sigma^{(1)}_{\widehat{S}_{\lambda}})^2\leq (\sigma^{(1)})^2\left[1+A\right]\leq 2(\sigma^{(1)})^2 \ , \\
 (\sigma^{(1)}_{\widehat{S}_{\lambda}})^2\leq (\sigma^{(2)})^2\left[1+A\right]\leq 2(\sigma^{(2)})^2\ .
\end{eqnarray*}
Let us lower bound $\mathcal{K}_1(\widehat{S}_{\lambda}) + \mathcal{K}_2(\widehat{S}_{\lambda})$ in terms of $\mathcal{K}_1+\mathcal{K}_2$.
First, we consider the ratio of the variances
\begin{eqnarray}
 \frac{(\sigma^{(1)}_{\widehat{S}_{\lambda}})^2}{(\sigma^{(2)}_{\widehat{S}_{\lambda}})^2}+ \frac{(\sigma^{(2)}_{\widehat{S}_{\lambda}})^2}{(\sigma^{(1)}_{\widehat{S}_{\lambda}})^2}- 2 &\geq& 
\left[\frac{(\sigma^{(1)})^2}{(\sigma^{(2)})^2}+ \frac{(\sigma^{(2)})^2}{(\sigma^{(1)})^2}\right]/(1+ A) -2\nonumber \\ &\geq& \frac{(\sigma^{(1)})^2}{(\sigma^{(2)})^2} + \frac{(\sigma^{(2)})^2}{(\sigma^{(1)})^2} - 2  - \frac{A}{1+A}\left[\frac{(\sigma^{(1)})^2}{(\sigma^{(2)})^2} + \frac{(\sigma^{(2)})^2}{(\sigma^{(1)})^2}\right]\nonumber\\
&\geq & \frac{(\sigma^{(1)})^2}{(\sigma^{(2)})^2} + \frac{(\sigma^{(2)})^2}{(\sigma^{(1)})^2} - 2 - 3A\ .\label{eq:lower_kullback}
\end{eqnarray}
Let us now lower bound the remaining part of $\mathcal{K}_1(\widehat{S}_{\lambda})+\mathcal{K}_2(\widehat{S}_{\lambda})$.
For $i=1,2$, $|\beta^{(i)}-\widehat{\beta}^{(i)}_{\lambda}|_0\leq |\theta_*|_0+ |\widehat{\theta}_{\lambda}|_0\leq k_*$  by Lemma \ref{lem:cardinal-lasso} and Condition \eqref{eq:hypothese_sparsite}.
\begin{eqnarray*}
 \lefteqn{\frac{\|\beta^{(1)}-\beta^{(2)}\|^2_{\Sigma^{(2)}}}{(\sigma^{(1)})^2}+ \frac{\|\beta^{(1)}-\beta^{(2)}\|^2_{\Sigma^{(1)}}}{(\sigma^{(2)})^2}}&&\\ &\leq& \frac{3}{(\sigma^{(1)})^2\wedge (\sigma^{(2)})^2}\sum_{i=1}^2\left[\|\beta^{(1)}-\beta^{(1)}_{\widehat{S}_{\lambda}}\|^2_{\Sigma^{(i)}}+ \|\beta^{(2)}-\beta^{(2)}_{\widehat{S}_{\lambda}}\|^2_{\Sigma^{(i)}}+ \|\beta^{(1)}_{\widehat{S}_{\lambda}}-\beta^{(2)}_{\widehat{S}_{\lambda}}\|^2_{\Sigma^{(i)}}\right]\\
&\leq &  L_1\left[\frac{\|\beta^{(1)}_{\widehat{S}_{\lambda}}-\beta^{(2)}_{\widehat{S}_{\lambda}}\|^2_{\Sigma^{(1)}}}{(\sigma^{(2)})^2}+\frac{\|\beta^{(1)}_{\widehat{S}_{\lambda}}-\beta^{(2)}_{\widehat{S}_{\lambda}}\|^2_{\Sigma^{(2)}}}{(\sigma^{(1)})^2}\right]\\ && + \frac{L_2}{(\sigma^{(1)}\wedge \sigma^{(2)})^2} \frac{\bigvee_{i=1,2}\Phi_{k_*,+}(\sqrt{\Sigma^{(i)}})}{\bigwedge_{i=1,2}\Phi_{k_*,-}(\sqrt{\Sigma^{(i)}})}\left[\sum_{i=1}^2\|\beta^{(i)}-\widehat{\beta}^{(i)}_{\lambda}\|^2_{\Sigma^{(i)}}\right]  \\
&\leq & L_1\left[\frac{\|\beta^{(1)}_{\widehat{S}_{\lambda}}-\beta^{(2)}_{\widehat{S}_{\lambda}}\|^2_{\Sigma^{(1)}}}{(\sigma^{(2)})^2}+\frac{\|\beta^{(1)}_{\widehat{S}_{\lambda}}-\beta^{(2)}_{\widehat{S}_{\lambda}}\|^2_{\Sigma^{(2)}}}{(\sigma^{(1)})^2}\right]+ L_2\frac{\bigvee_{i=1,2}\Phi_{k_*,+}(\sqrt{\Sigma^{(i)}})}{\bigwedge_{i=1,2}\Phi_{k_*,-}(\sqrt{\Sigma^{(i)}}) }A
\end{eqnarray*}
Gathering the last inequality with (\ref{eq:lower_kullback}) yields 
\begin{eqnarray*}
 \mathcal{K}_1(\widehat{S}_{\lambda}) + \mathcal{K}_2(\widehat{S}_{\lambda})\geq L_1\left[\mathcal{K}_1+\mathcal{K}_2\right]- L_2\frac{\bigvee_{i=1,2}\Phi_{k_*,+}(\sqrt{\Sigma^{(i)}})}{\bigwedge_{i=1,2}\Phi_{k_*,-}(\sqrt{\Sigma^{(i)}}) }A\ .
\end{eqnarray*}

\end{proof}

\begin{proof}[Proof of Lemma \ref{lemma:controle_puissance_aleatoire}]
For any non empty set $S$ of size smaller or equal to $k_*$, define $\delta_S=\delta\left(2\binom{|S|}{p}k_*\right)^{-1}$. 
If we take $L^*$ and $L^*_1$ in  (\ref{eq:definition_k_*}-\ref{eq:lasso_hypothese_alpha_delta}) small enough, then $1+\log[1/(\alpha_S\delta_S)]/(n_1\wedge n_2)$ is smaller than 
some constant $L$ small enough so that we can apply Theorem \ref{thrm_puissance}. Arguing as in the proof of this Theorem, we derive that 
 \[\mathbb{P}\left[\min_{i\in\{V,1,2\}}\widetilde{Q}_{i,S}( F_{S,i}|{\bf X}_{S})<\alpha_{S}\right]\geq 1-\delta_S\] if 
\begin{equation*}
\mathcal{K}_1(S)+ \mathcal{K}_2(S)\geq L\varphi_{S} \left(\frac{1}{n_1}+ \frac{1}{n_2}\right)  \left[|S|\log(p)+ \log\left(\frac{1}{\alpha\delta}\right) + \log(p) \right]\ .
\end{equation*}
Applying a union bound over all sets $S$ of size smaller or equal to $k_*$ allows us to prove
 \[\mathbb{P}\left[\min_{i\in\{V,1,2\}}\widetilde{Q}_{i,\widehat{S}_{\lambda}}( F_{\widehat{S}_{\lambda},i}|{\bf X}_{\widehat{S}_{\lambda}})<\alpha_{\widehat{S}_{\lambda}}\right]\geq 1-\delta\ .\]
\end{proof}

\subsection{Proof of Proposition \ref{prte:lasso_precis2}}

This proof follows the same steps as above. Taking $\tilde{L}^{*}$ small enough, we can assume that $n_1\vee n_2\leq 2(n_1\wedge n_2)$.
Rewrite the linear regression model ${\bf Y}= {\bf W}\theta_*+\eps$ as follows:
\begin{equation*}
{\bf Y}= {\bf W}^{(1)}\theta_*^{(1)}+ {\bf W}^{(2)}\theta_*^{(2)}+ \eps\ .
\end{equation*}
From the definition of the Lasso estimator $\widehat{\theta}_{\lambda}=\left(\begin{array}{cc}
                                                                                   \widehat{\theta}^{(1)}_{\lambda}\\\widehat{\theta}^{(2)}_{\lambda}
                                                                                  \end{array}\right)$, we derive that
$\widehat{\theta}^{(2)}_{\lambda}$ is the solution of the following minimization problem:
\begin{equation}\label{eq:definition_lasso_alternative}
  \arg\min_{\theta\in\mathbb{R}^p}\|\eps + {\bf W}^{(2)}\theta_*^{(2)} + {\bf W}^{(1)}\left(\theta_*^{(1)}-\widehat{\theta}^{(1)}_{\lambda} \right) - {\bf W}^{(2)}\theta\|+\lambda |\theta'|_1 \ .
\end{equation}
We fix 
\[\lambda= 16(\sigma^{(1)}\vee \sigma^{(2)})\sqrt{2(n_1+n_2)\Phi_{1,+}(\sqrt{\Sigma})\log(p)}\ .\] 
and we suppose that event $\mathcal{A}\cap \mathcal{B}$ (defined in the last proof) holds. Recall that $\mathbb{P}[\mathcal{A}\cap\mathcal{B}]\geq 1-\delta/4-1/p$.
We consider the set $\widehat{S}^{(2)}_{\lambda}$ defined as the support of $\widehat{\theta}^{(2)}_{\lambda}$. Note that $\widehat{S}^{(2)}_{\lambda}\in \widehat{\cS}_L^{(2)}\subset \SLasso$.

\begin{lemma}\label{lemma_correlation_petite} If we take constants $\tilde{L}^*$ and $L^*_2$ in Proposition \ref{prte:lasso_precis2} small enough, then the following holds.  There exists an $\mathcal{C}$ of probability larger than $1-1/p$ such that, under  $\mathcal{A}\cap\mathcal{B}\cap \mathcal{C}$, we have
\begin{equation}
|{\bf W}^{(2)\intercal}{\bf W}^{(1)}\left(\theta_*^{(1)}-\widehat{\theta}^{(1)}_{\lambda}\right)|_{\infty}\leq \lambda/8 
\end{equation}
\end{lemma}
It follows that on $\mathcal{A}\cap\mathcal{B}\cap \mathcal{C}$:
\[\left|{\bf W}^{(2)\intercal}\left[\eps + {\bf W}^{(1)}\left(\theta_*^{(1)}-\widehat{\theta}^{(1)}_{\lambda}\right)\right]\right|_{\infty}\leq \lambda /4\]
 Arguing as in the proof of Lemma \ref{lem:cardinal-lasso} and taking $L_2^*$ small enough, we derive that under $\mathcal{A}\cap \mathcal{B}$, 
\begin{eqnarray}\label{eq:upper_prediction_theta2}
\|{\bf W}^{(2)}(\theta_*^{(2)}-\widehat{\theta}^{(2)}_\lambda)\|^2&\leq& L_1\frac{\lambda^2/(n_1\wedge n_2)}{\kappa^2[6,\tilde{k}_*,\sqrt{\Sigma}] }|\theta_*^{(2)}|_{0}\ ,\\
|\widehat{\theta}^{(2)}_{\lambda}|_0 &\leq& L_2 \frac{\Phi_{k_*,+}(\sqrt{\Sigma})}{  \kappa^2[6,\tilde{k}_*,\sqrt{\Sigma}]} |\theta_*^{(2)}|_0\ \leq \tilde{k}_*/2\leq k_*/2\ . \label{eq:upper_l02}
\end{eqnarray}
This allows us to upper bound  $\|\theta_*^{(2)}-\widehat{\theta}_{\lambda}^{(2)}\|^2_{\Sigma}$ under event $\mathcal{A}$.
\begin{eqnarray*}
 \|\theta_*^{(2)}-\widehat{\theta}_{\lambda}^{(2)}\|^2_{\Sigma} &\leq& \frac{L}{n_1\wedge n_2} \left[\|\X^{(1)}(\theta_*^{(2)}-\widehat{\theta}_{\lambda}^{(2)})\|^2 + \|\X^{(2)}(\theta_*^{(2)}-\widehat{\theta}_{\lambda}^{(2)})\|^2 \right]\\
&\leq & \frac{L}{n_1\wedge n_2} \|{\bf W}^{(2)}(\theta_*^{(2)}-\widehat{\theta}_{\lambda}^{(2)})\|^2\ .
\end{eqnarray*}
Pythagorean inequality then gives 
\begin{eqnarray*}
 \|\beta^{(1)}-\beta^{(2)}\|^2_{\Sigma}&=& \|\beta_{\widehat{S}_{\lambda}^{(2)}}^{(1)}-\beta_{\widehat{S}_{\lambda}^{(2)}}^{(2)}\|^2_{\Sigma}+ \|\beta^{(1)}-\beta^{(2)}- \beta_{\widehat{S}_{\lambda}^{(2)}}^{(1)}+\beta_{\widehat{S}_{\lambda}^{(2)}}^{(2)}\|^2_{\Sigma}\\
&\leq& \|\beta_{\widehat{S}_{\lambda}^{(2)}}^{(1)}-\beta_{\widehat{S}_{\lambda}^{(2)}}^{(2)}\|^2_{\Sigma}+ \|\theta_*^{(2)}-\widehat{\theta}_{\lambda}^{(2)}\|^2_{\Sigma}\\
&\leq & \|\beta_{\widehat{S}_{\lambda}^{(2)}}^{(1)}-\beta_{\widehat{S}_{\lambda}^{(2)}}^{(2)}\|^2_{\Sigma}+ L \frac{|\theta_*^{(2)}|_0\log(p)}{n_1\wedge n_2}\frac{\Phi_{1,+}(\sqrt{\Sigma})}{\kappa^2[6,\tilde{k}_*,\sqrt{\Sigma}]}(\sigma^{(1)}\vee \sigma^{(2)})^2\ ,
\end{eqnarray*}
where we use the two previous upper bounds in the last line. Consequently, we obtain
\begin{eqnarray*}
 \mathcal{K}_1(\widehat{S}^{(2)}_{\lambda}) + \mathcal{K}_2(\widehat{S}^{(2)}_{\lambda})&\geq& L\frac{\|\beta^{(1)}-\beta^{(2)}\|^2_{\Sigma}}{\var(Y^{(1)})\vee \var(Y^{(2)})} - 
L' \frac{|\theta_*^{(2)}|_0\log(p)}{n_1\wedge n_2}\frac{\Phi_{1,+}(\sqrt{\Sigma})}{\kappa^2[6,\tilde{k}_*,\sqrt{\Sigma}]}\ .
\end{eqnarray*}
Applying Lemma \ref{lemma:controle_puissance_aleatoire} to $\widehat{S}^{(2)}_{\lambda}$, using \eqref{eq:upper_l02} and taking $L^*_3$ large enough then allows us to conclude.

\begin{proof}[Proof of Lemma \ref{lemma_correlation_petite}]

Given any matrix $A$, we define the norm $\|A\|_{\infty}=\max_{i,j} |A_{i,j}|$. Suppose that we are under events $\mathcal{A}\cap \mathcal{B}$ defined previously.
Arguing as in the proof of Lemma \ref{lem:cardinal-lasso}, we derive that   $|\theta_*|_0+|\widehat{\theta}_{\lambda}|_0 \leq \tilde{k}_*$  and
\begin{equation}\label{eq:upper_prediction2}
\|{\bf W}(\widehat{\theta}_\lambda-\theta_*)\|^2\leq L_1\frac{\lambda^2}{\kappa^2[6,|\theta_*|_0,\sqrt{\Sigma}] (n_1\wedge n_2)}\tilde{k}_*\ .
 \end{equation}
Thus, $|\theta_*^{(1)}-\widehat{\theta}^{(1)}_{\lambda}|_0\leq \tilde{k}_*$ and we derive
\begin{eqnarray}\nonumber 
\left|{\bf W}^{(2)\intercal}{\bf W}^{(1)}\left(\theta_*^{(1)}-\widehat{\theta}^{(1)}_{\lambda}\right)\right|_{\infty}&= &\left|\left({\bf X}^{(1)\intercal}{\bf X}^{(1)}- {\bf X}^{(2)\intercal}{\bf X}^{(2)}\right)\left(\theta_*^{(1)}-\widehat{\theta}^{(1)}_{\lambda}\right)\right|_{\infty} \\ \nonumber 
&\leq & \|\theta_*^{(1)}-\widehat{\theta}^{(1)}_{\lambda}\|\sqrt{\tilde{k}_*} \|{\bf X}^{(1)\intercal}{\bf X}^{(1)}- {\bf X}^{(2)\intercal}{\bf X}^{(2)}\|_{\infty}\\ \nonumber 
&\leq & \frac{\|{\bf W}(\theta_*-\widehat{\theta})\|}{\sqrt{\Phi_{k_*,-}({\bf W})}}\sqrt{\tilde{k}_*} \|{\bf X}^{(1)\intercal}{\bf X}^{(1)}- {\bf X}^{(2)\intercal}{\bf X}^{(2)}\|_{\infty}\\
&\leq &L \frac{\lambda\tilde{k}_* \|{\bf X}^{(1)\intercal}{\bf X}^{(1)}- {\bf X}^{(2)\intercal}{\bf X}^{(2)}\|_{\infty}}{\sqrt{n_1\wedge n_2}\kappa[6,|\theta_*|_0,\sqrt{\Sigma}]\sqrt{\Phi_{k_*,-}({\bf W})} } \label{eq:upper_prediction3}
\ ,
\end{eqnarray}
where we used (\ref{eq:upper_prediction2}) in the last line.

Combining deviations inequality for  $\chi^2$ distributions (Lemma \ref{lemma:concentration_chi2}) and for Gaussian distributions and a union bound, we derive that 
\begin{eqnarray}\label{eq:majoration_difference_covariance}
\|{\bf X}^{(1)\intercal}{\bf X}^{(1)}- {\bf X}^{(2)\intercal}{\bf X}^{(2)}\|_{\infty}\leq \Phi_{1,+}(\sqrt{\Sigma})\left[|n_1-n_2|+ L\sqrt{(n_1\vee n_2)\log(p)}\right]\ ,
\end{eqnarray}
with probability larger than $1-1/p$.
Consider some $\theta$ with $|\theta|_0\leq k_*$. When event $\mathcal{A}$ defined in Lemma \ref{lemma:controle_W} holds, we have 
\begin{eqnarray*}
\frac{\|{\bf W}\theta\|^2}{\|\theta\|^2}&=& \frac{\|\X^{(1)}(\theta^{(1)}+\theta^{(2)})\|^2}{\|\theta\|^2}+ \frac{\|\X^{(2)}(\theta^{(1)}-\theta^{(2)})\|^2}{\|\theta\|^2}\\ &\geq& \frac{\Phi_{k_*,-}(\sqrt{\Sigma})}{2}\frac{n_1\|\theta^{(1)}+\theta^{(2)}\|^2+ n_2 \|\theta^{(1)}-\theta^{(2)}\|^2 }{\|\theta\|^2}\\
& \geq  & \Phi_{k_*,-}(\sqrt{\Sigma})(n_1\wedge n_2)\ . 
\end{eqnarray*}
Let us note $T_{\Sigma}= \frac{\Phi_{1,+}(\sqrt{\Sigma})}{\kappa[6,k_*,\sqrt{\Sigma}] \Phi^{1/2}_{k_*,-}(\sqrt{\Sigma})} $.
Gathering the last upper  bound with \eqref{eq:upper_prediction3} and \eqref{eq:majoration_difference_covariance}, we get
\begin{eqnarray*}
\left|{\bf W}^{(2)\intercal}{\bf W}^{(1)}\left(\theta_*^{(1)}-\widehat{\theta}^{(1)}_{\lambda}\right)\right|_{\infty} \leq L\lambda \tilde{k}_*\left[\frac{|n_1-n_2|}{n_1\wedge n_2} + \sqrt{\frac{\log(p)}{n_1\wedge n_2}}\right]T_{\Sigma}\ ,
\end{eqnarray*}
since  $n_1\vee n_2\leq 2(n_1\wedge n_2)$.
Taking $\tilde{L}^{*}$ small enough in definition (\ref{eq:definition_k'*}) of $\tilde{k}_*$ allows us to conclude.
\end{proof}

\subsection{Proof of Proposition \ref{prte:minoration_minimax}}

By symmetry, we can assume that $n_1\leq n_2$. Let us fix $\beta^{(2)}=0$ and $\sigma^{(2)}=1$. Fix some positive integer $s\leq p^{1/2-\gamma}$ and fix $r\in (0,1/\sqrt{2})$.

We consider the test of hypotheses $\hyp_0:$  $\beta^{(1)}=0$, $\sigma^{(1)}=1$ against $\hyp_1:$ $|\beta^{(1)}|_0=s$, $\|\beta^{(1)}\|=r^2$, and $\sigma^{(1)}=\sqrt{1-r^2}$. Note that for this problem, the data $({\bf Y}^{(2)}, \X^{(2)})$ do not bring any information on the hypotheses.
This one-sample testing problem is a specific case of the two-sample testing problem considered in the proposition. Thus, a minimax lower bound for the one-sample problem provides us a minimax lower bound for the two-sample problem. 
 
According to Theorem 4.3 in \cite{2010_AS_Verzelen}, no level $\alpha$ test has power larger than $1-\delta$ if 
\begin{eqnarray*}
 \frac{r^2}{1-r^2}\leq  \frac{s}{2n_1}\log\left(1+\frac{p}{s^2}+\sqrt{\frac{2p}{s^2}}\right)
\end{eqnarray*}
Since $s\leq p^{1/2-\gamma}$, no level $\alpha$ test has power larger than $1-\delta$ if
\begin{eqnarray}
 \frac{r^2}{1-r^2}\leq  \gamma \frac{|s|}{n_1}\log(p)\ .
\end{eqnarray}
By Assumption $({\bf A.2})$, one may assume that that the right-hand side term is smaller than $1/2$.
Observe that 
\[2(\mathcal{K}_1+\mathcal{K}_2)= \frac{2r^2}{1-r^2}\quad \text{ and }  \quad \frac{\|\beta^{(1)}-\beta^{(2)}\|^2_{I_p}}{\var[Y^{(1)}]\wedge \var[Y^{(2)}]}= r^2\geq \frac{1}{2}\frac{r^2}{1-r^2}\ ,\]
for $r\leq \sqrt{2}$. The result follows.

\subsection{Technical lemmas}

In this section, some useful  deviation inequalities for $\chi^2$ random variables \cite{Laurent00} and for Wishart matrices~\cite{Davidson2001} are reminded.

\begin{lemma}\label{lemma:concentration_chi2}
 For any integer $d>0$ and any positive number $x$, 
 \begin{eqnarray}
\mathbb{P}\left(\chi^2(d) \leq d - 2\sqrt{dx} \right)& \leq& \exp(-x)
\nonumber\ ,\\
\mathbb{P}\left(\chi^2(d) \geq d + 2\sqrt{dx} + 2x \right) &\leq &\exp(-x)\ .
\nonumber
\end{eqnarray}
\end{lemma}

\begin{lemma}\label{lemma:concentration_vp_wishart}
Let $Z^{\intercal}Z$ be a standard Wishart matrix of parameters $(n,d)$ with $n>d$. For
any positive number $x$, 
$$\mathbb{P}\left\{\varphi_{\min}\left(Z^{\intercal}Z\right) \geq
n\left(\left\{1-\sqrt{\frac{d}{n}}-x\right\}\vee 0\right)
\right\} \leq \exp(-nx^2/2)\ ,
$$
and 
$$\mathbb{P}\left[\varphi_{\max}\left( Z^{\intercal}Z\right) \leq
n\left(1+\sqrt{\frac{d}{n}}+x\right)^2 \right]\leq \exp(-nx^2/2)\ .	$$
\end{lemma}
\subsection*{Acknowledgements}
We are grateful to Christophe Giraud, two anonynous reviewers, the Associate Editor and the Editor for careful suggestions. We also thank Nicolas Städler for providing us the code of DiffRegr. The research of N. Verzelen is partly supported by the french Agence Nationale
de la Recherche (ANR 2011 BS01 010 01 projet Calibration).

\bibliographystyle{acmtrans-ims}
\bibliography{biblio}
\end{document}